\documentclass[12pt]{amsart}
 \usepackage{amsfonts,amssymb,amsthm}
 \usepackage{enumitem}
 \setlist[itemize]{noitemsep,nolistsep}
  \setlist[enumerate]{noitemsep,nolistsep}
\usepackage[toc,page]{appendix}
\usepackage{graphicx,float} 
 \usepackage{color,array}
\usepackage[all,ps,cmtip,rotate]{xy}

\usepackage{mathrsfs}

\usepackage{bm}

\usepackage{hyperref}

   \allowdisplaybreaks

\def\Z{{\bf Z}}

\def\C{{\bf C}}

\def\P{{\bf P}}
\def\PP{{\bf P}}

\def\T{{\bf T}}

 \def\moins{\smallsetminus}

\def\wtilde{\widetilde}

\def\cI{\mathscr{I}}

\def\cA{\mathscr{A}}
\def\cB{\mathscr{B}}
\def\cF{\mathscr{F}}
\def\cL{\mathscr{L}}
\def\cO{\mathscr{O}}

\def\cE{\mathscr{E}}
\def\cEN{\mathscr{E\!N}}

\def\cM{\mathscr{M}}
\def\cN{\mathscr{N}}

\def\cU{\mathscr{U}}
\def\cV{\mathscr{V}}
\def\cW{\mathscr{W}}
\def\cX{\mathscr{X}}

\def\cY{\mathscr{Y}}

\def\bmu{{\boldsymbol\mu}}

\def\k{\mathbf k}

\def\bmu{\bar\mu}

\def\bq{\mathbf q}

\def\lra{\longrightarrow}
\def\llra{\hbox to 10mm{\rightarrowfill}}
\def\lllra{\hbox to 15mm{\rightarrowfill}}

\def\llla{\hbox to 10mm{\leftarrowfill}}
\def\lllla{\hbox to 15mm{\leftarrowfill}}

\def\dra{\dashrightarrow}
\def\thra{\twoheadrightarrow}
\def\hra{\hookrightarrow}

\def\lhra{\ensuremath{\lhook\joinrel\relbar\joinrel\rightarrow}}

\newcommand{\lthra}{}
\DeclareRobustCommand{\lthra}{\relbar\joinrel\twoheadrightarrow}

\def\isom{\simeq}
\def\eps{\varepsilon}
\def\ie{\hbox{i.e.}}
\def\eg{\hbox{e.g.}}

\def\vide{\varnothing}
\def\emptyset{\varnothing}

\DeclareMathOperator{\isomdra}{\stackrel{{}_{\scriptstyle\sim}}{\dra}}
\DeclareMathOperator{\isomlra}{\stackrel{{}_{\scriptstyle\sim}}{\lra}}

\DeclareMathOperator{\isomto}{\isomlra}

\DeclareMathOperator{\Aut}{Aut}

\DeclareMathOperator{\Bs}{Base}
\DeclareMathOperator{\Cl}{Cl}
\DeclareMathOperator{\codim}{codim}

\DeclareMathOperator{\Coker}{Coker}

\DeclareMathOperator{\End}{End}

\DeclareMathOperator{\Tor}{Tor}
\DeclareMathOperator{\GL}{GL}
\DeclareMathOperator{\Gr}{\mathsf{Gr}}
\DeclareMathOperator{\LGr}{\mathsf{LGr}}
\DeclareMathOperator{\OGr}{\mathsf{OGr}}
\DeclareMathOperator{\CGr}{\mathsf{CGr}}
\DeclareMathOperator{\CKGr}{\mathsf{C}_K\mathsf{Gr}}

\DeclareMathOperator{\Fl}{\mathsf{Fl}}

\DeclareMathOperator{\Hom}{Hom}
\DeclareMathOperator{\Id}{Id}
\def\Im{\mathop{\rm Im}\nolimits}

\DeclareMathOperator{\Ker}{Ker}

\DeclareMathOperator{\lin}{\underset{\mathrm lin}{\sim}}

\DeclareMathOperator{\PGL}{PGL}
\DeclareMathOperator{\Pic}{Pic}
\DeclareMathOperator{\pr}{\mathsf{pr}}

\DeclareMathOperator{\fsl}{\mathfrak sl}

\DeclareMathOperator{\SSpec}{\mathbf{Spec}}
\DeclareMathOperator{\rank}{rank}
\DeclareMathOperator{\Sing}{Sing}
\DeclareMathOperator{\Span}{Span}
\DeclareMathOperator{\Sym}{\mathsf S}

\def\spe{\text{spe}}

\def\bw#1#2{\textstyle{\bigwedge\hskip-0.9mm^{#1}}\hskip0.2mm{#2}}
\def\sbw#1#2{{\bigwedge\hskip-0.9mm^{#1}}\hskip0.1mm{#2}}

\def\symv{\VV}

\def\Gm{\mathbb{G}_{\mathrm{m}}}

\newtheorem{lemm}{Lemma}[section]
\newtheorem{theo}[lemm]{Theorem}
\newtheorem{coro}[lemm]{Corollary}
\newtheorem{prop}[lemm]{Proposition}

\theoremstyle{remark}
\newtheorem{defi}[lemm]{Definition}
\newtheorem{rema}[lemm]{Remark}

\def\VV{\mathbb{V}}
\def\LL{\mathbb{L}}
\def\II{\mathbb{I}}

\def\BVV{\overline{\mathbb{V}}}
\def\BLL{\overline{\mathbb{L}}}

\def\BA{{\overline{A}}}
\def\BW{{\overline{W}}}

\def\id{\mathsf{id}}
\def\opp{\mathrm{opp}}
\def\ord{\mathrm{ord}}
\def\spe{\mathrm{spe}}
 
\newcommand{\cone}[1]{\mathsf{C}_{#1}}
\newcommand{\pcone}[1]{\mathsf{C}_{#1}^\circ}

\DeclareMathOperator{\Dis}{
{Disc}}

\newcommand{\Xreg}{{X_{\rm sm}}}


\begin{document}
\title[Gushel--Mukai varieties]
{Gushel--Mukai varieties: classification and birationalities}

\author[O.~Debarre]{Olivier Debarre}
\address{Univ  Paris Diderot, \'Ecole normale su\-p\'e\-rieu\-re, PSL Research University, CNRS, D\'epar\-te\-ment Math\'ematiques et Applications\\45 rue d'Ulm, 75230 Paris cedex 05, France}
\email{{olivier.debarre@ens.fr}}

 \author[A. Kuznetsov]{Alexander Kuznetsov}
 \address{Algebra Section, Steklov Mathematical Institute,
  8 Gubkin str., Moscow 119991 Russia;
 The Poncelet Laboratory, Independent University of Moscow;
 Laboratory of Algebraic Geometry, National Research University Higher School of Economics, 7 Vavilova Str., Moscow, Russia, 117312 }
 \email{{\tt  akuznet@mi.ras.ru}}
 
\thanks{A.K. was partially supported by the Russian Academic Excellence Project ``5-100'', by the RFBR grants 15-01-02164 and 15-51-50045, and by the Simons foundation.}

\def\setminus{\smallsetminus}
\def\cong{\isom}
\newcommand{\An}{A}

\newcommand{\rem}[1]{{\color{blue}{#1}}}
\newcommand{\red}[1]{{\color{red}{#1}}}

\begin{abstract}
We perform a systematic study of Gushel--Mukai varieties---quadratic sections of linear sections of  cones over the Grassmannian $\Gr(2,5)$.
This class of varieties includes Clifford general curves of genus 6, Brill--Noether general polarized K3 surfaces of genus 6, prime Fano threefolds of genus 6, and their higher-dimensional analogues.

We establish an intrinsic characterization of normal Gushel--Mukai varieties in terms of their excess conormal sheaves, which leads to a new proof of the classification theorem of Gushel and Mukai.
We give a description of   isomorphism classes of Gushel--Mukai varieties and their automorphism groups in terms of linear algebraic data naturally associated to these varieties.

We carefully develop  the relation between Gushel--Mukai varieties and Eisenbud--Popescu--Walter sextics introduced earlier
by Iliev--Manivel and O'Grady. We describe explicitly all Gushel--Mukai varieties whose associated EPW sextics are isomorphic or dual
(we call them period partners or dual varieties respectively). Finally, we 
 show that in dimension 3 and higher,
period partners/dual varieties are always birationally isomorphic.
\end{abstract}

\maketitle

\section{Introduction}

This article is devoted to the investigation of the geometry of a class of varieties which we call Gushel--Mukai varieties, or GM varieties for short, which are  
 {\em dimensionally transverse intersections of 
a cone
 over the Grassmannian $\Gr(2,5)$ with   a linear space and a quadratic hypersurface,} all defined over a field $\k$ of characteristic zero.

  {\em Smooth}   GM varieties 
are    isomorphic either to   quadratic sections of a linear section of  $\Gr(2,5) \subset \P^9$ ({\em ordinary} GM varieties)
or to  double covers of a linear section of $\Gr(2,5)$ branched along a quadratic section 
({\em special} GM varieties). Their dimension  is at most 6. 

GM varieties  first appeared in the classification of  {complex} Fano threefolds: Gushel showed in \cite{gushel1983fano}  
that any smooth prime complex Fano threefold of degree 10 is a GM variety.   Mukai later extended in \cite{mukai1989biregular}   
Gushel's results to higher-dimensional smooth Fano varieties of coindex~3, degree~10, and Picard number~1,  to Brill--Noether general
polarized K3 surfaces of degree 10, and to Clifford general curves of genus~6 (see Section~\ref{sec23} for the definitions). 

In Section~\ref{section-classification}, we give an intrinsic characterization of normal GM varieties in terms of their   twisted excess conormal sheaf (Theorem~\ref{theorem:gm-intrinsic}). 
In particular, we show that the (rational) map from a GM variety to $\Gr(2,5)$ is given by the sections of this rank-2 sheaf, hence is canonical.
Developing this observation, we define {\em GM data} (Definition~\ref{defgm}) as sets of linear algebraic objects required to define a GM variety. 
We show that there is a functorial bijection  {(induced by an equivalence of appropriately defined groupoids)} between the set of all normal polarized GM varieties and a suitable set of   GM data (Theorem~\ref{theorem:gm-var-data}). 
We also describe the automorphism group  of a normal polarized GM variety in terms of the associated GM data (Corollary~\ref{corollary-aut-x-aut-data}).

Using this intrinsic characterization of GM varieties, we prove an extension of Mukai's classification results  {(Theorem~\ref{theorem:gushel-mukai})}:
any normal locally factorial complex projective variety~$X$ of dimension $n \ge 3$ 
with  terminal  singularities and $\codim (\Sing(X) )\ge 4$, together with an ample Cartier divisor class $H$ 
such that $K_X \lin -(n-2)H$, $H^n = 10$, and $\Pic(X) = \Z H$, is a 
 GM variety.

{In  contrast with the smooth case},
a singular GM variety may have arbitrarily high dimension; 
 besides the two  {smooth} types discussed above (ordinary and special), there are also quadric bundles of arbitrary dimensions
over a linear section of $\Gr(2,5)$, whose  discriminant locus is a quadratic section. 
However, if one restricts to  locally complete intersection (lci) GM varieties, one gets   
  the  ordinary and special types only. In particular, their dimension is again at most 6 
  and the theory becomes very much parallel to that of smooth GM varieties.

We end Section~\ref{section-classification} by studying     geometric properties of GM varieties.
We define the {\em Grassmannian hull} $M_X$ of a GM variety $X$ as the intersection of the cone over $\Gr(2,5)$ in which $X$ sits with the linear span of $X$, so that $X$ is a quadratic section of~$M_X$. 
We show that  {when $X$ is ordinary, $\dim (X) \ge 3$, and $\codim(\Sing(X)) \ge 4$, the scheme $M_X$ is smooth, and when $X$ is special with $\codim(\Sing(X)) \ge 4$, the scheme $M_X$ is a cone over a smooth variety~$M'_X$} (Proposition~\ref{lhull}). We also discuss  {the singularities of $M_X$ for smooth ordinary GM curves and surfaces}. 
We   study the special features of lci GM varieties, we explain the relations between the two types (ordinary and special), and we construct a (birational) involution on the class of lci GM varieties which exchanges ordinary and special varieties.

In Section~\ref{section-epw-sextics}, we discuss a relation between complex GM varieties and Eisenbud--Popescu--Walter (EPW) sextics (\cite{epw}, \cite{og1}).
An EPW sextic is a hypersurface of degree 6 in the projectivization $\PP(V_6)$ of a 6-dimensional vector space $V_6$ which depends on the choice of a Lagrangian subspace $\An \subset \bw3 V_6$. 
EPW sextics have many wonderful properties which were thoroughly investigated by O'Grady. 
A relation between  GM varieties and EPW sextics was discovered by Iliev--Manivel in~\cite{iliev-manivel}. 
They showed that if $V_6$ is the space of quadrics through an ordinary  5-dimensional GM variety $X$ in {the projective embedding defined by the ample generator $H$ of $\Pic(X)$},
one can construct a Lagrangian subspace $\An(X) \subset \bw3 V_6$ and a hyperplane $V_5(X) \subset V_6$; 
conversely, given a general Lagrangian subspace   $\An \subset \bw3 V_6$ and a general hyperplane $V_5 \subset V_6$, one can reconstruct a 5-dimensional GM variety $X_{\An,V_5}$.  

We suggest an extension of the Iliev--Manivel  construction which works for all lci GM varieties 
(of both types---ordinary and special---and in  all dimensions).
Formally, this is done by associating with an lci GM data another set which we call an  {extended} Lagrangian data.
This is a quadruple $(V_6,V_5,\An,A_1)$, where $\An \subset \bw3V_6$ is a Lagrangian subspace 
and $A_1$, which encodes the type of the GM data, is a Lagrangian subspace in a fixed 2-dimensional symplectic space.
In Theorem~\ref{theo-bijection}, we show that when the base field $\k$ is quadratically closed ($\k = \k^{1/2}$), there is a bijection between the set  of  isomorphism classes of lci GM data 
and the set of isomorphism classes of {extended} Lagrangian data.  

Concatenating this construction with the construction in Section \ref{section-classification}, we obtain  a diagram 

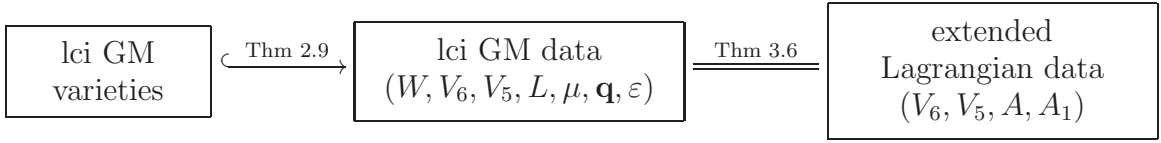
\begin{equation*}
\xymatrix@C+2em{
\xybox{(0,0.5)*+[F]{\parbox{6em}{\begin{center} {lci  GM varieties}\end{center}}}}\xybox{(0,0.5)}
\save +<16mm,0mm>\ar@{^(->}[r]^-{\text{Thm~\ref{theorem:gm-var-data}}} \restore&
\xybox{(0,0.5)*+[F]{\parbox{10em}{\begin{center}lci GM data\\$(W,V_6,V_5,L,\mu,\bq,\eps)$\end{center}}}}
\ar@{=}[r]^-{\text{\ Thm~\ref{theo-bijection}\ }} & 
\xybox{(0,0.5)*+[F]{\parbox{10em}{\begin{center}extended\\Lagrangian data\\$(V_6,V_5,\An,A_1)$ 
\end{center}}}}
}
\end{equation*}
where the objects in the boxes are taken  modulo isomorphisms.
Restricting ourselves to ordinary GM varieties simplifies the picture: 
$A_1$ becomes redundant, so we can restrict to triples $(V_6,V_5,\An)$ which we simply call Lagrangian data,   
the bijection  {becomes functorial and}   works over arbitrary fields, and 
the diagram above simplifies (see Theorem~\ref{theorem:merged}) to an  {\em inclusion} 
 
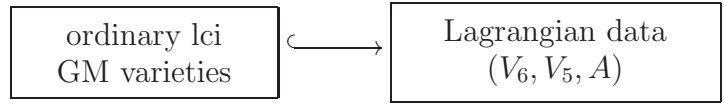
\begin{equation*}
\xymatrix@C+1em{
\xybox{(0,0.5)*+[F]{\parbox{8em}{\begin{center}ordinary lci\\ {GM varieties}\end{center}}}}
\save +<20mm,0mm> \ar@{^(->}[r] \restore&
\xybox{(0,0.5)*+[F]{\parbox{10em}{\begin{center}Lagrangian data\\$(V_6,V_5,\An)$\end{center}}}}
}
\end{equation*}

The correspondence between lci GM varieties and  {extended} Lagrangian data has many remarkable properties.
One of the most impressive is a
criterion for the smoothness of the GM variety $X_{\An,A_1,V_5}$ constructed 
from an {extended} Lagrangian data $(V_6,V_5,\An,A_1)$, which shows in particular 
that if $\dim(X_{\An,A_1,V_5}) \ge 3$, smoothness is equivalent to an explicit property
of the Lagrangian subspace $\An\subset\bw3V_6$ (that it contains no {\em decomposable vectors;} see Section~\ref{section-epw-overview}) and does not depend on the hyperplane $V_5\subset V_6$ (Theorem \ref{theorem-singdec}).

We spend some time discussing  the structure of lci GM varieties associated
with the same Lagrangian subspace $\An \subset \bw3 V_6$, but with different choices of $A_1$ and    hyperplane $V_5 \subset V_6$.
As it was already mentioned, $A_1$ just encodes the type of the GM variety, and switching the type of $A_1$ results in 
switching the type of the   GM variety. 
We show     that for fixed~$ (\An,A_1)$, the dimension of $X_{\An,A_1,V_5}$ only depends on which stratum of the EPW stratification
of the space $\PP(V_6^\vee)$ the hyperplane $V_5$ belongs to (extending again the results of Iliev and Manivel). 

If 
 $\dim (X_{\An,A_1,V_5}) = \dim (X_{\An,A'_1,V_5'})$,
we say that the GM varieties $X_{\An,A_1,V_5}$ and $X_{\An,A'_1,V_5'}$ are {\em period partners.} One of the main results of this article
is that smooth period partners of dimension $ \ge 3$ are birationally isomorphic ({Theorem~\ref{proposition-birationality-3} and} Corollary~\ref{coro417}). 
In particular, no smooth GM varieties   of dimension  $ \ge 3$ are   birationally rigid (Corollary~\ref{coro418}).

We also introduce a notion of duality for GM varieties. Given a Lagrangian subspace $\An \subset \bw3 V_6$,
its orthogonal complement $\An^\perp \subset \bw3 V_6^\vee$ is also Lagrangian;    the choice of a line $V_1 \subset V_6$
(which can be considered as a hyperplane in $V_6^\vee$)  and an arbitrary choice of $A'_1$ 
allow  us to construct a GM variety $X_{\An^\perp,A'_1,V_1}$.
If $\dim( X_{A^\perp,A_1',V_1} )= \dim( X_{\An,A_1,V_5})$, we say that the GM varieties $X_{\An^\perp,A'_1,V_1}$ and $X_{\An,A_1,V_5}$ are dual.
Our second main result is that
smooth dual GM varieties of dimension  $ \ge 3$ are birationally isomorphic (Theorem~\ref{theorem-birationality-duals}).

The birational isomorphisms relating period partners are generalizations of conic transformations    and those 
relating dual varieties are generalizations of line transformations\break (see~\cite{dim}). In relation with birationalities for dual 
varieties, we define   another interesting geometric object associated with GM varieties---a special hypersurface of degree 4 
in $\Gr(3,6)$. As the EPW sextic, this
  hypersurface is defined  as a special Lagrangian
intersection locus and they play  similar roles. We call it the EPW quartic and suspect that it may have
interesting geometry. These quartics are   investigated in more details in~\cite{ikkr}.

For the reader's convenience, we gathered some of the material we need in the main body of this article in three appendices.
In Appendix~\ref{section-excess-conormal-sheaves}, we introduce excess conormal sheaves for intersections of quadrics,
discuss how they change under simple geometric operations (taking cones and hyperplane  and quadratic sections),
and compute them for the Grassmannians $\Gr(2,5)$ and $\Gr(2,6)$.
In Appendix~\ref{section-epw}, we recall the definition of EPW sextics and their main properties. 
Most of the results there are extracted from a series of articles of O'Grady.
In Appendix~\ref{section-lagrangians-and-quadrics}, we discuss the classical relation between the Lagrangian geometry of a symplectic vector space
and the projective duality of quadrics. We also define a notion of   isotropic reduction in this context which is very useful for the description
of the various quadratic fibrations associated with Gushel--Mukai varieties. 

One of the motivations for this research was an attempt to construct a moduli space for GM varieties.
The unified and extended constructions we introduce  here should be important ingredients
to attack the moduli problem. We plan to do this in a forthcoming article.

In another companion article \cite{DK}, we show that period partners of even dimensions have   isomorphic primitive middle Hodge structures and describe this Hodge structure in terms of the associated EPW sextic.

In a joint work \cite{KP} of the second author with Alexander Perry, we discuss derived categorical aspects of GM varieties.

\subsection{Notation and conventions}\label{notconv}
All schemes are separated   of finite type over a field $\k$.
For simplicity, we will always assume that $\k$ is a subfield of $\C$, although many of our results remain valid in positive characteristics.
 A $\k$-variety is a geometrically integral scheme (separated of finite type) over $\k$.
We usually abbreviate ``local(ly) complete intersection'' as ``lci''. 

We write $\lin$ for linear equivalence of divisors. 

Given a vector space $V$, we denote by $V^\vee$ the dual space and, given a vector subspace $U \subset V$, 
we denote by $U^\perp \subset V^\vee$ its orthogonal complement, \ie, the space of linear functions on $V$
vanishing on $U$.

Given a vector space $V$, we denote by $\PP(V)$ its projectivization, \ie, the variety of lines in $V$.
Given a line $V_1 \subset V$, we write $[V_1]$, or even $V_1$, for the corresponding point of $\PP(V)$.
More generally, we denote by $\Gr(k,V)$ the Grassmannian of $k$-dimensional subspaces in $V$ and, given
a subspace $V_k \subset V$, we write $[V_k]$, or even $V_k$, for the corresponding point of the Grassmannian.
 Analogously, we denote by $\Fl(k_1,\dots,k_r;V)$ the flag variety of chains of subspaces in $V$ of dimensions $k_1 < \dots < k_r$. 
Finally, given a vector bundle $\cE$ on a scheme $X$, we denote by $\PP_X(\cE)$ its projectivization, \ie,
the scheme of lines in fibers of $\cE$ over~$X$.
 
Let $W$ be a vector space and let $X \subset \PP(W)$ be a subscheme. 
For a point $x \in X$, we write $\T_{X,x}$ for the embedded tangent space of $X$ at $x$.
For any vector space $K$, we denote by $\cone{K}X \subset \P(W \oplus K)$ the cone over $X$
with vertex $\P(K)$  and by $\pcone{K}X$ the punctured cone, \ie, $\pcone{K}X = \cone{K}X \setminus \P(K)$.

Given   closed subschemes $X_1,X_2 \subset \P(W)$, we say that the intersection $X_1 \cap X_2$ is   {\sf dimensionally transverse}
if $\Tor_{>0}(\cO_{X_1},\cO_{X_2}) = 0$ (this condition is also known as {\sf $\Tor$-independence}). 
When  $X_1$ and $X_2$ are both Cohen--Macaulay, this condition is equivalent to 
\begin{equation*}
\codim_x(X_1 \cap X_2) = \codim_x(X_1) + \codim_x(X_2)
\end{equation*}
for any closed point $x \in X_1 \cap X_2$ (where the codimensions are in $\P(W)$).
When $X_2  \subset \P(W)$ is a hypersurface,
the  intersection $X_1 \cap X_2$ is dimensionally transverse if and only if no associated point of $X_1$ is contained in $X_2$.

\subsection{Acknowledgements.}
 This article owns much to the ideas and work of Kieran O'Grady,   Atanas Iliev, and Laurent Manivel.
We would also like to thank Olivier Benoist,  Alex Perry, and Yuri Prokhorov for their help 
and many interesting  discussions.

\section{{Geometry} of  Gushel--Mukai varieties}\label{section-classification}

\subsection{Intrinsic characterization of GM varieties}\label{ss:intrinsic-characterization}

In the introduction, we   defined a GM variety as a dimensionally transverse intersection of a cone  over $\Gr(2,5)$
with a linear subspace and a quadratic hypersurface. 
 
More precisely, let $V_5$ be a $\k$-vector space  of dimension 5  and consider the Pl\"ucker embedding $\Gr(2,V_5) \subset \P(\bw2V_5)$. Let $K$ be an arbitrary $\k$-vector space. Consider  the cone
\begin{equation*}
\CKGr(2,V_5) \subset \P(\bw2V_5 \oplus K)
\end{equation*}
with vertex $\P(K)$ 
 and choose a vector subspace $W \subset \bw2V_5 \oplus K$
 and a subscheme $Q \subset \P(W)$ defined by one quadratic equation (possibly zero). 

\begin{defi}\label{defigm}
The scheme
\begin{equation}\label{equation-x-from-data}
X = \CKGr(2,V_5) \cap \P(W) \cap Q.
\end{equation} 
is called    a {\sf GM intersection}.
A  GM intersection $X$ is a    {\sf GM variety} if $X$ is geometrically integral and
\begin{equation}\label{eq:dimx-dimw}
\dim(X) = \dim(W) - 5  \ge 1.
\end{equation} 
\end{defi}

When $X$ is a GM variety, $Q$ is  a quadratic hypersurface and $\CKGr(2,V_5) \cap \P(W) $ and $ Q$ are Cohen--Macaulay Gorenstein subschemes of $\P(W)$, hence   their intersection $X$ is  dimensionally transverse. 
In particular, a GM variety is always Cohen--Macaulay Gorenstein.

A GM variety $X$ has a canonical polarization, the restriction $H$  of the hyperplane class on~$\P(W)$; we will call $(X,H)$ a {\sf polarized GM variety.}
 
\begin{defi}\label{def:gm-iso}
An {\sf isomorphism of polarized GM varieties} between   $(X,H)$ and $(X',H')$ is a pair $(\phi,\psi)$, where $\phi\colon X \isomto X'$ is an isomorphism of varieties 
and $\psi\colon \cO_X(H) \isomto \phi^*\cO_{X'}(H')$   an isomorphism of line bundles. 
 {We consider the set of all polarized GM varieties as a groupoid, with this notion of   isomorphism.}
\end{defi}

We denote by $\Aut(X,H)$ the group of automorphisms of a polarized GM variety    and by $\Aut_H(X)$ its image in $\Aut(X)$. 
There is an exact sequence
\begin{equation}\label{eq:aut-x-h}
1 \to \Gm \to \Aut(X,H) \to \Aut_H(X) \to 1, 
\end{equation} 
where the subgroup $\Gm \subset \Aut(X,H)$ acts trivially on $X$ and by dilations on $\cO_X(H)$.

The definition of a GM variety is not intrinsic. 
The following theorem gives an intrinsic characterization, at least for normal varieties  {(note that a GM variety is normal as soon as it is non-singular in codimension 1)}. 
A key ingredient is the excess conormal sheaf, which is defined over $\k$, for any variety
which is an intersection of quadrics (see Appendix~\ref{section-excess-conormal-sheaves}).

Recall that a coherent sheaf $\cF$ is {\sf simple} if $\Hom(\cF,\cF) = \k$.

\begin{theo}\label{theorem:gm-intrinsic}\renewcommand{\theenumi}{\alph{enumi}}\renewcommand{\labelenumi}{\textup{(\theenumi)}}
A normal polarized projective variety $(X,H)$ of dimension $n \ge 1$  is a polarized GM variety if and only if all the following    conditions hold:
\begin{enumerate}
\item\label{item-degree-kx}
$H^n = 10$ and $K_X = -(n-2)H$; in particular, $X$ is Gorenstein;
\item\label{item-w} 
$H$ is very ample and the vector space 
\begin{equation*}
W := H^0(X,\cO_X(H))^\vee
\end{equation*}
has dimension $n + 5$;
\item\label{item-v6}
$X$ is an intersection of quadrics in $\P(W)$ and the vector space 
\begin{equation*}
V_6 := H^0(\P(W),\cI_X(2))\subset \Sym^2\!W^\vee  
\end{equation*}
of quadrics through $X$ has dimension~$6$;
\item\label{item-ux-simple} 
the twisted excess conormal sheaf $\cU_X := \cEN^\vee_X(2H)$ of $X$ in $\P(W)$ is simple.
\end{enumerate}
\end{theo}

\begin{proof} 
We first prove that  conditions (a)--(d) are satisfied by normal GM varieties. Let $X$ be such a  variety, defined by~\eqref{equation-x-from-data}, 
with $\dim(X) = n$ and  $\dim(W) = n + 5$.

\eqref{item-degree-kx}
Since the degree of $\CKGr(2,V_5)$ is 5 and the degree of $Q$ is 2, the dimensional transversality implies that the degree of $X$ is 10.
 Let $\dim (K) = k$. The canonical class of $\CKGr(2,V_5)$ is $-(5+k)H$. On the other hand,
the codimension of $W$ in $\bw2V_5 \oplus K$ is $(10 + k) - (n + 5) = 5 + k - n$, hence the canonical class of $X$ is
\begin{equation*}
K_X = (-(5 + k) + (5 + k - n) + 2)H = -(n-2)H. 
\end{equation*}

\eqref{item-w}
The very ampleness of $H$ is clear. By~\eqref{eq:dimx-dimw}, it is enough to show  $H^0(X,\cO_X(H)) = W^\vee$.  
We use  the  resolution  
\begin{equation*}
0 \to \cO(-5) \to V_5^\vee \otimes \cO(-3) \to V_5 \otimes \cO(-2) \to \cO \to \cO_{\CKGr(2,V_5)} \to 0
\end{equation*}
of the cone $\CKGr(2,V_5) $  in $\P(\bw2V_5 \oplus K)$.
Restricting it to $\P(W)$, tensoring with the resolution $0 \to \cO(-2) \to \cO \to \cO_Q \to 0$ of the quadric $Q$,
and using the dimensional transversality of the intersection, we obtain the   resolution 
\begin{multline}\label{eq:resolution-x}
0 \to \cO(-7) \to (V_5^\vee \oplus \k) \otimes \cO(-5) \to 
(V_5 \otimes \cO(-4)) \oplus (V_5^\vee \otimes \cO(-3)) \to \\ \to 
(V_5 \oplus \k) \otimes \cO(-2) \to 
\cO \to 
\cO_X \to 0
\end{multline}
of $X$ in $\P(W)$.
Twisting it   by $\cO(1)$, we get $H^0(X,\cO_X(H)) = H^0(\P(W),\cO_{\P(W)}(1)) = W^\vee$
(since  {$\dim(W) = n + 5 \ge 6$}, the only other term that could contribute is the term $\cO(-6)$ at the very beginning;   the contribution is non-zero only for $n = 1$, 
but it actually contributes to $H^1(X,\cO_X(H))$). 

\eqref{item-v6}
The resolution~\eqref{eq:resolution-x} implies that $X$ is an intersection of quadrics. 
Furthermore, twisting the resolution by $\cO(2)$, we see that $H^0(\P(W),\cI_X(2H)) = V_5 \oplus \k$
is 6-dimensional. 

\eqref{item-ux-simple} To prove that $\cU_X$ is simple, we may assume  $\k=\C$.

Assume   $n =   1$. By~\eqref{item-degree-kx}, we have $\deg(K_X) = 10$, so $X$ is a smooth curve of genus~$6$.
Moreover, by~\eqref{item-w} and~\eqref{item-v6}, its canonical class $K_X=H$ is very ample and  its canonical model is an intersection
of quadrics, hence $X$ is neither hyperelliptic, nor trigonal, nor a plane quintic.
 Mukai constructs in~\cite[Section 5]{mukcg} 
a  stable vector bundle of rank 2 on $X$,
proves that it embeds $X$ into  the cone $\CGr(2,5)$ over $\Gr(2,5)$ with vertex a point, and that $X$ is an intersection 
$\CGr(2,5) \cap \P^5 \cap Q$.  A combination of Proposition~\ref{en-g25}
and Lemmas~\ref{lemma:excess-cone}, \ref{en-hplane} and~\ref{en-quadric} then shows that Mukai's bundle is isomorphic to 
the twisted excess conormal bundle~$\cU_X$, which is therefore stable, hence simple.

We finish the proof by induction on $n $:  assume   $n \ge2$ and let $X' \subset X$ be a dimensionally transverse 
normal hyperplane section of $X$, so that $X'$ is a GM variety of dimension $n-1$, and~$\cU_{X'}$ is simple by the induction hypothesis. 
Applying  the functor $\Hom(\cU_X,-)$ to the exact sequence
\begin{equation*}
0 \to \cU_X(-H) \to \cU_X \to \cU_X\vert_{X'} \to 0
\end{equation*}
 and using pullback-pushforward adjunction and the isomorphism~$\cU_{X'} \cong \cU_X\vert_{X'}$ (Lemma~\ref{en-hplane}), we obtain an exact sequence
\begin{equation*}
0 \to \Hom(\cU_X,\cU_X(-H)) \to \Hom(\cU_X,\cU_X) \to \Hom(\cU_{X'},\cU_{X'}).
\end{equation*}
If $\dim(\Hom(\cU_X,\cU_X)) > 1$,  the simplicity of $\cU_{X'}$ implies $\Hom(\cU_X,\cU_X(-H)) \ne 0$.
On the other hand, an analogous argument produces an exact sequence
\begin{equation*}
0 \to \Hom(\cU_X,\cU_X(-2H)) \to \Hom(\cU_X,\cU_X(-H)) \to \Hom(\cU_{X'},\cU_{X'}(-H))=0,
\end{equation*}
which implies $\Hom(\cU_X,\cU_X(-2H)) \ne 0$. Iterating the argument, we see that  for all $k \ge 0$,
we have $\Hom(\cU_X,\cU_X(-kH)) \ne 0$. This is possible only if    $\cU_X$ has zero-dimensional torsion.
But it has not  since, by \eqref{defent}, it is a subsheaf of   $V_6 \otimes \cO_X$  {and $X$ is integral}.

We now prove that   conditions (a)--(d) are also sufficient. 

Let $X$ be a normal projective variety with a Cartier divisor $H$ which satisfies  properties (a)--(d) of Theorem~\ref{theorem:gm-intrinsic}.
Consider the spaces $W$ and $V_6$ of respective dimensions $n + 5$ and $6$, defined by conditions~\eqref{item-w} and~\eqref{item-v6}.
Since $V_6$ is a space of quadrics in $W$, it comes with a map 
\begin{equation*}
\bq\colon V_6 \to \Sym^2\!W^\vee. 
\end{equation*}

Consider the twisted excess conormal sheaf $\cU_X$. Its restriction $\cU_{X_{\rm sm}}$  to the smooth locus $X_{\rm sm}$ of~$X$  is locally free of rank 2.
 Since the rank of the conormal sheaf $\cN^\vee_{\Xreg/\P(W)}$ is 4, we compute, using the  exact sequence \eqref{defent},
 \begin{eqnarray*}
\det(\cU_\Xreg) &\cong& \det(\cN_{\Xreg/\P(W)}(-2H)) \cong \det(\cN_{\Xreg/\P(W)})(-8H)   \\ &\cong&
\cO_\Xreg(- 8H - K_{\P(W)}\vert_\Xreg + K_\Xreg) \cong \cO_\Xreg(- 8H + (n+5)H - (n-2)H) \\&\cong& \cO_\Xreg(-H).
\end{eqnarray*}
Set $L := H^0(\Xreg,\bw{2}\cU_\Xreg(H))$; by the above isomorphism and the normality of $X$, this vector space is one-dimensional 
and we have a canonical isomorphism
\begin{equation}\label{def-l}
L \otimes \bw{2}{\cU_\Xreg^\vee} = \cO_{X_{\rm sm}}(H).
\end{equation}
Dualizing the   embedding $\cU_\Xreg \hra V_6\otimes\cO_\Xreg$ and taking its wedge square, we get a surjection
$\bw{2}{V_6^\vee} \otimes \cO_\Xreg \thra \bw{2}{\cU_\Xreg^\vee}$. 
 Taking into account~\eqref{def-l}, we obtain a linear map
\begin{equation*}
L \otimes \bw{2}{V_6^\vee} \to H^0(\Xreg,\cO_\Xreg(H)) = H^0(X,\cO_X(H)) = W^\vee 
\end{equation*}
(we use again the normality of $X$  {to identify sections on $X$ and on  $\Xreg$})
 and by duality a linear map $\mu \colon L \otimes W \to \bw{2}{V_6}$. This map can be factored through an injection $\bmu\colon  L \otimes W \hra \bw{2}{V_6} \oplus K$
for some vector space $K$. The   maps
\begin{equation*}
\Xreg \hookrightarrow \P(W) \cong \P(L \otimes W) \hookrightarrow \P(\bw2V_6 \oplus K) \dashrightarrow \P(\bw2V_6)
\end{equation*}
and 
\begin{equation*}
\Xreg \to \Gr(2,V_6) \hookrightarrow \P(\bw2V_6), 
\end{equation*}
 {where the map $\Xreg \to \Gr(2,V_6)$ is induced by the embedding $\cU_\Xreg \hookrightarrow V_6 \otimes \cO_\Xreg$,}
are given by the same linear system, hence they agree. 
This shows $X \subset \cone{K}\!\Gr(2,V_6)$ and $X_{\rm sm} \subset \pcone{K}\!\Gr(2,V_6)$, so by Proposition~\ref{en-funct},
we have a commutative diagram
\begin{equation*}
\xymatrix@R=5mm@C=2.5mm
{
0 \ar[r] & (V_6\otimes\cU_{X_{\rm sm}})/\Sym^2\!\cU_{X_{\rm sm}} \ar[r] \ar[d] & \bw{2}{V_6} \otimes \cO_{X_{\rm sm}} \ar[r] \ar[d] & \det (V_6) \otimes \bmu^*\cN^\vee_{\Gr(2,V_6)}(2) \ar[r] \ar[d] & 0 \\
0 \ar[r] & L^{ -2} \otimes \det (V_6) \otimes \cU_{X_{\rm sm}} \ar[r] & L^{ -2} \otimes \det (V_6) \otimes V_6\otimes\cO_{X_{\rm sm}} \ar[r] & L^{ -2} \otimes \det (V_6) \otimes \cN^\vee_{X_{\rm sm}}(2) \ar[r] & 0,
}
\end{equation*}
where the top row comes from the pullback to $X_{\rm sm}$ of the excess conormal sequence of $\Gr(2,V_6)$ (see Proposition \ref{en-g26}) and we twisted everything by $\det (V_6)$. 
The left vertical arrow induces a morphism $\lambda'\colon V_6\otimes\cU_{X_{\rm sm}} \to L^{ -2} \otimes \det (V_6) \otimes \cU_{X_{\rm sm}}$. 

By condition~\eqref{item-ux-simple}, the sheaf $\cU_X$ is simple, hence $\cU_{X_{\rm sm}}$ is simple as well: $\cU_X$ is torsion-free,
 hence any endomorphism of $\cU_{X_{\rm sm}}$ extends to an endomorphism of $\cU_X$. Therefore,  $\lambda'$ is given 
by a linear form
\begin{equation*}
\lambda\colon V_6 \to L^{ -2} \otimes\det (V_6). 
\end{equation*}

Since  $\lambda'$ vanishes on $\Sym^2\!\cU_{X_{\rm sm}}$, 
the image of $\cU_{X_{\rm sm}}$ in $V_6\otimes\cO_{X_{\rm sm}}$ (via the sequence \eqref{defent}) is contained in $\Ker(\lambda)\otimes\cO_{X_{\rm sm}}$.
Moreover,  the middle vertical map in the diagram above is given by $v_1\wedge v_2 \mapsto \lambda(v_1)v_2 - \lambda(v_2)v_1$.

Let us show that the form $\lambda$ is non-zero. If $\lambda = 0$, 
the middle vertical map in the diagram is zero, which means 
that all the quadrics cutting out $\cone{K}\!\Gr(2,V_6)$ contain $\P(W)$, \ie, $\P(W) \subset \cone{K}\!\Gr(2,V_6)$. In other words,
$\P(W)$ is a cone  over $\P(W') \subset \Gr(2,V_6)$, with vertex   a subspace of $K$. The map $X_{\rm sm} \to \Gr(2,V_6)$
induced by $\cU_{X_{\rm sm}}$ therefore factors through $\P(W')$, \ie, the vector bundle $\cU_{X_{\rm sm}}$ is a pullback 
from $\P(W')$ of the restriction of the tautological bundle of $\Gr(2,V_6)$ to $\P(W')$. We show that this is impossible. 

There are two types of linear spaces on $\Gr(2,V_6)$: the first type corresponds to   2-dimensional subspaces containing a given vector
and the second type to those contained in a given 3-subspace $V_3 \subset V_6$. If $W'$ is of the first type,   the restriction of the tautological
bundle to $\P(W')$ is isomorphic to $\cO \oplus \cO(-1)$, hence $\cU_\Xreg \cong \cO_\Xreg \oplus \cO_\Xreg(-H)$. In particular, it is not simple,
which is a contradiction. If $W'$ is of the second type, the embedding $\cU_X \to V_6 \otimes \cO_X$ factors through
a subbundle $V_3 \otimes \cO_X \subset V_6 \otimes \cO_X$. Recall that $V_6$ is the space of quadrics passing through $X$ in $\P(W)$.
Consider the scheme-theoretic intersection $M$ of the quadrics corresponding to the vector subspace $V_3 \subset V_6$. Since the embedding of the excess conormal sheaf
factors through $V_3 \otimes \cO_X$, the variety $X$ is the complete intersection of~$M$ with the 3 quadrics corresponding to the quotient space $V_6/V_3$.
But   the degree of $X$ is then divisible by~8, which contradicts the fact that it is 10 by~\eqref{item-degree-kx}.

Thus $\lambda \ne 0$ and 
 $V_5 := \Ker(\lambda)$
 is a hyperplane in $V_6$. It fits in the  exact sequence
\begin{equation*}
0 \to V_5 \to V_6 \xrightarrow{\ \lambda\ } L^{ -2} \otimes\det (V_6) \to 0,
\end{equation*}
which induces a canonical isomorphism 
\begin{equation*}
  \varepsilon \colon \det (V_5) \xrightarrow{\ \sim\ } L^{\otimes 2}.
\end{equation*}
Moreover, the composition $\cU_X \hookrightarrow V_6\otimes\cO_X \xrightarrow{\ \lambda\ } L^{\otimes 2} \otimes \det (V_6) \otimes \cO_X$
vanishes on $X_{\rm sm}$, hence on $X$ as well. This shows that the embedding $\cU_X \hookrightarrow V_6 \otimes \cO_X$ factors through $V_5 \otimes \cO_X$.
  
We now replace $V_6$ with $V_5$ and repeat the above argument. We   get a linear map 
\begin{equation*}
\mu \colon  L \otimes W \to \bw{2}{V_5}
\end{equation*}
and an embedding $\bmu\colon  L \otimes W \hra \bw{2}{V_5} \oplus K$ which induce embeddings $X \subset \cone{K}\!\Gr(2,V_5)$ and $X_{\rm sm} \subset \pcone{K}\!\Gr(2,V_5)$.
The commutative diagram 
\begin{equation*}
\xymatrix@R=5mm{
0 \ar[r] & \cU_{X_{\rm sm}} \ar[r] \ar[d] & V_5 \otimes \cO_{X_{\rm sm}} \ar[r] \ar[d] & \bmu^*\cN^\vee_{\Gr(2,V_5)} \ar[r] \ar[d] & 0 \\
0 \ar[r] & \cU_{X_{\rm sm}} \ar[r] & V_6\otimes\cO_{X_{\rm sm}} \ar[r] & \cN^\vee_{X_{\rm sm}} \ar[r] & 0
}
\end{equation*}
of Proposition~\ref{en-funct} (we use  $\varepsilon$   to cancel out the twists by $\det (V_5^\vee)$ and by $L^2$ in the top and bottom rows) 
then shows that inside the space $V_6$ of quadrics cutting out $X$ in $\PP(W)$, the hyperplane $V_5$ is the space of quadratic equations of $\Gr(2,V_5)$, \ie,
of Pl\"ucker quadrics.

As the Pl\"ucker quadrics cut out
the cone $\CKGr(2,V_5)$  in $\P(\bw{2}{V_5} \oplus K)$,  they cut out 
$\CKGr(2,V_5) \cap \P(W)$ in $\P(W)$. Since $X$ is the intersection of 6 quadrics by condition~\eqref{item-v6}, 
we finally obtain 
\begin{equation}\label{des} 
X = \CKGr(2,V_5) \cap \P(W) \cap Q,
\end{equation}
where $Q$ is any non-Pl\"ucker quadric (corresponding to a point in $V_6 \setminus V_5$),
so $X$ is a GM variety.
\end{proof}

 {A   consequence of   Theorem~\ref{theorem:gm-intrinsic} is  that being a GM variety is a geometric property.
Recall that the base field $\k$ is by our assumption a subfield of $\C$. We denote by $X_\C$ the extension of scalars from $\k$ to $\C$,
and by $H_\C$ the induced polarization of $X_\C$. 

\begin{coro}\label{corollary:x-c}
A normal polarized variety $(X,H)$ is a polarized GM variety    if and only if $(X_\C,H_\C)$ is.
\end{coro}

\begin{proof}
One direction is clear from the definition, so we only have to check that if $(X_\C,H_\C)$ is GM, so is $(X,H)$.
We show that if the conditions of Theorem~\ref{theorem:gm-intrinsic} are satisfied for $(X_\C,H_\C)$,   they are satisfied for $(X,H)$ as well.
The only property  for which this is not completely obvious is the equality $K_X = -(n-2)H$, but it follows   from the injectivity of the map $\Pic(X) \to \Pic(X_\C)$,
which holds since $X$ is integral and projective.
\end{proof}}

 {We  introduce some more terminology. Given a GM variety $X$ defined by~\eqref{equation-x-from-data}, the
  twisted excess conormal sheaf $\cU_X$ that was crucial for the proof of Theorem~\ref{theorem:gm-intrinsic} will be called its {\sf   Gushel sheaf}.
 As we   showed in the proof, the projection of $X$  from the   vertex $\P(K)$ of the cone $\cone{K}\Gr(2,V_5)$
defines a rational map $X \dashrightarrow \Gr(2,V_5)$  and the Gushel sheaf $\cU_X$ is isomorphic (on the smooth locus of $X$) to the pullback under this map of the tautological vector bundle on $\Gr(2,V_5)$.
The map $X \dashrightarrow \Gr(2,V_5)$ is thus determined by~$\cU_X$ and   is canonically associated with $X$.
We call this map {\sf the Gushel map} of $X$.}

\subsection{GM data}

In the course of the proof of Theorem~\ref{theorem:gm-intrinsic}, we associated with any   normal polarized  GM variety
  vector spaces $W$, $V_6$, $V_5$,   $L$, and   maps $\bq$, $\mu$,~$\varepsilon$. We axiomatize these as follows.

\begin{defi}\label{defgm}
A {\sf GM data $(W,V_6,V_5,L,\mu,\bq,\eps)$ of dimension $n$} (over $\k$) consists of
\begin{itemize}
\item a $\k$-vector space $W$ of dimension $n + 5$,
\item a $\k$-vector space $V_6$ of dimension $6$,
\item a  {$\k$}-hyperplane $V_5 \subset V_6$,
\item a $\k$-vector space $L$  {of dimension $1$}, 
\item a $\k$-linear map $\mu \colon  {L \otimes}W \to \bw2{V_5}$,  
\item a $\k$-linear map $\bq \colon  V_6 \to \Sym^2\!W^\vee$,
\item a  $\k$-linear isomorphism $\eps\colon  \det(V_5) \to L^{\otimes 2}$,
\end{itemize}
such that 
the following diagram commutes
\begin{equation}\label{equation-mu-q-plucker}
\vcenter
{\xymatrix@R=5mm@C=40pt@M=7pt
{
\,V_5\, \ar@{^{(}->}[r] \ar[d]_-{\varepsilon}^<<[right]{\sim}& V_6 \ar[d]^-{\bq} \\
{L^{\otimes 2} \otimes} \bw{4}{V_5^\vee} \ar[r]^-{\Sym^2\!\mu^\vee} & \Sym^2\!W^\vee.
}}
\end{equation}
 {In other words, $\bq(v)(w_1,w_2) = \eps(v \wedge \mu(w_1) \wedge \mu(w_2))$ for all $v \in V_5$ and $w_1,w_2 \in W$.}
\end{defi}

\begin{defi}\label{def:gm-data-iso}
An {\sf isomorphism of GM data} between GM data $(W,V_6,V_5,L,\mu,\bq,\eps)$ and\break $(W',V'_6,V'_5,L',\mu',\bq',\eps')$ is a triple of $\k$-linear isomorphisms 
$\varphi_W\colon W \to W'$, $\varphi_V\colon V_6 \to V'_6$, and $\varphi_L\colon L \to L'$,
such that 
\begin{equation*}
\varphi_V(V_5) = V'_5 
, \qquad
{\eps' \circ \bw5{\varphi_V} = \varphi_L^{\otimes 2}\circ  \eps,} 
\end{equation*}
and the following diagrams commute
\begin{equation*}
 \xymatrix{
 V_6 \ar[d]_{\varphi_V} \ar[r]^-{\bq} & \Sym^2\!W^\vee  \\
V'_6 \ar[r]^-{\bq'} & \Sym^2\!W'^\vee \ar[u]_{\Sym^2\!\varphi_W^\vee}
}
\quad\quad\quad
\quad\quad\quad
\xymatrix{
 L \otimes W \ar[d]_{{\varphi_L \otimes }\varphi_W} \ar[r]^-{\mu} & \bw2V_5 \ar[d]^{\sbw2\varphi_V}  \\
L' \otimes W' \ar[r]^-{\mu'} & \bw2V'_5 .
} 
\end{equation*}
 In particular, the automorphism group 
 of a GM data $(W,V_6,V_5,L,\mu,\bq,\eps)$ is
the subgroup of $\GL(W) \times \GL(V_6) \times \Gm$ of elements $(g_W,g_V,g_L)$
 such that
\begin{equation*}
 g_V(V_5) = V_5,
\qquad 
\det(g_V|_{V_5}) = g_L^2,
\qquad
(\Sym^2\!g_W^\vee) \circ \bq \circ g_V = \bq,
\qquad 
(\bw2g_V) \circ \mu = \mu \circ (g_L \otimes g_W).
 \end{equation*}
 {We consider the set of all polarized GM data as a groupoid, with this notion of   isomorphism.} 
  \end{defi}

\begin{lemm}\label{lemma:from-x-to-data}
The collection $(W,V_6,V_5,L,\mu,\bq,\eps)$ associated by Theorem~\textup{\ref{theorem:gm-intrinsic}} with a  normal polarized  GM variety $(X,H)$ is a GM data. 
The association 
\begin{equation*}
(X,H) \longmapsto (W,V_6,V_5,L,\mu,\bq,\eps) 
\end{equation*}
is a  {fully faithful} functor from the groupoid of  {normal} polarized GM varieties to the groupoid of GM data,
\ie, any isomorphism of polarized GM varieties induces an isomorphism of the associated GM data  {and vice versa}.
 {Moreover, this association} works in families.
\end{lemm}

\begin{proof}
In the course of the proof of Theorem~\ref{theorem:gm-intrinsic}, we showed that the subspace $V_5 \subset V_6$ of quadrics through $X$
cuts out $\CKGr(2,V_5) \cap \P(W) $ in $ \P(W)$. Thus these quadrics are restrictions to $\P(W)$ of the Pl\"ucker quadrics. This is equivalent
to the commutativity of~\eqref{equation-mu-q-plucker}, so the constructed data is a GM data.

Let $(\phi,\psi)$ be an isomorphism of polarized GM varieties between $(X,H)$ and $(X',H')$. We denote by $(W',V'_6,V'_5,L',\mu',\bq',\eps')$ the GM data for $X'$.
The isomorphism $ (\phi,\psi)$ induces an isomorphism between $W$ and $W'$, and also between $V_6$ and $V_6'$. Moreover,
it induces an isomorphism between the excess conormal sheaves of $X$ and $X'$, and an isomorphism between $L$ and $L'$,
compatible with the isomorphisms~\eqref{def-l}. These isomorphisms are compatible 
with the hyperplanes $V_5$ and with the maps $\bq$, $\mu$, and $\varepsilon$, hence provide an isomorphism of the associated GM data.
Moreover, the composition of isomorphisms of polarized GM varieties corresponds to the composition of the corresponding
isomorphisms of GM data. This proves the functoriality.

 {Conversely, an isomorphism of GM data includes an isomorphism $\varphi_W \colon \P(W) \to \P(W')$ which induces an isomorphism  of polarized GM varieties between $X$ and $X'$.
This proves that the functor is fully faithful.}

Finally, given a flat 
family $\cX \to S$ of normal projective varieties with a Cartier divisor~$H$ on~$\cX$, such that each fiber
of the family satisfies the conditions of Theorem~\ref{theorem:gm-intrinsic}, the same  construction provides  vector bundles $\cW$, $\cV_6$, $\cV_5$,  $\cL$ on $S$
(in the definition of  {$\cV_6$, $\cW$, and~$\cL$,}
one should replace global sections with   pushforwards to $S$)
and   maps $\bq\colon \cV_6 \to \Sym^2\!\cW^\vee$, $\mu\colon \cL \otimes \cW \to \bw2\cV_5$, and $\varepsilon\colon \det(\cV_5) \to \cL^{\otimes 2}$.
This shows that the association works in families.
\end{proof}

The following lemma characterizes the image of the functor.
Let $(W,V_6,V_5,L,\mu,\bq,\eps)$ be a GM data. 
For each vector $v \in V_6$, we have a quadratic form $\bq(v) \in \Sym^2\!W^\vee$ and we denote by~$Q(v) \subset \PP(W)$ the subscheme it defines (a quadratic hypersurface when $\bq(v)\ne 0$).

\begin{lemm}\label{lemma:from-data-to-x}
If $(W,V_6,V_5,L,\mu,\bq,\eps)$ is a GM data of dimension $n$,
\begin{equation}\label{equation-x-intersection-quadrics}
X = X(W,V_6,V_5,L,\mu,\bq,\eps) := \bigcap_{v \in V_6} Q(v)   \subset \P(W) 
\end{equation}
is a GM intersection of dimension $ \ge n$. It is a GM variety if and only if $X$ is integral of dimension $n$.
\end{lemm}

\begin{proof}
Lift the map $\mu\colon L \otimes W \to \bw2V_5$ to an embedding $\bar\mu\colon L \otimes W \hookrightarrow \bw2V_5 \oplus K$ for some~$K$
(for example, one can take $K = \Ker(\mu)$).
It follows from the commutativity of~\eqref{equation-mu-q-plucker} that for $v \in V_5$,
the quadrics $Q(v)$ are the restrictions to $\PP(W) = \P(L \otimes W)$ of the cones (with vertex~$\P(K)$) over the Pl\"ucker quadrics 
in $\P(\bw2V_5)$ with respect to the embedding $\bmu$. Since the intersection of   the Pl\"ucker quadrics
is $\Gr(2,V_5)$, this implies
\begin{equation}\label{dess}
\bigcap_{v \in V_5} Q(v) = \CKGr(2,V_5) \cap \PP(W) 
\end{equation}
 and, for any $v\in V_6\setminus V_5$, we have
\begin{equation*} 
X = \CKGr(2,V_5) \cap \P(W) \cap Q(v).
\end{equation*} 
This is a GM intersection, which is a GM variety if and only $X$ is integral of dimension $n$.\end{proof}

\begin{theo}\label{theorem:gm-var-data}
The constructions of Lemmas~\textup{\ref{lemma:from-x-to-data}} and~\textup{\ref{lemma:from-data-to-x}} are mutually inverse
and define  {an equivalence of groupoids} between
\begin{itemize}
\item  the  {groupoid} of all normal polarized GM varieties $(X,H)$ of dimension~$n$ 
 and  
\item the  {groupoid} of all GM data $(W,V_6,V_5,L,\mu,\bq,\eps)$ of dimension $n$ such that  {the GM intersection} $X(W,V_6,V_5,L,\mu,\bq,\eps)$ is $n$-dimensional and normal.
\end{itemize}
This equivalence {induces an injection from the set of isomorphism classes of normal polarized GM varieties into the set of isomorphism classes of GM data and} also  works in families.
\end{theo}

\begin{proof}
Let $(X,H)$ be a GM variety and let $(W,V_6,V_5,L,\mu,\bq,\eps)$ be the associated GM data. We saw during the proof of Theorem~\ref{theorem:gm-intrinsic} that one can we can write $X$ as  in~\eqref{des}. Comparing with~\eqref{dess}, we obtain $X(W,V_6,V_5,L,\mu,\bq,\eps) \cong X$.
Conversely, start with a GM data $(W,V_6,V_5,L,\mu,\bq,\eps)$ and set $X: = X(W,V_6,V_5,L,\mu,\bq,\eps)$. Applying to $X$
the construction of Theorem~\ref{theorem:gm-intrinsic},  one gets the original GM data back.
Functoriality and the fact that the equivalence works in families are clear from the construction.
\end{proof}

\begin{rema}\label{remark:cone-kernel}
The proof of Lemma~\ref{lemma:from-data-to-x} shows that if $X$ is a normal polarized GM variety, one can always assume   $K = \Ker(\mu)$.
Thus a normal GM variety with GM data $(W,V_6,V_5,L,\mu,\bq,\eps)$ can be written as
\begin{equation*}
X = \cone{\Ker(\mu)}(\Gr(2,V_5) \cap \P(\mu(W))) \cap Q(v)
\end{equation*}
for any $v \in V_6 \setminus V_5$.
\end{rema}

Recall that $\Aut(W,V_6,V_5,L, \mu,\bq,\eps) $ is a subgroup of $ \GL(W) \times \GL(V_6) \times \Gm$.

\begin{coro}\label{corollary-aut-x-aut-data}
Let $(X,H)$ be a normal polarized GM variety, with associated GM data $(W,V_6,V_5,L, \mu,\bq,\eps) $. 
 {There is an isomorphism 
 \begin{equation*}
\Aut(X,H) \cong \Aut(W,V_6,V_5,L,\mu,\bq,\eps)б
\end{equation*}
of algebraic groups
 and} an exact sequence
\begin{equation*}
1 \to \Gm \to \Aut(W,V_6,V_5,L,\mu,\bq,\eps) \to \Aut_H(X) \to 1.
\end{equation*}
\end{coro}

\begin{proof}
Follows from Theorem~\ref{theorem:gm-var-data} and~\eqref{eq:aut-x-h}.
 \end{proof}

\subsection{Mildly singular GM varieties}\label{sec23}

In this section, we give a simpler intrinsic characterization of mildly singular GM varieties which will allow us to extend the Gushel--Mukai classification.

We  say that a smooth projective   curve $X$   of genus 6 is {\sf Clifford general}   if $X_\C$ is neither hyperelliptic, 
nor trigonal, nor a plane quintic.

\begin{prop}\label{gm-curves-surfaces}
A smooth projective curve is a GM curve if and only if it is a Clifford general curve of genus $6$.
 Equivalently, its canonical model  is an intersection of 6 quadrics in~$\P^5$.  
\end{prop}

\begin{proof}
This follows from Theorem~\ref{theorem:gm-intrinsic} combined with the Enriques--Babbage theorem (\cite[Section~III.3]{acgh}).
\end{proof}

 {An analogous characterization  holds for GM surfaces.} 
Following Mukai (\cite[Definition~3.8]{mukai1995new} (with a misprint in the English translation),  \cite[Definition~10.1]{johnsen2004k3},  \cite[Definition~1.3]{glt}),
we say that a complex polarized K3 surface $(S,H)$   is {\sf Brill--Noether  general} if   
\begin{equation*}
h^0(S,D) h^0(S,H-D) < h^0(S,H) 
\end{equation*}
for all divisors $D$ on $S$ not linearly equivalent to $0$ or $H$.  When $H^2=10$, this is equivalent, by~\cite[Proposition  10.5]{johnsen2004k3}, to the fact that $|H|$ contains a Clifford general smooth curve
(by \cite{gl}, all smooth curves in $|H|$ are then Clifford general).

 \begin{prop}\label{gm-surfaces}
A smooth projective surface $X$ is a GM surface if and only if $X_\C$   is a Brill--Noether general polarized $K3$ surface of degree~$10$.
\end{prop}

\begin{proof}
By Corollary~\ref{corollary:x-c}, we may assume $\k = \C$.
If $X$ is a smooth GM surface, we have $K_X = 0$ by Theorem~\ref{theorem:gm-intrinsic}(a), and the resolution~\eqref{eq:resolution-x} implies   $H^1(X,\cO_X) = 0$, hence $X$ is a K3 surface.
Moreover, a general hyperplane section of $X$ is a GM curve, hence a Clifford general curve of genus~6, hence $X$ is Brill--Noether general.
The other direction is proved in Theorem~\ref{theorem:gushel-mukai} below.
\end{proof}

If $X$ is a GM variety, any integral hyperplane section of $X$ is also a GM variety.
We are going to show that the converse is also true under some additional assumptions.
We start with the following result.

\begin{lemm}\label{lemma:gm-projectively-normal}
If $(X,H)$ is a normal polarized GM variety, it is projectively normal, \ie, the canonical maps 
\begin{equation*}
\Sym^m\! H^0(X,\cO_X(H)) \to H^0(X,\cO_X(mH))
\end{equation*}
are surjective for all $m \ge 0$.
\end{lemm}

\begin{proof}
Take any integer $m$ and consider the twist \begin{multline*}
0 \to \cO_{\P(W)}(m-7) \to 
\cO_{\P(W)}(m-5)^{\oplus 6} \to 
\cO_{\P(W)}(m-4)^{\oplus 5} \oplus \cO_{\P(W)}(m-3)^{\oplus 5} \to \\ \to 
\cO_{\P(W)}(m-2)^{\oplus 6} \to 
\cO_{\P(W)}(m) \to 
\cO_X(mH) \to 0  
\end{multline*}
by $\cO(m)$ of the resolution~\eqref{eq:resolution-x}.
We want to show that the map on global sections induced by the rightmost map is surjective. 
This holds because, $\P(W)$ being a projective space of dimension $n + 4 \ge 5$,
 the spaces
$H^1(\P(W),\cO_{\P(W)}(m-2))$, $
H^2(\P(W),\cO_{\P(W)}(m-3))$, $ 
H^2(\P(W),\cO_{\P(W)}(m-4)) $, $   
H^3(\P(W),\cO_{\P(W)}(m-5)) $, $  
H^4(\P(W),\cO_{\P(W)}(m-7))$ all vanish.
\end{proof}

\begin{prop}\label{proposition:hyperplane-criterion}
Let $(X,H)$ be a normal polarized  variety of dimension $n \ge 2$ such that  $K_X = -(n-2)H$  and $H^1(X,\cO_X) = 0$.
If there is a hypersurface $X' \subset X$ in the linear system $|H|$ such that $(X',H\vert_{X'}) $ is a normal polarized GM variety,   $(X,H)$ is also a polarized GM variety.
\end{prop}

\begin{proof}
Let us verify the conditions of Theorem~\ref{theorem:gm-intrinsic}. Condition \eqref{item-degree-kx} is true by our assumptions, since $H^n = H\cdot H^{n-1} =(H\vert_{X'})^{n-1}=10$.
Furthermore,  $X'$ is projectively normal by Lemma~\ref{lemma:gm-projectively-normal}. 
Using~\cite[Lemma~(2.9)]{iskovskih1977fano}, we conclude that $X$ is also projectively normal.
Since $H$ is ample,   it is very ample by~\cite[p.~38]{mum}. The  exact sequence
\begin{equation*}
0 \to \cO_X \to \cO_X(H) \to \cO_{X'}(H) \to 0
\end{equation*}
with the assumption $H^1(X,\cO_X)=0$ imply   
\begin{equation*}
h^0(X,\cO_X(H)) = h^0(X',\cO_{X'}(H)) + 1 = n + 5. 
\end{equation*}
This proves~\eqref{item-w}, and  \cite[Lemma~(2.10)]{iskovskih1977fano} proves~\eqref{item-v6}. Finally, since $\cU_{X'}$ is simple (because $X'$ is a GM variety),
the argument from the proof of Theorem~\ref{theorem:gm-intrinsic}\eqref{item-ux-simple} proves that $\cU_X$ is also simple.  
\end{proof}

\begin{theo}\label{theorem:gushel-mukai}
Let $X$ be a normal locally factorial complex 
projective variety of dimension $n \ge 1$,
with  terminal  singularities and $\codim (\Sing(X) )\ge 4$, together with an ample Cartier divisor $H$ 
such that $K_X \lin -(n-2)H$ and $H^n = 10$. If we assume that
\begin{itemize}
\item when $n\ge 3$, we have $\Pic(X) = \Z H$;
\item when $n= 2$, the surface $X $ is a Brill--Noether general K3 surface;
 \item when $n= 1$, the genus-$6$ curve $X $ is Clifford general;
 \end{itemize}
 then $X$ is a GM variety.
\end{theo}

To prove the theorem, we   verify that a general hyperplane section of $X$ satisfies the same conditions,
then   use   induction on $n$. 

\begin{lemm}\label{fgmhyp}
Let $X$ be a normal complex projective  variety   of dimension $n\ge 3$ 
with 
terminal singularities, $\codim (\Sing(X) )\ge 4$, and with an ample Cartier divisor $H$ such that $H^n = 10$ and 
 $K_X \lin -(n-2)H$. 
 
 The linear system $|H|$ is very ample and a   general $X'\in |H|$ satisfies the same conditions: 
$X'$ is normal with  terminal singularities, $\codim (\Sing(X') )\ge 4$ and, if  $H':=H\vert_{X'}$, we have $(H')^{n-1} = 10$ and 
 $K_{X'} \lin -(n-3)H'$. 
 \end{lemm}

\begin{proof} 
By \cite[Proposition~1-1]{alexeev1991theorems}, the linear system $|H|$ is non-empty and  by \cite[Remark 2.6]{mel}, a general   
$X'\in |H|$ is normal with terminal singularities. Moreover, $(H')^{n-1}=10$ and, by the adjunction formula, 
\begin{equation*}
K_{X'} = (K_X + H)\vert_{X'} \lin -(n-3)H\vert_{X'}.
\end{equation*}
Since $X$ has terminal singularities, we have  $H^1(X,\cO_X)=0$ by Kawamata--Viehweg vanishing,
hence the linear 
series $|H'|$ is just the restriction  of $|H|$ to $X'$ and the base loci of $|H|$ and $|H'|$ are the same.
Taking successive linear sections, we arrive at a linear section $Y$ of dimension 3 with 
terminal singularities, $K_Y\lin -H_Y$, and $H_Y^3=10$.

Assume some $x\in   \Sing(X)$ is in the base locus of $ |H| $. That point $x$ is then on $X'$,  and, since a Cartier divisor 
is necessarily singular at a singular point of the ambient variety, it is also singular on $X'$. Repeating that argument, 
we see that $x$ is singular on the threefold $Y$, and still a base point of $|H_Y|=|H|\vert_Y$. Again, all divisors in $|H_Y|$ are singular 
at $x$, hence we are in one of the two   cases described in~\cite[Theorem 2.9]{mel}. Since  $H_Y^3\ne 10$ in both of these cases, 
we get a contradiction. Therefore, $\Sing(X)\cap \Bs(|H|)=\vide$ hence, by Bertini's theorem, $\Sing(X')$ still has codimension 
at least 4 in $X'$ and eventually, $Y$ is smooth.
  
If $\Pic(Y)=\Z H_Y$, the pair $(Y, H_Y)$ is projectively normal by \cite[Corollary~4.1.13]{ip}. 

If not, we use the Mori--Mukai classification of smooth Fano threefolds $Y$ with Picard number $ \ge 2$  (\cite{momu}, \cite[Chapter 12]{ip}) and see 
that there is only one family with  {anticanonical} degree~10: $Y$ must be a divisor of bidegree $(3,1)$ in $\P^3\times \P^1$ and the pair $(Y, H_Y)$ is again projectively normal.

A repeated use of  {\cite[Lemma~2.9]{iskovskih1977fano}} then implies that in all cases, $(X,H)$ is projectively normal, hence $H$ is very ample by~\cite[p.~38]{mum}.
\end{proof}
 
\begin{lemm}\label{lemma-fgm-hyperplane}
Let $(X,H)$ be a polarized complex variety of dimension $n\ge 2$   which satisfies the hypotheses of Theorem~\textup{\ref{theorem:gushel-mukai}}. 
A general element of $|H|$ then satisfies the same properties.
 \end{lemm}
 
\begin{proof} 
Assume first $n>3$. 
By Lemma \ref{fgmhyp},  we need only to prove that a general $X'\in |H|$  is locally factorial and that $\Pic(X')$ is generated 
by $H':=H\vert_{X'}$. Let $\Cl(X)$ be the group of Weil divisors on $X$, modulo linear equivalence. 
Since $|H|$ is very ample (Lemma \ref{fgmhyp}), restriction of divisors induces, by~\cite[Theorem~1]{rs}, an isomorphism
\begin{equation*}
\Cl(X)\isomlra \Cl(X').
\end{equation*}
Since $X$ is normal and locally factorial, the canonical inclusion $\Z H=\Pic(X)\hra \Cl(X)$ is an isomorphism. 
It follows that the canonical inclusion $\Pic(X')\hra \Cl(X')$ is also an isomorphism, \ie, $X'$ is locally factorial
and $\Pic(X')$ is generated by $H'$. 
 
When $n=2$, $X$ is a smooth Brill--Noether general K3 surface with a polarization $H$ of degree 10,
hence a general element of $|H|$ is a smooth Clifford general curve of genus 6.

When $n=3$, $X$ is a smooth Fano threefold with $\Pic(X) \cong \Z$. By \cite[Corollary~4.1.13]{ip}, it is an intersection of quadrics. 
Any smooth hyperplane section $S$ of $X$ is a degree-10 smooth K3 surface which is still an intersection of quadrics.
A general hyperplane section   of $S$ is still an intersection of quadrics, hence is
 a Clifford general curve. This proves that $S$ is Brill--Noether general.  \end{proof}

We can now prove Theorem~\ref{theorem:gushel-mukai}.

\begin{proof}[Proof of Theorem~\textup{\ref{theorem:gushel-mukai}}]
We use induction on $n$.  The case  $n = 1$ was  proved in Proposition~\ref{gm-curves-surfaces}, so we  assume $n \ge 2$.
A   general hyperplane section $X'$ of $X$ has the same properties
by Lemma~\ref{lemma-fgm-hyperplane}, hence is a GM variety by the induction hypothesis. On the other hand,
we have $H^1(X,\cO_X) = 0$:  this follows from  {Kawamata--Viehweg}  
vanishing for $n \ge 3$, and from the fact that  $X$ is a K3 surface for $n = 2$. By Proposition~\ref{proposition:hyperplane-criterion}, we conclude that $X$ is a GM variety.
 \end{proof}

\begin{rema}\label{remt7}
If we relax the conditions on the  Picard group or on the singularities  of $X$, the conclusion of Theorem~\ref{theorem:gushel-mukai} may be false.  
Indeed, a general divisor of bidegree $(3,1)$ in~$\P^3\times\P^1$ is an example of a smooth Fano threefold of coindex 3 and degree 10 which is not
an intersection of quadrics. Further  counterexamples can be found in~\cite{przhyjalkowski2005hyperelliptic}:  one is the threefold~$T_7$ in Theorem 1.6 
of that article, which can be defined as the anticanonical image of the blow up of a quartic double solid in a line 
(\cite[Example 1.11]{przhyjalkowski2005hyperelliptic}). It has one ordinary double point hence has terminal Gorenstein singularities, 
but is not locally  factorial, because it is the image of a non-trivial small contraction. Its Picard number is 1 and   $|-K_{T_7}|$ is 
very ample of degree~10 and projective dimension 7, but $T_7$ is not an intersection of quadrics.  In particular, the conclusion of
Theorem~\ref{theorem:gushel-mukai} is definitely false for it.
\end{rema}

\subsection{Grassmannian hulls}\label{subsec-grhulls}

We fix some more terminology and notation.
Given a  normal polarized GM variety $(X,H)$ of dimension $n$, we consider the     GM data $(W,V_6, V_5,L,\mu,\bq,\eps)$
 constructed  in  Lemma~\ref{lemma:from-x-to-data} 
and define the {\sf Grassmannian hull} of $X$ by
\begin{equation*}
M_X := \bigcap_{v \in V_5} Q(v) = \cone{K}\!\Gr(2,V_5) \cap \P(W),
\end{equation*}
 {where the second equality is~\eqref{dess}.}
Since $\dim(X)=n$, this intersection is dimensionally transverse and  $M_X$ is a variety of degree 5 and dimension $n+1$, defined over the field~$\k$.
Moreover, by Lemma~\ref{lemma:from-data-to-x}, we have
\begin{equation*}
X = M_X \cap Q(v) 
\end{equation*}
 for any $v \in V_6 \setminus V_5$. Since $X$ is   irreducible, so is $M_X$.

When $\mu $ is not injective, the Grassmannian hull is   a cone.
It is   convenient  to consider also the dimensionally transverse intersection
\begin{equation*}
M'_X := \Gr(2,V_5) \cap \P(\mu (W)).
\end{equation*}
 {In terms of $M'_X$, the structure of a GM variety can be described as follows.

\begin{lemm}\label{lemma:gushel-map}
Let $X$ be a GM variety with associated   GM data $(W,V_6, V_5,L,\mu,\bq,\eps)$. Set $K := \Ker(\mu)$ and  $k := \dim(K)$, choose $v \in V_6 \setminus V_5$ and set $Q := Q(v)$.

If $k = 0$, we have  $X = M'_X \cap Q$. 

If $k > 0$, 
let $\tilde{X}$ be the blow up of $X$ with center   $\P(K) \cap Q \subset X$. The map $\mu$ then induces a regular map 
$\tilde{X} \xrightarrow{\ \tilde{\mu}\ } M'_X$
which is generically a $(k - 1)$-dimensional quadric bundle if $\P(K) \not\subset Q$, or a $\P^{k-1}$-bundle if $\P(K) \subset Q$.
\end{lemm}

\begin{proof}
The first part is evident, since $M'_X = M_X$ when $K = 0$.
For the second part, note that the blow up of the cone $M_X = \cone{K}(M'_X)$ at its vertex $\P(K)$ is a $\P^k$-bundle over $M'_X$.
The blow up of $X$ along $X \cap \P(K) = \P(K) \cap Q(v)$ is the same as the proper preimage of $X$.
In particular, it comes with a regular map $\tilde{X} \to M'_X$.

If $\P(K) \not\subset Q$, the proper preimage coincides with the total preimage, hence its fibers are intersections $\P^k \cap Q$, \ie, generically, they are quadrics of dimension $k-1$.
If $\P(K) \subset Q$,   the general fiber of the total preimage of $X$ is a quadric $Q \cap \P^k$ which contains the hyperplane  $\P(K) = \P^{k-1}$. 
Therefore, the general fiber of the proper preimage is the residual hyperplane.
\end{proof}

Later on, we will give more details for the case $k = 1$ and $\P(K) \not\subset Q$.

In   this article, we are mostly interested in the case of smooth GM varieties.
However, for many questions, the condition $\codim(\Sing(X)) \ge 4$ is almost as good as smoothness.
Keeping in mind possible applications of singular GM varieties, we keep the above assumption as long as possible.}

\begin{lemm}\label{p218} 
Let $(W,V_6,V_5,L,\mu,\bq,\eps)$ be a GM data of dimension $n$. If the associated GM intersection~\eqref{equation-x-intersection-quadrics} is a GM variety $X$ of dimension $n$  with $\codim(\Sing(X)) \ge 4$, we have
\begin{equation*}
\dim (\mu (W)) \ge \min\{7,n+4\}.
\end{equation*}
\end{lemm}

\begin{proof}
As in the proof of Lemma~\ref{lemma:from-data-to-x}, we can write $X$ as in \eqref{equation-x-from-data}  {with} $K=\Ker(\mu )$. Then,
 $\P(K)  $ is contained in the singular locus of $\CKGr(2,V_5) \cap \P(W)$,
hence $\P(K) \cap Q(v)$  is contained in the singular locus of $X$. It has dimension at least $(n+5)-r-2 = n+3-r$, where  $r:=\rank (\mu )$. Thus the condition $\codim(\Sing(X)) \ge 4$ implies that
either $n + 3 - r < 0$, or $n + 3 - r \le n - 4$. The first is equivalent to $r \ge n + 4$ and the second to $r \ge 7$.
 \end{proof}

\begin{prop}\label{lhull}
Let $X$ be a GM variety of dimension $n$  {with $\codim(\Sing(X)) \ge 4$}. 
If $n \ge 3$ or $\Ker(\mu ) \ne 0$, the variety
$M'_X$ is  smooth.
Otherwise \textup{(}if $n \le 2$ and $\Ker(\mu) = 0$\textup{)},  $M'_X $ is equal to $M_X $ and has finitely many rational double points. 
\end{prop}

\begin{proof}
We begin with a  {useful} observation. Let $W_0 \subset \bw2{V_5}$ be a linear subspace. Assume $M' := \Gr(2,V_5) \cap \P(W_0)$
is a dimensionally transverse intersection.  
Viewing  elements of the orthogonal complement $W_0^\perp \subset \bw2{V_5}^\vee$ as skew-symmetric forms on $V_5$, one has   
(see, \eg, \cite[Corollary~1.6]{pv}) 
\begin{equation}\label{esing}
\Sing(M')=\P(W_0)\ \cap\hskip-3mm \bigcup_{\omega\in W_0^\bot\setminus\{0\}}\hskip-4mm \Gr(2,\Ker(\omega )).
\end{equation}
 If $\dim (W_0) = 10$, then $M' = \Gr(2,V_5)$ is smooth. 
If $\dim (W_0) = 9$, either a generator $\omega$ of~$W_0^\bot$ has rank 2 and $\Sing(M')$ is a 2-plane, or else $M'$ is smooth. If $\dim (W_0) = 8$, either
  some  $\omega\in W_0^\bot$ has rank 2, in which case $\Gr(2,\Ker(\omega ))$ is a 2-plane contained in the hyperplane~$\omega^\bot$, whose intersection 
 with $\P(W_0)$ therefore contains a line along which $M'$ is singular, or else $M'$ is smooth.  Finally, if $\dim(W_0) \le 7$, we have $\dim (M') \le 3$. 
 It follows that in all cases,   either $M'$ is smooth or $\codim(\Sing(M')) \le 3$.

Now let $X$ be a GM variety such that  $M'_X$ is singular. By the above observation,  we have $\codim(\Sing(M'_X)) \le 3$. 
Since $X = \cone{\Ker(\mu )}M'_X \cap Q(v )$ for any $v \in V_6 \setminus V_5$ (see Remark~\ref{remark:cone-kernel})
and $\codim(\Sing(X)) \ge 4$,
the only possibility   is that $Q(v )$ does not intersect the cone over~$\Sing(M'_X)$.
As $Q(v )$ is an ample hypersurface, this is only possible if $\Ker(\mu )=0$ and $\Sing(M'_X)$ is finite.
In this case, $M'_X = M_X$ has dimension 2 or 3, hence 
$n\le 2$, so we are in the second case of the proposition.

It remains to show that in that case, the only singularities of $ M_X$ are rational double points.
If $n = 1$,  we have $\dim (W) = 6$ and $M_X = \Gr(2,V_5) \cap \P(W)$ is a normal  surface (it is lci  with finite singular locus). The restriction map 
$H^0(\P(W),\cO_{\P(W)}(1)) \to H^0(M_X,\cO_{M_X}(1))$ is bijective, since its composition with the
injection $ H^0(M_X,\cO_{M_X}(1))\to  H^0(X,\cO_X(1))$ is. It follows that  $M_X$ is linearly normal. 
If it is a cone with vertex $w$, it is contained in $\Gr(2,V_5)\cap \T_{\Gr(2,V_5),w}$ (intersection with a tangent space), 
which is a 4-dimensional cone with vertex $w$ and span $\T_{\Gr(2,V_5),w} \cong \PP^6$. Since $M_X$ spans the 5-dimensional $\P(W)$, 
it is a hyperplane section of this cone, but this  contradicts the fact that $M_X$ is a surface. 
 
 So the quintic surface $M_X$ is not a cone,
hence it is a del Pezzo surface (\cite[Proposition~8.1.6 and Definition~8.1.5]{dol}).  This implies that the singularities of $M_X$ are rational double points (\cite[Theorem 8.1.11]{dol}). 

When $n=2$, the singular points of the threefold $M_X$ are rational double points, since a general hyperplane section through such a point  has this property, by the case $n=1$.
\end{proof}

\begin{rema}
Given a  {general} {\em singular}
 linear section $M'$ of $\Gr(2,V_5)$ of dimension $n+1\ge 4$, one checks that 
a general quadratic section of $M'$   is singular 
in codimension 3 and   its singularities are canonical but not terminal.
 \end{rema}

The following elementary result
 is useful for checking that Grassmannian hulls are smooth dimensionally transverse intersections.

\begin{prop}\label{p12}
Let $W_0 \subset \bw2{V_5}$ be a linear subspace.
Set $M:= \Gr(2,V_5) \cap \P(W_0)$ and  $M^\sharp:= \Gr(2,V_5^\vee) \cap \P(W_0^\bot)$. 
Then $M$ is a smooth dimensionally transverse intersection if and only if $M^\sharp$ has the same property. 
In that case, if moreover both $M$ and $M^\sharp$ are non-empty, one has $\dim (M) + \dim (M^\sharp) = 2$.
\end{prop}

\begin{proof}
The scheme $M$ is a smooth dimensionally transverse intersection at a point $w\in M$ if and only if $\T_{\Gr(2, V_5) , w}$ and $ \P(W_0)$ span $\P(\bw{2}{V_5})$. 
This is not the case exactly when there exists a hyperplane $H\subset \P(\bw{2}{V_5})$ containing both $\T_{\Gr(2, V_5) , w}$ and $ \P(W_0)$.

The pair $(w,H)$ then belongs to the incidence variety of $\Gr(2, V_5)$, whose projective dual is $\Gr(2, V_5^\vee)$. 
It follows that the point $H^\bot\in \P(\bw{2}{V_5^\vee})$   is in $\Gr(2, V_5^\vee)$ and that the hyperplane $w^\bot$ 
contains $\T_{\Gr(2, V_5) , H^\bot}$. Since $ w\in \P(W_0)$, we also have $w^\bot \supset \P(W_0^\bot)$, so that   
$M^\sharp= \Gr(2,V_5^\vee) \cap \P(W_0^\bot)$ is not a smooth dimensionally transverse intersection at $H^\bot $.
\end{proof}


\begin{rema}
Let $X$ be a GM variety  of dimension $n \ge 3$ with $\codim(\Sing(X)) \ge 4$. By Lemma~\ref{p218}, we have $\dim (\mu(W)) \ge 7$. 
By Proposition~\ref{lhull},
$M'_X$ is  smooth of dimension $6-\codim (\mu(W))\ge 3$. 
Proposition~\ref{p12} therefore says that $(M'_X)^\sharp $ is empty (\ie, {the space} $\mu(W)^\perp \subset \bw2{V_5^\vee}$ contains no forms of rank 2). 
\end{rema}

For  {smooth} GM curves, $(M'_X)^\sharp $ is never empty. For  {smooth} GM surfaces, $(M'_X)^\sharp $ may be non-empty. 
We finish this subsection with a discussion of these two cases.

Let $X$ be a  {smooth} GM curve. As we will explain in more details in Section \ref{section-lci-gm}, 
 either $M'_X$ is a smooth elliptic quintic curve in $\P(\mu(W))=\P^4$, there is a double cover $X\to M'_X$ (so that $X$ is bielliptic),   
$M_X$ is the cone over $M'_X$, and $(M'_X)^\sharp$ is isomorphic to $M'_X$,
 or else $M_X=M'_X$ is a   quintic del Pezzo surface  in $\P(\mu(W))=\P^5$.
 
Mukai remarked (\cite[Theorem 5.1.(2)]{mukcg}) that $(M'_X)^\sharp$  has a simple geometric interpretation as the set of all $g^1_4$ on $X$. This allows us to 
 recover a classical result without appealing to the description of del Pezzo quintics as blow ups of the projective plane  (compare with \cite[Proposition 4]{sb}). 
 
\begin{prop}\label{coro224}
Let $X$ be a  {smooth} complex
GM curve. The Grassmannian hull $M_X$ is smooth, \ie,  $X$ is the intersection of a {\em smooth} 
 quintic del Pezzo surface with a quadric, if and only if $X$ has exactly $5$ distinct $g^1_4$.
\end{prop}

\begin{proof} 
As we just explained, $X$ is the intersection of a  smooth
quintic del Pezzo surface with a quadric if and only if $M'_X$ is a smooth (quintic) surface. By Proposition \ref{p12}, this is the case if and only if $(M'_X)^\sharp$ is a smooth scheme of dimension 0. On the other hand,  
 the scheme $W^1_4(X)$ parametrizing the $g^1_4$ on $X$ is either infinite, 
  in which case $X$ is bielliptic, or it has length~5. The proposition therefore follows from Mukai's remark   
 (\cite[Theorem~5.1.(2)]{mukcg}) that there is a bijection between the sets $(M'_X)^\sharp$ and $W^1_4(X)$.
 \end{proof}

 Let $(X,H)$ be a  {smooth} GM surface. As in the case of curves, 
 $M'_X$ may be singular. By Proposition \ref{lhull}, this may only happen when $\mu$ is injective. By Proposition \ref{p12}, $M_X=M'_X$ is then a singular threefold and 
 $(M'_X)^\sharp$ is non-empty. 
This has strong  consequences on $X$. 

More precisely, one can show that $(M'_X)^\sharp $ is non-empty if and only if the Gushel bundle   can be written over $\C$ as an extension  
\begin{equation*}
0 \to \cO_{X_\C}(-h_2) \to \cU_{X_\C} \to \cO_{X_\C}(-h_1) \to 0,
\end{equation*}
where $h_1$ and $h_2$ are divisors on ${X_\C}$ such that
\begin{equation*}
h_1^2 = 0,
\qquad
h_1h_2 = 4,
\qquad
h_2^2 = 2, \qquad h_1 + h_2 = H 
\end{equation*}
(compare with \cite[Lemma 1.5]{glt}). In particular, the   Picard number of ${X_\C}$ must be 
$>1$.

\subsection{Locally complete intersection GM varieties}\label{section-lci-gm}

In the rest of the article, we mostly deal with GM varieties which are lci.
This condition can be   described in terms of GM data.

\begin{defi}\label{def225}
A GM data $(W,V_6,V_5,L,\mu,\bq,\eps)$ is called {\sf lci} if
\begin{itemize}
\item either $\Ker(\mu ) = 0$, in which case   the data is called {\sf ordinary};
\item or $\dim(\Ker(\mu ) ) = 1$ and the point of $\P(W)$ corresponding to the kernel subspace $\Ker(\mu ) \subset W$ 
does not lie on the quadric $Q(v)$ for any $v \in V_6 \setminus V_5$, in which case   the data is called {\sf special}.
\end{itemize}
\end{defi}

\begin{prop}\label{proposition-lci-data}
A  {normal} GM variety   is lci if and only if its canonically associated GM data is lci.
\end{prop}

\begin{proof}
Since $\Gr(2,V_5)$ is not a complete intersection, the non-lci locus of the cone $\CKGr(2,V_5)$ is precisely its vertex $\P(K)$.
Since a GM variety $X$ is a dimensionally transverse intersection~\eqref{equation-x-from-data}, it is lci if and only if
it does not contain points of $\P(K)$, \ie, if and only if $X \cap \P(K) = \P(\Ker(\mu )) \cap Q(v)$ is empty. The latter condition
can be rephrased as $\dim(\Ker(\mu )) \le 1$ and, in   case of   equality, 
  the corresponding point is not in $Q(v)$.
\end{proof}

 {The next lemma will  usually be applied to smooth GM varieties, but as usual we prove it under a weaker assumption.
In this form, it is a partial converse to Theorem \ref{theorem:gushel-mukai}.

\begin{lemm}\label{lemma:pic-gm}
Let $(X,H)$ be an lci polarized complex GM variety of dimension $n \ge 3$. 
If $\codim (\Sing(X) )\ge 4$, the variety $X$ is locally factorial and $\Pic(X)=\Z H$. 
In particular, the polarization $H$ is the unique GM polarization on $X$, hence $\Aut_H(X) = \Aut(X)$.
\end{lemm}

\begin{proof}
By \cite[Exp.~XI, cor.~3.14]{sga2}, $X$ is locally factorial.
It remains to show that $\Pic(X) = \Z H$. 
Assume first $n = 3$, so that $X$ is smooth.
By Theorem~\ref{theorem:gm-intrinsic}, $X$ is a Fano variety of degree~10 and is an intersection of quadrics.
As we already mentioned in the proof of Lemma~\ref{fgmhyp}, the only Fano threefold of degree 10 and Picard number greater than 1 is not an intersection of quadrics.
 This proves the claim for $n = 3$.
When $n>3$, we proceed by induction: a general hyperplane section $X'$ of $X$ satisfies the same properties and restriction induces an injection $\Pic(X)\hra \Pic(X')$ (\cite[Exp.~XII, cor.~3.6]{sga2}). 
\end{proof}
}

For any lci GM variety $X$, the Gushel sheaf $\cU_X$ is locally free {(we will call it {\sf the Gushel bundle}) and the Gushel map $X \to \Gr(2,V_5)$ is regular}.
 Moreover,  

\begin{itemize}
\item either $\mu \colon W \to \bw2{V_5}$ is injective, the Gushel map 
is a closed embedding, we have $M_X = M'_X$, and~\eqref{equation-x-from-data}  can be rewritten 
  as
\begin{equation*}
X = \Gr(2,V_5) \cap \P(W) \cap Q(v),
\end{equation*}
so that $X$ is a quadratic section of $M_X$.
These GM varieties will be called {\sf ordinary}.
For these varieties,  
\begin{equation*}
\dim (X) = \dim (W) - 5 \le \dim( \bw2{V_5} )- 5 = 5.
\end{equation*}
\item or else $\Ker(\mu)$ has dimension 1, the Gushel map induces a double covering $X \thra M'_X$, and $M_X = \cone{}M'_X$. 
These GM varieties will be called {\sf special}. For these varieties, 
\begin{equation*}
\dim (X) = \dim (W) - 5 \le \dim( \bw2{V_5} )+1- 5 = 6.
\end{equation*}
A special GM variety   comes with a canonical involution---the involution of the double covering $X \thra M'_X$. 
\end{itemize}

We make a useful observation. Let $(W,V_6,V_5,L,\mu,\bq,\eps)$ be an lci GM data, set
\begin{equation}
W_1 := \Ker(\mu) \subset W
 ,\qquad
W_0 := W/W_1,
\end{equation}
and let $\mu_0 \colon W_0 \hookrightarrow \bw2 V_5$ be the embedding induced by $\mu$.

\begin{prop}\label{l228}
For any lci GM data $(W,V_6,V_5,L,\mu,\bq,\eps)$, there is a unique splitting of the canonical exact sequence
$0 \to W_1 \to W \to W_0 \to 0$, \ie, a direct sum decomposition
\begin{equation}
W = W_0 \oplus W_1,
\end{equation}
such that the map $\bq$ decomposes as
\begin{equation}
\bq = \bq_0 + \bq_1,
\qquad\text{where}\quad 
\bq_0\colon V_6 \to \Sym^2\!W_0^\vee
\quad\text{and}\quad
\bq_1\colon V_6 \to \Sym^2\!W_1^\vee.
\end{equation}
 \end{prop}

\begin{proof} 
If the data is ordinary, we have $W_1 = 0$, $W_0 = W$ and there is nothing to prove.
If $W_1  $ is non-zero, it is 1-dimensional. Let $w_1 \in W_1$ be any non-zero vector
and choose an arbitrary $v \in V_6\setminus V_5$. Since the data is lci, the quadric $Q(v)$ does not pass through the point $w_1 \in W_1$, hence the linear form $\bq(v)(w_1,-) \in W^\vee$ does not vanish on $w_1$  and   gives a   decomposition
$W = W_0 \oplus W_1$ such that $\bq(v) \in \Sym^2\!W_0^\vee \oplus \Sym^2\!W_1^\vee \subset \Sym^2\!W^\vee$. 
By~\eqref{equation-mu-q-plucker}, we have $\bq(V_5) \subset \Sym^2\!W_0^\vee$, hence the   decomposition of $W$ does not depend on the choice of $v$ and the image of $\bq$ is contained in  $\Sym^2\!W_0^\vee \oplus \Sym^2\!W_1^\vee \subset \Sym^2\!W^\vee$. 
We let $\bq_0$ and $\bq_1$ be the summands of $\bq$ corresponding to this direct sum decomposition.
\end{proof}

\begin{rema}\label{remark-nonlci-decomposition}
One could prove the same result as in Proposition~\ref{l228}, replacing the lci condition by the assumption
that,  for $v \notin V_5$, the restriction to $W_1 := \Ker(\mu) \subset W$ of the quadratic form $\bq(v)$
is non-degenerate (this assumption does not depend on the choice of $v$). The proof is the same---$W_0$ is defined as the orthogonal of $W_1$ in $W$ with respect to $\bq(v)$.
\end{rema}

\begin{coro}\label{corollary-aut-x-gl6}
Let $(X,H)$ be a  {normal} lci polarized   GM variety. The kernel of the canonical morphism  $\Aut_H(X) \to \PGL(V_6)$  is
trivial if $X$ is ordinary, and is generated by the canonical involution of $X$ if $X$ is special.   
 \end{coro}

\begin{proof}
By Corollary~\ref{corollary-aut-x-aut-data}, the group $\Aut_H(X)$ is a quotient of the group
\begin{equation*}
\Aut(W,V_6,V_5,L,\mu,\bq,\eps) \subset \GL(W)\times\GL(V_6) \times \Gm
\end{equation*}
of automorphisms of its associated GM data.
Such an automorphism $g = (g_W,g_V,g_L)$  maps to the identity in $\PGL(V_6)$ if its component $g_V$ is scalar.
Its actions on $V_5$, on $\bw2V_5$, and hence on $W_0 \subset \bw2 V_5$, are then also scalar. For ordinary varieties 
$X \subset \PP(W_0)$,   the action of such a $g$ on $X$ is therefore trivial.

For special varieties, since the direct sum decomposition $W = W_0 \oplus W_1$ is canonical, it is preserved by $g_W$, 
and since $W_1$ is 1-dimensional,  $g_W$ acts by multiplication on each summand, by scalars 
$t_0$ and $t_1$. Then $g$ acts  by multiplication by $t_0^{-2}$ on $\Sym^2\!W_0^\vee$ and by $t_1^{-2}$ on $\Sym^2\!W_1$.
Since the map $\bq\colon V_6 \to \Sym^2\!W_0^\vee \oplus \Sym^2\!W_1^\vee$ commutes with the action of $g$, 
and the action of $g$ on $V_6$ is scalar, we get $t_0^{-2} = t_1^{-2}$. Thus $t_1 = \pm t_0$, hence 
either $g$ acts on $\PP(W)$ identically, or by a reflection with respect to $\PP(W_0)$, \ie, by the canonical
involution of $X$.
 \end{proof}

For a special GM variety $X$ of dimension $n \ge 2$, the branch locus $X'$ of the Gushel map $X \thra M'_X$ is the dimensionally transverse 
intersection of $M'_X$ with a quadric, so it is an ordinary GM variety.
This leads to an important birational operation on the set of all GM varieties: 
interchanging ordinary and special varieties. We formulate this operation on the level of GM data.

\begin{lemm}\label{lemma-ord-spe-data}
If $(W,V_6,V_5,L,\mu,\bq,\eps)$ is a special lci GM data,   $(W_0,V_6,V_5,L,\mu_0,\bq_0,\eps)$ is an ordinary lci GM data. 

Conversely, let $(W_0,V_6,V_5,L,\mu_0,\bq_0,\eps)$ be an ordinary lci GM data. Let~$W_1$ be a $1$-dimensional vector space,
choose an isomorphism $V_6/V_5 \cong \Sym^2\!W_1^\vee$, and let $\bq_1$ be the composition
$V_6 \to V_6/V_5 \to \Sym^2\!W_1^\vee$. Then $(W_0 \oplus W_1,V_6,V_5,L,\mu_0,\bq_0+\bq_1,\eps)$
is a special lci GM data.

If the field $\k$ is quadratically closed, \ie, if $\k = \k^{1/2}$, these operations define a bijection between the sets of isomorphism classes of special and ordinary GM data respectively.
\end{lemm}

\begin{proof}
The first part is obvious. For the second part, note that different choices of an isomorphism $V_6/V_5 \cong \Sym^2\!W_1^\vee$ produce isomorphic GM data 
(via the isomorphism defined by $g_V = \id_V$, $g_W = \id_{W_0} + \sqrt{t}\,\id_{W_1}$, $g_L = \id_L$ for an appropriate $t \in \k$).
\end{proof}

When $\k = \k^{1/2}$, this bijection can be interpreted as a bijection between isomorphism classes of special and ordinary lci GM  {varieties}.
We will denote by $X_\ord$ the ordinary lci GM intersection associated with a given special lci GM  {variety} $X$  
and by $X_\spe$ the special lci GM intersection associated with a given ordinary lci GM  {variety} $X$. 
We define the {\sf opposite $X_\opp$} of an lci GM  {variety} $X$ by 
\begin{equation}\label{eq:opposite-gm}
X_\opp := 
\begin{cases}
X_\ord & \text{if $X$ is special},\\
X_\spe  & \text{if $X$ is ordinary}.
\end{cases}
\end{equation}



\section{GM varieties, Lagrangian data, and EPW sextics}\label{section-epw-sextics}

 {Eisenbud--Popescu--Walter (EPW) sextics are special hypersurfaces of degree~6 in $\P(V_6)$ which can be constructed from  Lagrangian subspaces $A \subset \bw3V_6$.}
The definition and   main properties of EPW sextics can be found in Appendix~\ref{section-epw}. 
A relation between GM varieties and EPW sextics was found in~\cite{iliev-manivel}.
In this section, we develop the approach of Iliev and Manivel and extend their construction
to include both ordinary and special varieties.

\subsection{The discriminant locus
}\label{section-epw-from-fgm}

The easiest way to relate an EPW sextic to a GM variety is via the discriminant locus.
Let $(X,H)$ be a normal polarized GM variety of dimension~$n$ as in Definition~\ref{defigm}. By Theorem \ref{theorem:gm-intrinsic}, 
   the space 
$V_6 $ of quadrics  in $\P(W)$ containing~$X$ is $6$-dimensional. We define $\widetilde\Dis(X)$ as  
the subscheme of $\P(V_6)$ of {\em singular} quadrics containing~$X$. It is either $\P(V_6)$ or 
a hypersurface of degree $\dim(W) = n + 5$, and in the latter case, the multiplicity in $\widetilde\Dis(X)$ 
of the hyperplane $\P(V_5)$ of (restrictions of) Pl\"ucker quadrics is  {at least} the corank of a general such quadric, 
which is at least $\dim(W)-6=n-1$. We define the {\sf discriminant locus} $\Dis(X)$  as follows: 
if $\widetilde\Dis(X) = \P(V_6)$, we set $\Dis(X)=\P(V_6)$; otherwise 
(and we will see in Corollary \ref{coro311} that this is almost always the case), 
we set
\begin{equation}
\Dis(X) :=  \widetilde\Dis(X) -(n-1) \P(V_5).  
\end{equation}
This is a sextic hypersurface   in $\P(V_6)$.

\begin{theo}[Iliev--Manivel]\label{ilma}
Let $X$ be a general \textup{(}smooth and ordinary\textup{)} complex GM variety of dimension $ n \in\{3,4,5\}$. The discriminant locus $\Dis(X)\subset \P(V_6)$ is an EPW sextic.
\end{theo}

The result is proved  in~\cite[Proposition 2.1 and Lemma 2.2]{iliev-manivel} for general GM fivefolds by explicitly
constructing   a Lagrangian   $A \subset \bw{3}{V_6}$ such that $\Dis(X) = Y_A$ (see also \cite[Proposition~2.18]{og2}).
This construction   extends to general GM varieties of   dimensions~$3$, $4$, or~$5$ (\cite[Proposition 2.4]{iliev-manivel}).
In the next section, we present a version of this construction
which works better in families and allows us to treat at the same time both 
ordinary and special GM varieties. In particular, we give in   Proposition~\ref{proposition-kernel-nonplucker}  an extension of Theorem  \ref{ilma}.

For the time being, we examine the discriminant locus of  GM curves and relate the discriminant locus of a GM variety to that of its opposite as defined in~\eqref{eq:opposite-gm}.

\begin{prop}\label{disccurve}
Let $X$ be a  {smooth} GM curve. The discriminant locus $\Dis(X)$ is a reduced sextic hypersurface and 
\begin{itemize}
\item if $X$ is ordinary,  $ \Dis(X)$ is  geometrically integral and normal;
\item if $X$ is special \textup{(}\ie, bielliptic\textup{)},  $ \Dis(X)$ is the union of the Pl\"ucker hyperplane $\P(V_5)$ and a geometrically integral quintic hypersurface.
\end{itemize}
\end{prop}

\begin{proof}We may assume $\k=\C$.
 A local calculation shows that   the Zariski tangent space to $\widetilde\Dis(X)$ at a point corresponding to a quadric  $Q$ containing  $X $ has dimension 
  {greater than~4}
 if and only if
\begin{itemize}
\item[a)] either $Q$ has corank $1$ but its vertex is on $X$;
\item[b)]  or else  $Q$ has corank $\ge 2$  (hence  rank $\le 4$).
\end{itemize}

{\em Assume first that $X$ is  ordinary.} In  case a), $Q$ must  {be a Pl\"ucker quadric,}
because otherwise, $X$ would be equal to $M_X\cap Q$ 
and would be singular at the vertex of $Q$. But \hbox{corank-1} quadrics with vertex at $w\in X$ are in one-to-one correspondence with smooth quadrics containing 
the projection of $M_X$ from $w$. This projection is  {a del Pezzo surface of degree 4, hence is} the base locus of a pencil of quadrics in $\P^4$.  {Therefore,} these quadrics form an open subset in a ruled surface
  with base $X$.

We now examine case b). The locus   of quadrics of rank $\le 4$ containing  $X$ was  studied extensively in \cite{ah}. It contains
 five 2-planes (one for each $g^1_4$ on~$X$, counted with multiplicities;
 \cite[(5.10)]{ah})
  of  Pl\"ucker quadrics and one more irreducible surface which comes from 
the singular points of the theta divisor of the Jacobian of $X$, \ie, from the $g^1_5$ on $X$. 

All in all, this proves that the singular locus of $ \Dis(X)$ has dimension 2. It follows that $ \Dis(X)$ is an integral normal sextic hypersurface.

{\em Assume now that $X$ is  bielliptic.} In case a), $Q$ must again  be a Pl\"ucker quadric.
It is then  also singular   at the vertex of $M_X$, hence cannot be of corank~1. So case a) is impossible.

 Pl\"ucker quadrics
are all singular at the vertex and have general corank~1, hence $\P(V_5)$ is a reduced component of $ \Dis(X)$; write   $\Dis(X)=\P(V_5)+D$, 
where $D$  is a quintic hypersurface not containing $\P(V_5)$. Then $ \Sing( \Dis(X))\cap \P(V_5)=\P(V_5)\cap D$ corresponds to singular quadrics   containing the quintic 
elliptic curve $M'_X$  and this was proved in \cite[p.~170]{ah} to be an integral quintic hypersurface.
 It follows that $D$ is also an integral quintic hypersurface. 
\end{proof}

\begin{lemm}\label{lemma-disc-ord-spe}
Let $X$ be a special lci GM variety with associated ordinary variety $X_\ord$. 
Then $\Dis(X) = \Dis(X_\ord)$.
\end{lemm}

\begin{proof}
Let $(V_6,V_5,L,W,\mu,\bq,\eps)$ be the GM data associated with $X$. 
Let $W = W_0 \oplus W_1$ be the canonical direct sum decomposition of Proposition~\ref{l228}
and let $\bq = \bq_0 + \bq_1$ be the decomposition of $\bq$. By Lemma~\ref{lemma-ord-spe-data}, 
the ordinary GM variety $X_\ord$ is determined by $\bq_0\colon V_6 \to \Sym^2\!W_0^\vee$. In particular,
the rank of $\bq(v)$ is the sum of the ranks of $\bq_0(v)$ and $\bq_1(v)$. By Lemma~\ref{lemma-ord-spe-data},
the map $\bq_1$ vanishes on $V_5$ and induces an isomorphism $V_6/V_5 \cong \Sym^2\!W_1^\vee$. Therefore, 
  the rank of $\bq(v)$ equals the rank of $\bq_0(v)$ for Pl\"ucker quadrics, and increases by~1
for non-Pl\"ucker quadrics, which means that the corank of Pl\"ucker quadrics increases by 1, and the corank
of non-Pl\"ucker quadrics stays the same. This shows the lemma.
\end{proof}

\subsection{GM data and Lagrangian data}\label{section-fgm-from-lagrangian}

In this  section we construct, following~\cite{iliev-manivel},
  a bijection between the set of lci GM data and the set of what we call   (extended) Lagrangian data.

Consider a 6-dimensional vector space $V_6$ and endow the space $\bw3V_6$ with the canonical $\det(V_6)$-valued symplectic form given by   wedge product.

\begin{defi}\label{defi-lagrangian-data-ordinary}
A {\sf Lagrangian data} is a collection $(V_6,V_5,A)$, where 
\begin{itemize}
\item $V_6$ is a $\k$-vector space of dimension $6$,
\item $V_5 \subset V_6$ is  a hyperplane, 
\item $A \subset \bw3V_6$ is a Lagrangian subspace.
\end{itemize}
An isomorphism of Lagrangian data between $(V_6,V_5,A)$ and $(V'_6,V'_5,A')$ is a linear  isomorphism $\varphi\colon V_6 \isomto V'_6$ such that $\varphi(V_5) = V'_5$ and $(\bw3\varphi)(A) = A'$.
\end{defi}

A natural extension of the Iliev--Manivel construction gives a bijection between the set of isomorphism classes of Lagrangian data and the set of isomorphism classes of ordinary GM data.
Using the bijection of Lemma~\ref{lemma-ord-spe-data}, one can also use Lagrangian data to parameterize all special GM data. However, to deal simultaneously with ordinary and special
GM varieties, one needs a generalization of the Iliev--Manivel construction. We suggest such a generalization; it uses   an additional Lagrangian subspace which encodes 
the type of a GM variety.

To lighten the notation (and for forward compatibility), we set
\begin{equation}\label{def-ll}
L := (V_6/V_5)^\vee 
\end{equation} 
and endow the space $\k \oplus L$ with the canonical $L$-valued symplectic form.
Note that the group $\GL(L) \cong \Gm$ acts naturally on the Lagrangian Grassmannian $\LGr(\k \oplus L) \cong \P^1$.

\begin{defi}\label{defi-lagrangian-data-extended}
An {\sf extended Lagrangian data} is a collection $(V_6,V_5,\An,A_1)$, where $(V_6,V_5,\An)$ is a Lagrangian data and
 $A_1 \subset \k \oplus L$ is a  Lagrangian  subspace.

An {\sf  isomorphism of extended Lagrangian data} between $(V_6,V_5,\An,A_1)$ and $(V'_6,V'_5,\An',A'_1)$ is an isomorphism 
between $(V_6,V_5,A)$ and $(V'_6,V'_5,A')$ and an element $t \in \Gm$ such that $t(A_1) = A'_1$.
\end{defi}

The $\GL(L)$-action on  {the line} $\LGr(\k \oplus L) \cong \P^1$ has three orbits.
To simplify the notation, we choose a coordinate on this line such that 
\begin{itemize}
\item the subspace $\k \subset \k \oplus L$ corresponds to the point $\{ \infty \} \in \P^1$;
\item the subspace $L \subset \k \oplus L$ corresponds to the point $\{ 0 \} \in \P^1$.
\end{itemize}
The points $\{0\}$ and $\{\infty\}$ are two of the $\GL(L)$-orbits  and the point $\{1\} \in \P^1$ is in  the third orbit.
 {To simplify the notation, we will sometimes write $A_1 = \infty$, $A_1 = 0$, and $A_1 = 1$ instead of $A_1 = \k$, $A_1 = L$, and $A_1 \ne \k,L$, respectively.}

\begin{theo}\label{theo-bijection}
 {For any field $\k$, there is a functor \textup{(}defined in the proof\textup{)} from the groupoid of lci GM data to the groupoid of Lagrangian data.
It induces a bijection between the set of isomorphism classes of ordinary GM data and the set of isomorphism classes of Lagrangian data.

If $\k$ is quadratically closed, the functor extends to a bijection between the set of isomorphism classes of lci GM data and the set of isomorphism classes of extended Lagrangian data with $A_1 \ne   \infty$.}
\end{theo}

\begin{rema} 
The construction of GM data can  also be applied to Lagrangian data with $A_1 = \infty$; 
however, the GM data obtained from this is not lci, and the lci condition is crucial for the inverse construction, since we need the canonical decomposition $W = W_0 \oplus W_1$
of Proposition~\ref{l228}. 
\end{rema}

Before describing the construction, we need  one more piece of notation.
The canonical projection 
\begin{equation*}
\lambda \colon V_6 \to V_6/V_5 = L^\vee 
\end{equation*}
defines, by the Leibniz rule, for  all $p\ge 1$, maps $\lambda_p \colon \bw{p}{V_6} \to \bw{p-1}{V_5} \otimes  {L^\vee}$. 
They fit into exact sequences  
 \begin{equation*}
 0 \to \bw{p}{V_5} \hookrightarrow \bw{p}{V_6} \xrightarrow{\ \lambda_p\ } \bw{p-1}{V_5} \otimes L^\vee \to 0.
\end{equation*}

\begin{proof}[Proof of Theorem~\textup{\ref{theo-bijection}}]
 {We start by explaining how to construct a GM data from an extended Lagrangian data.
After that, we explain the inverse construction and discuss its functoriality.}

Let $(V_6,V_5,A,A_1)$ be an extended Lagrangian data. 
 We   use  the notation~\eqref{def-ll} and  we choose an arbitrary isomorphism
\begin{equation*}
\eps \colon \det(V_5) \xrightarrow{\ \sim\ } L^2.
\end{equation*}
Consider the   maps 
\begin{alignat*}{3}
&\An && \lhra \bw3 V_6 && \stackrel{\lambda_3}\lthra \bw2 V_5 \otimes L^\vee ,\\
&A_1 && \lhra \k \oplus L && \stackrel{\pr_1}\lthra \k.
\end{alignat*}
Let $W_0$ and $W_1$ be their respective images, 
 so that we have  canonical factorizations
\begin{alignat*}{3}
&\An && \lthra W_0 && \stackrel{\mu_0}\lhra \bw2 V_5 \otimes L^\vee ,\\
&A_1 && \lthra W_1 && \stackrel{\mu_1}\lhra \k 
\end{alignat*}
(defining the maps $\mu_0$ and $\mu_1$). We set
\begin{equation*}
W = W_0 \oplus W_1,
\qquad 
\mu = \mu_0 \oplus \mu_1.
\end{equation*}
We have by definition
\begin{equation*}
\Ker(\An \thra W_0) = \An \cap \bw3 V_5,
\qquad
\Ker(A_1 \thra W_1) = A_1 \cap L.
\end{equation*}
It remains to define the map $\bq\colon V_6 \to \Sym^2\! W^\vee$. For this, we first define
\begin{align*}\label{equation-tilde-q}
&\tilde\bq_0 \colon V_6 \otimes \Sym^2\!\An \to \k,
&&
\tilde\bq_0(v)(\xi_1,\xi_2) := - \eps(\lambda_4(v \wedge \xi_1) \wedge \lambda_3(\xi_2)),\\
&\tilde\bq_1 \colon V_6 \otimes \Sym^2\!A_1\to \k,
&&
\tilde\bq_1(v)(x_1,x_1',x_2,x'_2) := \lambda(v)x_1x'_2.
\end{align*}
It is not immediately clear that $\tilde\bq_0(v)$ and $\tilde\bq_1(v)$ are symmetric in their arguments; 
we will show  later that this follows from the Lagrangian property of $\An$ and $A_1$, so the above definition makes sense.

For each $v \in V_6$ and $i \in \{0,1\}$, the kernel of the quadratic form $\tilde\bq_i(v)$ contains  the kernel of the projections 
$A \twoheadrightarrow W_0$ and $A_1 \twoheadrightarrow W_1$.
Indeed, if $\xi_2 \in \An \cap \bw3V_5$, we have $\lambda_3(\xi_2) = 0$, hence $\tilde\bq_0(v)(\xi_1,\xi_2) = 0$ for any $\xi_1 \in \An$.
Analogously, if $(x_1,x'_1) \in L$, \ie, $x_1 = 0$, then $\tilde\bq_1(v)(x_1,x'_1,x_2,x'_2) = 0$ for any $(x_2,x'_2) \in A_1$.
This means that $\tilde\bq_i(v)$, considered as a map $\Sym^2\!A \to \k$ or $\Sym^2\!A_1 \to \k$, factors through $\Sym^2\! W_i$ 
and thus defines a quadratic form $\bq_i(v) \in \Sym^2\! W_i^\vee$. We finally define
\begin{equation*}
\bq = \bq_0 + \bq_1 \colon V_6 \to \Sym^2\!W_0^\vee \oplus \Sym^2\!W_1^\vee \subset \Sym^2\!W^\vee.
\end{equation*}
Let us show that $(W,V_6,V_5,L,\mu,\bq,\eps)$ is a GM data. It only remains to check
that $\bq$ is symmetric (a fact which we already used a couple of times) and that  the relation~\eqref{equation-mu-q-plucker} holds.

For the symmetry of  $\bq$, we use the Lagrangian-quadratic correspondence (see~Appendix~\ref{section-lagrangians-and-quadrics}).
We set $\symv = \symv_0 \oplus \symv_1 := \bw3V_6 \oplus (\k \oplus L)$, a $\Z/2$-graded   vector space endowed with the $L$-valued symplectic form
\begin{equation}\label{symplectic-form}
\omega((\xi,x,x'),(\eta,y,y')) = \eps(\lambda_6(\xi\wedge\eta)) + yx' - xy'.
\end{equation}
Take   $v \notin V_5$, so that $\lambda(v) \ne 0$, and consider the   $\Z/2$-graded Lagrangian direct sum decomposition
\begin{equation}\label{dsd}
\symv = 
\bigl( \bw3{V_5}  \oplus L \bigr) \oplus \bigl((v\wedge\bw2{V_5}) \oplus \k \bigr).
\end{equation}
Since $\hat{A} := \An \oplus A_1 \subset \symv$ is   another $\Z/2$-graded Lagrangian subspace, 
the corresponding quadratic form {(as defined in Proposition~\ref{lemma-lagrangian-quadratic})} evaluated on elements $(\xi_1,x_1,x'_1),(\xi_2,x_2,x'_2) \in \hat{A}$ is
 $\omega(\pr_1(\xi_1,x_1,x'_1),\pr_2(\xi_2,x_2,x'_2))$ (its symmetry follows   from the Lagrangian property of~$\hat{A}$;
see the proof of Proposition~\ref{lemma-lagrangian-quadratic}). For the decomposition~\eqref{dsd}, the projections are given by
\begin{equation*}
\begin{array}{rcll}
\pr_1(\xi_1,x_1,x'_1) & = &(\lambda(v)^{-1}\lambda_4(v \wedge \xi_1),x'_1) & \in (\bw3{V_5}) \oplus L,\\
\pr_2(\xi_2,x_2,x'_2) & =& (\lambda(v)^{-1}v \wedge \lambda_3(\xi_2),x_2) & \in (v\wedge \bw2{V_5}) \oplus \k.
\end{array}
\end{equation*}
Substituting these into~\eqref{symplectic-form}, 
we obtain the form $\tilde\bq_0 + \tilde\bq_1$ up to a rescaling by~$\lambda(v)$.
It follows that for $v \notin V_5$, the forms $\tilde\bq_i(v)$  and $\bq_i(v)$  are symmetric. By continuity, the same
is true for $v \in V_5$. On the other hand, for $v \in V_5$ (so that $\lambda(v) = 0$), we have   
$\tilde\bq_1(v) = 0$ and 
\begin{equation*}
\tilde\bq_0(v)(\xi_1,\xi_2) = 
-\eps(\lambda_4(v\wedge\xi_1)\wedge\lambda_3(\xi_2)) =
\eps(v\wedge \lambda_3(\xi_1)\wedge\lambda_3(\xi_2)).
\end{equation*}
By definition of $\mu$, this implies  
\begin{equation*}
\bq(v)(w_1,w_2) = \bq_0(v)(w_1,w_2) = \eps(v \wedge \mu_0(w_1) \wedge \mu_0(w_2)),
\end{equation*}
which proves~\eqref{equation-mu-q-plucker} (and gives   another proof of the symmetry of $\bq(v)$ for $v \in V_5$).

Let us  show also that this GM data is lci. We have $W_1 = 0$ if $A_1 = L$, and   $W_1 = A_1$ otherwise.
But the quadratic form $\tilde\bq_1$ on $A_1$ is non-trivial unless $A_1 = L$ or $A_1 = \k$. So the only case
where $W_1 \ne 0$ and $\bq_1 = 0$ simultaneously (the non-lci case) is the case $A_1 = \k$, which is excluded from our consideration.

Finally,   if we rescale $\eps$ by $t \in \k^\times$, the map $\bq_0$   will be also rescaled by $t$
 while the other data will not change. 
But the action of the element $(\varphi_V,\varphi_W,\varphi_L) \in \GL(V_6) \times \GL(W) \times \GL(L)$ defined by
 $\varphi_V = t\,\id_V$,
$\varphi_W = \id_{W_0} + \sqrt{t}\,\id_{W_1}$,
$\varphi_L = t^2\id_L$
 precisely realizes such a rescaling (this is the only place where we use the assumption  {$\k   =  \k ^{1/2}$};
note however that it is unnecessary if $W_1 = 0$, \ie, for ordinary GM data).
This means that different choices of $\eps$ produce  {isomorphic} GM data.

We now explain the inverse construction. Let $(W,V_6,V_5,L,\mu,\bq,\eps)$ be a GM data.
Let $W = W_0 \oplus W_1$ and $\bq = \bq_0 + \bq_1$ be the canonical direct sum decompositions of Proposition~\ref{l228}.
Choose an arbitrary embedding  
\begin{equation*}
\mu_1 \colon W_1 \hookrightarrow \k ,
\end{equation*}
  consider the maps 
\begin{equation}\label{equation-ax-monad}
{\xymatrix@R=6mm
@C=4mm{
V_5 \otimes W \otimes L \ar[rr]^-{f_1} && \bw3{V_5}  \oplus L \oplus (V_6\otimes W \otimes L) \ar[rr]^-{f_2} \ar[d]^{f_3} && W^\vee \otimes L \\
&& \bw3 V_6 \oplus (\k \oplus L) = \symv 
}}
\end{equation}
defined by
\begin{eqnarray*}
 f_1(v\otimes w \otimes l) &=& (-v \wedge \mu_0(l \otimes w), 0, v\otimes w \otimes l),\\
 f_2(\xi,x',v\otimes w \otimes l)(w') &=& \eps(\xi \wedge \mu_0(w')) + \mu_1(w') \otimes x' + \bq(v)(w,w') \otimes l,\\
 f_3(\xi,x',v\otimes w \otimes l) & =& (\xi + v\wedge \mu_0(l \otimes w), l\otimes \lambda(v)\mu_1(w), x'),
\end{eqnarray*}
(note that $\lambda(v) \in V_6/V_5 = L^\vee$, so $l \otimes \lambda(v) \in \k$), and define
\begin{equation*}
\hat{A} := f_3(\Ker(f_2)) \subset \symv.
\end{equation*}
We will check below that $\hat{A} \subset \symv$ is a graded Lagrangian subspace.

 We introduce   gradings on   the terms of~\eqref{equation-ax-monad} as follows. On the leftmost and the rightmost terms,
the grading is induced by the direct sum decomposition $W =W_0 \oplus W_1$; in the middle column,  
$\bw3V_5 \oplus (V_6\otimes W_0\otimes L)$ and $\bw3V_6$  are the even parts, while 
$L \oplus (V_6 \otimes W_1 \otimes L)$ and~$\k \oplus L$ are the odd parts. Then the gradings are preserved by all the maps.
In particular, $\hat{A}$ is a direct sum $\hat{A} = \An \oplus A_1$ with $\An \subset \symv_0 = \bw3V_6$ and $A_1 \subset \symv_1 = \k \oplus L$.

 Note that $f_2\circ f_1 = 0$ and $f_3 \circ f_1 = 0$. The second equality is obvious and   the first follows from 
$\bq(v)(w,w') = \eps(v\wedge\mu(w)\wedge\mu(w'))$
for all $v \in V_5$, $w,w' \in W$,  a reformulation of~\eqref{equation-mu-q-plucker}.
These equalities imply   
\begin{equation*}
\hat{A} = f_3\bigl(\Ker(\Coker(f_1) \xrightarrow{\ f_2\ } W^\vee)\bigr) .
\end{equation*}
Since the third component of $f_1$ is the natural embedding $V_5\otimes W \otimes L \to V_6 \otimes W \otimes L$, choosing a vector $v_0 \in V_6\setminus V_5$
splits this map, and we can rewrite $\hat{A}$ as the vector space fitting into the exact sequence
\begin{equation}\label{eq:ax-kernel}
0 \to\hat{A} \to (\bw3{V_5} \oplus L) \oplus (v_0\otimes W \otimes L) \to W^\vee \otimes L \to 0,
\end{equation}
which coincides with the exact sequence \eqref{eqlem},  {written for the Lagrangian decomposition~\eqref{dsd}}. 
Moreover, the embedding of $\hat{A}$ into $\symv$ induced by $f_3$
coincides with the embedding discussed in the  line after~\eqref{eqlem}.
So Lemma~\ref{lemma-quadric-lagrangian} proves that $\hat{A}$ is Lagrangian.

Finally, we   check how   the constructed extended Lagrangian data depends on the choice of~$\mu_1$.
The even part (with respect to the $\Z/2$-grading) of~\eqref{equation-ax-monad} does not depend on this choice, hence the same is true for $\An$. 
 {This shows that the constructed Lagrangian data $(V_6,V_5,A)$ does not depend on choices and is functorial (an isomorphism of GM data induces an isomorphism of the associated Lagrangian data).}

On the other hand, a simple 
computation shows that $A_1 = L$ (\ie, corresponds to the point $0 \in \P^1$) if and only if $W_1 = 0$; 
moreover, $A_1 = \k$ (\ie, corresponds to the point $\infty \in \P^1$) if and only if $W_1 \ne 0$ and $\bq_1 = 0$ 
(which is the non-lci case); otherwise, $A_1$ corresponds to a point of $\P^1 \setminus \{0,\infty\}$. 
From this, it is clear that different choices of $\mu_1$ produce isomorphic extended Lagrangian data.

The two constructions we explained are mutually inverse---this follows from Lemma~\ref{lemma-quadric-lagrangian}.
This proves the existence of a bijection (if the field $\k$ is  {quadratically} closed) between the sets of isomorphism classes of lci GM data and extended Lagrangian data.

It remains to show that for any field $\k$, we have a bijection between the sets of isomorphism classes of ordinary GM data and Lagrangian data.
For this, note that the conditions $A_1 = L$ and $W_1 = 0$ are equivalent {and that the map from the set of isomorphism classes of Lagrangian data to the set of isomorphism classes of GM data is well defined for any field~$\k$.
The latter follows from the observation that}
when $W_1 = 0$, a rescaling of $\varepsilon$ by $t \in \k^\times$ can be realized by an automorphism $\varphi_V = t$, $\varphi_W = 1$, and $\varphi_L = t^2$ of GM data, 
 {hence the isomorphism class of the obtained GM data does not depend on the choice of $\varepsilon$.}
 \end{proof}

 {\begin{rema}
The functor between the groupoid of ordinary GM data and the groupoid of ordinary GM data defined in Theorem~\ref{theo-bijection} is not an equivalence.
First, it takes an automorphism $(g_V,g_W,g_L) = (1,-1,-1)$ to the identity automorphism of the associated Lagrangian data  hence is not faithful;
second, to lift an automorphism $g_V$ of the Lagrangian data associated with a given ordinary GM data $(W,V_6,V_5,L,\mu,\bq,\varepsilon)$ to an automorphism of that GM data,
one has to extract a square root of $\det(g_V\vert_{V_5})$ (to define $g_L$).
\end{rema}}

\begin{rema}
One can generalize Theorem \ref{theo-bijection} by weakening the lci assumption for GM data as in Remark~\ref{remark-nonlci-decomposition}, but one then has 
to modify  further the definition of Lagrangian data. To be more precise, fix a vector space $K$,
consider the symplectic space $\symv_1(K) := K \oplus (K^\vee\otimes L)$ with the canonical $L$-valued symplectic form
(with the  {isomorphism} relation for Lagrangian subspaces $A_1 \subset \symv_1(K)$ induced by the natural action 
of the group $\GL(K)$ on $\symv_1(K)$) and define $K$-Lagrangian data as quadruples $(V_6,V_5,\An,A_1)$, 
where $\An \in \LGr(\symv_0)$ and $A_1 \in \LGr(\symv_1(K))$. Then, there is a bijection between 
 the set of   {isomorphism} classes of $K$-Lagrangian data such that $A_1 \cap K = 0$ and  the set of  {isomorphism} classes
of GM data such that   $\bq(v)\vert_{\Ker(\mu)}$ is   non-degenerate for (all) $v \notin V_5$, and   $\dim(\Ker(\mu)) \le \dim (K)$.
\end{rema}

\subsection{GM varieties and Lagrangian data}

As explained in the introduction, the constructions of Theorems~\ref{theorem:gm-var-data} and~\ref{theo-bijection} can be combined as follows.

\begin{theo}\label{theorem:merged}
With each normal lci polarized GM variety $(X,H)$, one can associate  a canonical \textup{(}functorially depending on $(X,H)$\textup{)}  
Lagrangian data $(V_6,V_5,A)$, and also an extended Lagrangian data $(V_6,V_5,A,A_1)$ with $A_1$ defined up to a $\Gm$-action.

Conversely, with any Lagrangian data, one can associate a canonical \textup{(}functorially depending on the data\textup{)} ordinary GM intersection, 
and also, if the field $\k$ is quadratically closed,  a special GM intersection defined up to an isomorphism.

These two constructions are mutually inverse.
\end{theo}

\begin{proof}
Starting with $(X,H)$, we  apply Theorem~\ref{theorem:gm-var-data} to obtain a GM data, and    Theorem~\ref{theo-bijection} 
to obtain  {a} Lagrangian data $(V_6,V_5,A)$ which depends functorially on $(X,H)$. Moreover, we saw in the proof of Theorem~\ref{theo-bijection} that the orbit of~$A_1$ only depends on the type of~$X$.

Conversely, assume a Lagrangian data is given. 
We apply Theorem~\ref{theo-bijection} and obtain a GM data.
The choice of $\varepsilon$ involved in the construction affects~$\bq_0$, the even part of the family of quadrics cutting out $X$ in $\P(W)$, by scalar multiplication.
In particular, the ordinary GM intersection associated with this data does not depend on this choice.
However, the associated special GM intersection is uniquely determined if the field is quadratically closed, but only up to isomorphism.
For functoriality, note that given an isomorphism of Lagrangian data, one can choose $\varepsilon$ and $\varepsilon'$ in a compatible way;   
the construction of Theorem~\ref{theo-bijection}   then provides an isomorphism of the corresponding GM data, which by Theorem~\ref{theorem:gm-var-data}   gives an isomorphism of the corresponding GM varieties.

It is clear that the constructions are mutually inverse.
\end{proof}

Using this, one can give a criterion for   GM varieties to be isomorphic and   describe the automorphism group of a GM variety in terms of the associated Lagrangian data.
 Consider the group 
\begin{equation}\label{pgl-av}
\PGL(V_6)_{\An,V_5} := \{ g \in \PGL(V_6) \mid  (\bw3g)(\An) = \An   ,\  g(V_5) = V_5\}
\end{equation}
of automorphisms of $\P(V_6)$ stabilizing $\An$ and $V_5$.

\begin{coro}\label{corollary:isom-gm}
Let $(X,H)$  and $(X',H')$ be normal lci polarized GM varieties,   with  corresponding Lagrangian data
$(V_6,V_5,\An)$ and $(V_6',V_5',\An')$. 
\begin{itemize}
\item[\rm (a)]
Any isomorphism $\phi\colon (X,H) \isomto (X',H')$ induces an isomorphism 
$V_6  \isomto V_6'$ 
which takes     $V_5 $ to $V_5'$ and
${\An}$ to~${\An'}$. 
\item[\rm (b)]
Conversely, 
if either $X$ and $X'$ are both ordinary, or they are both special and $\k$ is quadratically closed,
every isomorphism 
$V_6  \isomto V_6'$ that takes $V_5 $ to $V_5'$ and
${\An}$ to~${\An'}$ is induced
by an isomorphism between $(X,H)$ and $(X',H')$. 
\item[\rm (c${}'$)]
If $(X,H)$ is ordinary, $\Aut_H(X) \isom \PGL(V_6)_{\An,V_5}$.
\item[\rm (c${}''$)]
If $(X,H)$ is special, there is an exact sequence
\begin{equation*}
0 \to \Z/2 \to \Aut_H(X) \to \PGL(V_6)_{\An,V_5} \to 1.
\end{equation*}
\end{itemize}
\end{coro}
\begin{proof}
Part (a) follows from Theorem~\ref{theorem:merged} and part (b) is also explained in the proof of this theorem.
 This implies that the image of the morphism $\Aut_H(X) \to \PGL(V_6)$ is $\PGL(V_6)_{\An,V_5}$  and (c${}'$) and  (c${}''$) then follow from
Corollary~\ref{corollary-aut-x-gl6}.
\end{proof}

When we use Theorem~\ref{theorem:merged} for a normal lci (polarized) GM variety~$X$, we denote by $W$, $V_6$, $V_5$, $\An$, $A_1$, etc.,  the associated vector spaces;
to emphasize the dependence of the data on the original variety, we sometimes write $W(X)$, $V_6(X)$, $V_5(X)$, $\An(X)$, $A_1(X) $, etc.
 
Conversely, given an extended Lagrangian data $(V_6,V_5,\An,A_1)$, we denote by $X_{\An,A_1,V_5}$
the corresponding GM intersection (usually the vector space $V_6$ will be fixed, so we exclude it from the notation). 
Sometimes, we will write $X_{A,0,V_5}$ for the GM intersection corresponding to~$A_1 = L$, 
$X_{A,1,V_5}$ for the GM intersection corresponding to the choice $A_1 \ne L$ and $A_1 \ne\k$, and 
$X_{A,\infty,V_5}$ for the GM intersection corresponding to $A_1 = \k$. 
The meaning of this notation is explained by the following lemma.

\begin{lemm}\label{lemma-dimx-ax0}
Let $(V_6,V_5,\An)$ be a Lagrangian data  such that $X := X_{\An,0,V_5}$, with the choice  $A_1 = L$, is a  GM variety.
Then $X$ is ordinary, $X_{\An,1,V_5}  {\cong X_\spe}$ is the double covering of~$M_X$ branched along~$X$ (hence is a special GM variety), 
and $X_{\An,\infty,V_5}$ is the cone over $X$ with vertex a point.
\end{lemm}

\begin{proof} 
If $A_1 = L$, we have $W_1 = \Im(A_1 \to \k) = 0$, hence $X$ is ordinary.
If $A_1 \ne L $ and $A_1 \ne  \k$, the space $W_1$ is 1-dimensional and $\bq_1$ induces an isomorphism $V_6/V_5 \cong \Sym^2W_1^\vee$.
By Lemma~\ref{lemma-ord-spe-data}, we have $X_{\An,1,V_5} = X_\spe$.
 
Finally, if $A_1 = \k$, the space $W_1$ is $1$-dimensional  but $\bq_1 = 0$. 
Therefore, $X_{\An,\infty,V_5}$ is cut out in the cone over~$M_X$ by one equation $\bq_0(v) = 0$ (where $v$ is any vector in $V_6 \setminus V_5$) hence is the cone over $X$.
\end{proof}

The construction of a GM intersection $X_{\An,A_1,V_5}$ 
is quite involved. However, some of its geometrical properties   can be read off directly from $\An$, $A_1$, and $V_5$.
The next proposition explains this.

\begin{prop}\label{lemma-wx}
Let $X = X_{\An,A_1,V_5}$ be the GM intersection associated with an lci  Lagrangian data $(V_6,V_5,\An,A_1)$ \textup{(}in particular, $A_1 \ne \infty$\textup{)}. 
\begin{itemize}
\item[\rm (a)]
We have
\begin{align*}
W(X)_0 &= \An/(\An \cap \bw3V_5), &
W(X)_1 &= A_1/(A_1 \cap L), \\
W(X)_0^\perp &= \An \cap \bw3V_5, &
W(X)_1^\perp &= A_1 \cap L.
\end{align*} 
\item[\rm (b)]
For $v \in V_6\setminus V_5$, we have
\begin{equation}\label{kernel-qv}
 \Ker (\bq(v)) = \An \cap (v \wedge \bw2V_5).
\end{equation}
\item[\rm (c)]
If $X$ is a GM variety, we have
\begin{equation}\label{eq22}
\dim (X) = 
\begin{cases}
5 - \dim(\An \cap \bw3V_5),  & \text{if } {A_1 = 0},\\
6 - \dim(\An \cap \bw3V_5),  & \text{if } {A_1 = 1}.
\end{cases}
\end{equation}
\end{itemize}
 \end{prop}

\begin{proof}
(a) The first line     is just the definition of $W_0$ and $W_1$ in the proof of Theorem~\ref{theo-bijection}. 
For the second line, note that $W(X) = W(X)_0 \oplus W(X)_1$ is the image of the Lagrangian subspace $\hat{A} = \An \oplus A_1$ 
by the projection to the second summand of the   direct sum decomposition~\eqref{dsd}. Hence its annihilator is 
the intersection of $\hat{A}$ with the first summand.

(b) By Proposition~\ref{lemma-lagrangian-quadratic}, the kernel of the quadratic form $\tilde\bq(v)$ on $\hat{A}$ is 
the direct sum $\bigl(\hat{A} \cap (\bw3V_5 \oplus L)\bigr) \oplus \bigl(\hat{A} \cap ((v \wedge \bw2V_5) \oplus \k)\bigr)$. The quadratic form $\bq(v)$ is induced by 
taking the quotient with respect to the first summand, hence its kernel is the second summand.
Since  $A_1 \ne \k$ by the lci assumption, it is equal to the right   side of~\eqref{kernel-qv}. 

(c) If $X$ is a GM variety, we have $\dim(X) = \dim(W(X)) - 5$.
On the other hand, by part~(a),
we have
\begin{equation*}
\dim(W(X)_0) = 10 - \dim(\An \cap \bw3V_5),
\qquad 
\dim(W(X)_1) = 1 - \dim(A_1 \cap L).
\end{equation*}
Combining these, we get \eqref{eq22}. 
\end{proof}

In the next proposition, we discuss the relation between the Lagrangian data of a GM variety $X$ and that of a hyperplane section.
Such a hyperplane section is given by a linear function on the space $W(X) = W(X)_0 \oplus W(X)_1$. Moreover, by Proposition~\ref{lemma-wx}(a),
we have an identification $W(X)_0^\vee = \bw3V_5/(A \cap \bw3V_5)$.
Any linear function on $W(X)_0$ can therefore be lifted to an element of $\bw3V_5$. 

\begin{prop}\label{proposition:hyperplane-section}
Let $X$ be a normal lci GM variety and let $(V_6,V_5,A)$ be the corresponding Lagrangian data.
Let $X' \subset X$ be a hyperplane section of $X$ which is also normal and lci and let $\eta_0 \in \bw3V_5$ be a lift of the even part of the equation of $X'$.
The Lagrangian data of $X'$ is then isomorphic to $(V_6,V_5,A(X'))$ and
\begin{itemize}
\item[\rm(a)] if $X'$ has the same type as $X$, we have $A(X') = (A \cap \eta_0^\perp) \oplus \k \eta_0$; 
\item[\rm(b)] if $X$ is special and $X'$ is ordinary, we have $A(X') = (A \cap \eta^\perp) \oplus \k \eta$ for some $\eta \in A \oplus \k\eta_0$;
\item[\rm(c)] if $X$ is special and $X' = X_\ord$, we have $A(X') = A$.
\end{itemize}
In particular, $\dim(A \cap A(X')) = 9$, unless $X$ is special and $X' = X_\ord$.
\end{prop}

\begin{proof}
Let us show   $A \cap \eta_0^\perp \subset A(X')$. Consider the  commutative diagram
\begin{equation*}
\xymatrix@C=2em{
0 \ar[r] & 
A(X') \ar[r] &
\bw3V_5 \oplus W(X')_0 \ar[rr]^-{(\mu_0,\bq(v_0))} \ar@{^{(}->}[d]^\alpha && 
W(X')_0^\vee \ar[r] & 
0 
\\
0 \ar[r] & 
A \ar[r] &
\bw3V_5 \oplus W(X)_0 \ar[rr]^-{(\mu_0,\bq(v_0))} && 
W(X)_0^\vee \ar[r] \ar@{->>}[u] & 
0,
}
\end{equation*}
where the rows are the exact sequences~\eqref{eq:ax-kernel} defining $A$ and $A(X')$ respectively, $v_0 \in V_6 \setminus V_5$ is a fixed vector,
and we omit for simplicity the $L$-factors. The image of $A \cap \eta_0^\perp$ under the left bottom arrow
is contained in $\bw3V_5 \oplus (W(X)_0 \cap \eta_0^\perp)$, hence in the image of $\alpha$.
Therefore, by the commutativity of the diagram, 
it is in the kernel of the top right arrow. This proves   $A \cap \eta_0^\perp \subset A(X')$  and in particular $\dim(A(X') \cap A) \ge 9$.

Any Lagrangian subspace $A'$ 
 {containing $A \cap \eta_0^\perp$ can be written as $(A \cap \eta^\perp) \oplus \k \eta$ for some $\eta \in A \oplus \k\eta_0$. This proves (b).
Furthermore, if $X'$ has the same type as $X$, the odd part of the equation of $X'$ is zero, hence $\eta_0 \in W(X')_0^\perp = A(X') \cap \bw3V_5 \subset A(X')$.
Thus $\eta_0 \subset A(X')$, hence $\eta = \eta_0$  and (a) follows. Finally, (c) was already proved in Lemma~\ref{lemma-dimx-ax0}.}
\end{proof}

\subsection{Strongly smooth GM varieties}

We say that a subspace $A \subset \bw3V_6$ {\sf contains no decomposable vectors} if (see Section \ref{section-epw-overview})
\begin{equation*}
\P(A) \cap \Gr(3,V_6) = \emptyset.
\end{equation*}
The crucial Theorem~\ref{theorem-singdec} below shows among other things that when a Lagrangian subspace~$\An$ contains no decomposable vectors, 
the GM intersections $X_{\An,A_1,V_5}$ defined at the end of Section~\ref{section-fgm-from-lagrangian} are smooth GM varieties 
for all choices of $A_1 \ne \k$ and all choices of hyperplanes~$V_5 \subset V_6$.

Recall that for any GM intersection~$X$, we  {defined}
$M_X  = \cone{K}\Gr(2,V_5) \cap  \P(W)$ and    $M'_X = \Gr(2,V_5) \cap \P(W_0)$ (Section~\ref{subsec-grhulls}).

\begin{defi}\label{definition-strongly-smooth}
A GM variety  $X$ is   {\sf strongly smooth} if both intersections $X$ and $M'_X$ are   dimensionally transverse and smooth.
\end{defi}

By Proposition~\ref{lhull},
any smooth GM variety is   strongly smooth except perhaps if it is an ordinary GM surface or an ordinary GM curve.
Note also that $X$ is strongly smooth if and only if $X_\opp$ is strongly smooth (unless $X$ is a special GM curve, in which case 
  $X_\opp = X_\ord$ is just not defined).

\begin{theo}\label{theorem-singdec} 
Assume $\k=\C$.
Let $(V_6,V_5,\An,A_1)$ be an extended Lagrangian data with  {$A_1 \ne \infty$} and let $X:= X_{\An,A_1,V_5}$ be the corresponding GM intersection.
The following conditions are equivalent:
\begin{enumerate}
\item[\rm (i)]
the GM intersection $X $ is a strongly smooth GM variety;
\item[\rm (ii)]
we have  $\dim  ( A  \cap  \bw3{V_5}) + \dim(A_1 \cap L) \le 5$, 
the decomposable vectors in $\An$ are all in~$\bw{3}{V_5}$, and  $\P(\An\cap \bw{3}{V_5})\cap \Gr(3,V_5)$     is a smooth dimensionally transverse  intersection in  $\P( \bw{3}{V_5})$.
\end{enumerate}
 In particular, if  {$\dim(A \cap \bw3V_5) \le 3$, the GM intersection $X $ is a smooth GM variety}
  if and only if $\An$ contains no decomposable vectors. 
\end{theo}

\begin{proof}
Assume that condition (ii) holds. By Proposition~\ref{lemma-wx}, we   have   $W_0^\perp = \An \cap \bw{3}{V_5}$, and 
$\P(W_0^\perp) \cap \Gr(3,V_5)$ is   a smooth dimensionally transverse intersection. By Proposition~\ref{p12},  so is $M'_X $. 
Assume now, by contradiction, that  for $v\in V_6 \setminus V_5$, the GM intersection $X_{\An,A_1,V_5} = \cone{\Ker(\mu )} M'_X \cap Q(v)$  is not dimensionally transverse or is singular.
In both cases, there exists $w \in \cone{\Ker(\mu )} M'_X \cap Q(v)$ such that the tangent space $\T_{Q(v),w}$ contains~$\T_{M'_X,w}$.

As $A_1 \ne \k$, we know by the proof of Lemma~\ref{lemma-dimx-ax0} that $w$ cannot be the vertex of the cone. Since $M'_X$ is smooth, $\T_{M'_X,w}$ is therefore the intersection of the tangent spaces to Pl\"ucker quadrics, 
hence $\T_{Q(v),w}$ coincides with $\T_{Q(v'),w}$ for some $v' \in V_5$. Therefore, for some $t \in \k$, the quadric $Q(v+tv')$ is singular at $w$ and, replacing $v$ with $v + tv'$, we may assume that $w$ is in the kernel of $\bq(v)$. By~\eqref{kernel-qv},
this kernel is  $\An \cap (v\wedge \bw2{V_5})$; 
 therefore,  $v \wedge \mu (w) \in \An$.
Since $\mu (w) \in M'_X$, we have $\mu (w) \in \Gr(2,V_5)$, so $v\wedge\mu (w)$ is a decomposable vector in $\An$ which is not in $\bw{3}{V_5}$,  contradicting (ii). 
Therefore, $X_{\An,A_1,V_5}$ is a smooth GM variety (its dimension is $\ge 1$ because of the condition  {$\dim  ( A  \cap  \bw3{V_5}) + \dim(A_1 \cap L) \le 5$}).

Conversely, assume that condition (i) holds. Since $\dim (X) \ge 1$, 
we have  $\dim  ( A  \cap  \bw3{V_5}) + \dim(A_1 \cap L) \le 5$ by Proposition~\ref{lemma-wx}. Assume, again by contradiction, that $\An$ contains a decomposable vector 
not in $\bw{3}{V_5}$. This vector can be written as $v\wedge v_1\wedge v_2$, where $v \in V_6\setminus V_5$ and $v_1,v_2 \in V_5$.
As this vector is both in $\An$ and in $v \wedge \bw2{V_5}$, 
it is by~\eqref{kernel-qv} in the kernel of~$\bq(v)$.
On the other hand, $v_1\wedge v_2 \in \Gr(2,V_5)$ and, as $\lambda_3(v\wedge v_1\wedge v_2) = \lambda(v)v_1 \wedge v_2$, it is also in~$\P(W_0)$. 
So $v_1 \wedge v_2$ is in $M'_X$ and is a singular point of  {the quadric}~$Q(v)$.
Hence it is a singular point of $X = \cone{\Ker(\mu)}M'_X \cap Q(v)$, contradicting~(i).
 Finally, $\P(\An \cap \bw3V_5) \cap \Gr(3,V_5)$ is a smooth dimensionally transverse intersection by Proposition~\ref{p12},  {since} $M'_X = \P((\An \cap \bw3V_5)^\perp) \cap \Gr(2,V_5)$ is such by definition of strong smoothness. 

When $\dim (X_{\An,A_1,V_5}) \ge 3$, smoothness is equivalent to strong smoothness by Proposition~\ref{lhull} 
and $\dim(\An \cap \bw3V_5) \le 3$ by~\eqref{eq22}, hence the intersection $\P(\An\cap\bw3V_5) \cap \Gr(3,V_5)$
is dimensionally transverse if and only if it is empty. Condition (ii) is therefore equivalent to the absence of decomposable vectors in $\An$. 
\end{proof}

\begin{rema}\label{rem39}
Let $X$ be a  strongly smooth complex GM variety, with associated Lagrangian subspace $\An   = \An(X)$.
{Set $\ell:=\dim (\An \cap \bw{3}{V_5})$.}
 Formula~\eqref{eq22} then  reads
\begin{equation*}
\ell = 
\begin{cases} 
5 - \dim(X) &\text{ if $X$ is ordinary;}\\ 
6 - \dim(X) &\text{ if $X$ is special.}
\end{cases}
\end{equation*}
 Theorem \ref{theorem-singdec} implies the following. 
\begin{itemize}
\item 
If $\ell\le 3$
{(\ie, $\dim(X) \ge 3$ or $X$ is an ordinary GM surface)}, $\An$ contains no decomposable vectors. 
By Theorem~\ref{theorem-ogrady-stratification}, 
 the condition $\dim (\An\cap \bw{3}{V'_5})\le3$  holds for all hyperplanes $V'_5\subset V_6$.
\item 
If $\ell=4$
{(\ie, $X$ is an ordinary GM curve or a special GM surface)}, $\An$ contains exactly~5 decomposable vectors  and they   are all in  $\P(\bw{3}{V_5})$. 
\item 
If $\ell=5$ 
 {(\ie, $X$ is a special (bielliptic) curve)}, 
the decomposable vectors in $\P(\An)$ form a smooth elliptic quintic curve contained in $\P(\bw{3}{V_5})$.  
\end{itemize}
\end{rema}

\subsection{EPW sextics}

For a GM variety $X$, we defined in Section~\ref{section-epw-from-fgm} the discriminant locus $\Dis(X) \subset \P(V_6)$, which   is     either a sextic hypersurface or the whole space.
We now give a unified proof of a generalization of  Theorem~\ref{ilma} of Iliev and Manivel, proving that the  schemes $\Dis(X)$ and $Y_{A(X)}$ are equal.
 
\begin{prop}\label{proposition-kernel-nonplucker}
 {Let $X$ be a normal lci polarized GM variety and let $(V_6,V_5,A)$ be the corresponding Lagrangian data.
If either $Y_{\An}$ or $\Dis(X)$ is a reduced hypersurface, we have $Y_{\An} = \Dis(X)$.}
\end{prop}

\begin{proof}
We have $\Ker (\bq(v)) = \An \cap (v \wedge \bw2{V_5})$ by (\ref{kernel-qv}), hence, by definition of $Y_A$ (see Definition~\ref{definition-epw}),
 $Y_{\An}$ and $\Dis(X_{\An,A_1,V_5})$ coincide as sets on the complement of $\P(V_5)$. 
On the other hand, each of $Y_{\An}$ and $\Dis(X_{\An,A_1,V_5})$ is either a sextic hypersurface  or  $\P(V_6)$.
If one is a reduced hypersurface,  so is the other, and they are equal.
\end{proof}

The following result was proved in \cite[Proposition~2.18]{og3}  when $X$ is a smooth ordinary fivefold (it is then automatically strongly smooth).

\begin{coro}\label{coro311}
Let $X$ be a strongly smooth complex GM variety. 
The subschemes $Y_{A(X)}$ and $\Dis(X)$ of $\P(V_6)$ are  equal and they are reduced sextic hypersurfaces. They are also integral and normal, unless $X$ is a bielliptic curve, in which case they are the union of a hyperplane and an integral quintic hypersurface.
\end{coro}

\begin{proof}
As in Remark \ref{rem39}, let us set $\ell:=\dim (A(X) \cap \bw{3}{V_5})$ (this is $5-\dim(X)$ if $X$ is ordinary, and $6-\dim(X)$ if $X$ is special).

If $\ell\le 3$, we saw in Remark \ref{rem39} that  {$A(X)$ contains no decomposable vectors, hence}~$Y_{A(X)}$ is an integral normal sextic hypersurface 
whose singularities are described in Theorem~\ref{theorem-ogrady-stratification}. By Proposition \ref{proposition-kernel-nonplucker}, it is equal to $\Dis(X)$.

 If $\ell=4$ and $X$ is an ordinary GM curve, $\Dis(X)$ is a reduced  hypersurface  (Proposition~\ref{disccurve}). 
By Proposition \ref{proposition-kernel-nonplucker} again, it is equal to $Y_{A(X)_0}$. If $X$ is a special GM surface 
then, by Lemma~\ref{lemma-disc-ord-spe}, its discriminant locus equals $\Dis(X_\ord)$ which is an ordinary GM curve, and the previous argument applies.

Finally, if $\ell = 5$, then $X$ is a special (\ie, bielliptic) GM curve, and the statements follow from Proposition~\ref{disccurve} 
 {in the same way.}
\end{proof}

\begin{rema}\label{rem312}
Let $X$ be a    strongly smooth complex GM curve, with Pl\"ucker hyperplane $V_5 := V_5(X) \subset V_6$ and associated Lagrangian $\An = A(X)_0$.

{\em When $X$ is ordinary,   $V_5 $ is the only point of $Y^{\ge 4}_{\An^\bot}$.}  
Indeed,
$\P(\An)$ contains exactly 5 decomposable vectors $\bw3{V_{3,1}},\dots,\bw3{V_{3,5}}$, 
which are all in $\bw3{V_5}$ (Remark~\ref{rem39}). Moreover, by \cite[(5.10)]{ah}   or \cite[(2.3.2)]{og6}, 
we have $V_5=V_{3,i}+V_{3,j}$ for all $i\ne j$.
 Assume that $V'_5\subset V_6$ is another hyperplane in $Y^{\ge 4}_{\An^\bot}$, so that  $\dim(\An\cap\bw3{V'_5})\ge 4$. 
The points of the scheme $M':=\P(\An\cap\bw3{V'_5})\cap \Gr(3,V'_5)$ must be among $V_{3,1},\dots,V_{3,5}$. This implies $V'_5=V_5$, unless 
  $M' $ is a single point, say $w:=V_{3,1}$, with multiplicity 5. Since $M'$ 
is a dimensionally transverse intersection, one   checks, using its Koszul resolution  in $\Gr(3,V'_5)$  and Kodaira vanishing, that its linear span is $\P(\An\cap\bw3{V'_5})$. 
 This implies $\P(\An\cap\bw{3}{V'_5}) \subset \T_{\Gr(3,V'_5),w}$. But    the 3-plane $\P(\An\cap\bw{3}{V'_5})$ must then meet 
the 4-dimensional Schubert cycle $\Gr(3,V'_5) \cap \T_{\Gr(3,V'_5),w}  = \cone{}(\PP^2\times\PP^1) $ along a curve, which consists of decomposable vectors 
in $\P(\An)$. This is a contradiction. 

{\em When $X$ is special (\ie, bielliptic),   $V_5 $ is the only point of $Y^{\ge 5}_{\An^\bot}$.}  
Indeed, the set $\Theta_{\An}$ of decomposable vectors in $\P(\An)$ is a smooth elliptic curve contained in $\bw3{V_5}$ (Remark~\ref{rem39}).
If $V'_5\subset V_6$ is another point of   $Y^{\ge 5}_{\An^\bot}$, the scheme $\P(\An\cap\bw3{V'_5})\cap \Gr(3,V'_5)$ has everywhere dimension $\ge 1$ and is   contained in $\Theta_{\An}$. These two schemes are thus equal,   hence  contained in $\Gr(3,V_5)\cap \Gr(3,V'_5)=\Gr(3,V_5\cap V'_5) $. Since the linear span of $\Theta_{\An}$ has dimension 4, this implies $V_5=V'_5$.

By  {Corollary}~\ref{corollary:isom-gm}(b), the curve   $X$  {is} determined (up to isomorphism) by $A(X)$ and the hyperplane $V_5\subset V_6$.  In both cases, it is therefore uniquely determined  by $A(X)$. 
In fact,  {\em $X$ is also uniquely determined (up to isomorphism) by the sextic $Y_{A(X)}$.} When $X$ is  special, this is   because $\P(V_5)$ is the only degree-1 component of $Y_{A(X)}$ (Propositions \ref{disccurve} and \ref{proposition-kernel-nonplucker}).
When $X$ is ordinary,   the singular locus of $Y_{A(X)}$ is the union of 5 planes (one for each decomposable vector in $A$; see Remark \ref{rem39} and   \eqref{singya}), 
an integral surface $S$ coming from the singular locus of the theta divisor of the Jacobian of  $X$  (see \cite[Theorem~1.4]{ah}; one checks that it has degree~30 and is contained in $Y^{\ge 2}_{A(X)}$) 
and another irreducible surface of degree 10, also contained in~$Y^{\ge 2}_{A(X)}$.
Moreover, the intersection of $S $ with each of these 5 planes is a plane sextic curve whose normalization is $X$ (this follows from \cite[footnote~1, p.~174]{ah}).  
 \end{rema}

 When $\k = \C$, we can restate the results of  {Corollary}~\ref{corollary:isom-gm} in terms of the EPW sextic~$Y_{\An(X)}$.
The group $\PGL(V_6)_{\An,V_5}$ was defined in~\eqref{pgl-av}; we define also
\begin{equation}\label{pgl-avv}
\PGL(V_6)_{Y_{\An},V_5} :=  \{ g \in \PGL(V_6) \mid   g  (Y_{\An}) = Y_{\An}\ , g (V_5) = V_5 \}. 
\end{equation}  
 
\begin{prop}\label{corollary-aut-x-aut-a}
Let $(X,H)$  and $(X',H')$ be  {normal} lci complex polarized GM varieties,   with  corresponding Lagrangian data
$(V_6,V_5,\An)$ and $(V_6',V_5',\An')$. 
\begin{itemize}
\item[\rm (a)]
Any isomorphism $\phi\colon (X,H) \isomto (X',H')$ induces an isomorphism 
$V_6  \isomto V_6'$ 
which takes     $V_5 $ to $V_5'$ and
$Y_{\An}$ to~$Y_{\An'}$. 
\item[\rm (b)] 
If $X$  is smooth and  $X$ and $X'$ have   same type   and same dimension $\ge 3$,   
every isomorphism 
$\varphi \colon V_6  \isomto V_6'$ that takes $V_5 $ to $V_5'$ and
$Y_{\An}$ to~$Y_{\An'}$ is induced
by an isomorphism between $(X,H)$ and $(X',H')$. 
\item[\rm (c)]
 If $X$ is smooth of dimension $\ge 3$, we have $\PGL(V_6)_{\An,V_5} = \PGL(V_6)_{Y_{\An},V_5}$ 
 and the group $\Aut(X)$ is finite.
\item[\rm (d)]
If $X$  {is smooth of} dimension $\ge 3$, the group $\Aut(X)$ is trivial if $X$ is very general  ordinary   and $\Aut(X) \cong \Z/2$ if $X$ is  very general special.
 \end{itemize}
\end{prop}

\begin{proof}
(a) follows from  {Corollary}~\ref{corollary:isom-gm}  {and the definition of $Y_{\An}$}. 

(b)  By Theorem~\ref{theorem-singdec}, $\An$ contains no decomposable vectors and 
by Proposition \ref{prop33}(b), we have $(\bw3\varphi)(\An ) = \An'$.  Item (b) then follows from  {Corollary}~\ref{corollary:isom-gm}.
 
(c)   The equality
 $\PGL(V_6)_{\An,V_5} = \PGL(V_6)_{Y_{\An},V_5}$ is proved by the argument in part (b).
Furthermore, $\Aut_H(X) = \Aut(X)$ when $X$ is smooth and $\dim (X) \ge 3$ 
 {(Lemma~\ref{lemma:pic-gm})}.
Therefore,  to show that $\Aut(X)$ is finite, it is enough by  {Corollary}~\ref{corollary:isom-gm}(c) to show that $\PGL(V_6)_{\An,V_5}$ is finite. 
But the latter is a subgroup of the group $\PGL(V_6)_{\An}$ defined in \eqref{pgla} and that group is finite by Proposition~\ref{autoepw}(a).
 
(d)  It is enough to show that $\PGL(V_6)_{\An} = \Aut(Y_{\An})$ is trivial for a very general Lagrangian subspace $\An \subset \bw3V_6$. This  is Proposition \ref{autoepw}(b).
 \end{proof}

Lengthy  computations   also show directly that for any smooth complex GM variety $X$ of dimension $\ge 3$, one has $H^0(X,T_X)=0$ 
(these computations were done in \cite[Theorem 3.4]{dim} for ordinary GM threefolds  and in \cite[Proposition 4.1]{dims} for   GM fourfolds). 
We recover the fact that the automorphism group of $ X $, being discrete, is finite, as asserted in the proposition.

\subsection{Period partners}

As we saw in the previous  section, many properties of a GM variety depend only 
on $\An(X)$---the even part of the corresponding Lagrangian subspace. More evidence will come further on.

\begin{defi}\label{pepa}
 {Normal} lci GM varieties $X_1$ and $X_2$ are {\sf period partners} if 
$\dim (X_1) = \dim (X_2)$ and there exists an isomorphism $\varphi \colon V_6(X_1) \isomto V_6(X_2)$ 
such that $(\bw3\varphi)(A(X_1)) = A(X_2)$. 
\end{defi}

\begin{rema}\label{rem326}
The last condition in the definition implies $\varphi(Y_{A(X_1)}) = Y_{A(X_2)}$. 
Conversely, assume that the lci   GM varieties $X_1$ and $X_2$ are such that $\dim (X_1) = \dim (X_2)\ge 3$, that
there exists an isomorphism $\varphi \colon V_6(X_1) \isomto V_6(X_2)$ such that $\varphi(Y_{A(X_1)}) = Y_{A(X_2)}$,   
and that $X_1$ is smooth. Proposition \ref{prop33} then shows that  $X_1$ and $X_2$ are period partners and 
that $X_2$ is also smooth by Theorem~\ref{theorem-singdec}.
\end{rema}

Period partners may have different type---an ordinary GM variety may be a period partner of a special GM variety.

\begin{rema}
The name ``period partners'' suggests a relation with  the period map. In the article \cite{DK}, we   show that indeed,   smooth GM varieties of the same dimension are period partners if and only
if they 
 are in the same fiber of an  appropriately defined  period map from the moduli space of GM varieties  to
  an appropriate period domain.
\end{rema}

A major difficulty in realizing the program of the remark is the construction of the moduli space of GM varieties.
While this space is not available, we can formulate a naive description of the fiber of the period map as follows. 
The dual EPW stratification  {is} defined in ~\eqref{dualYA} and the group $\PGL(V_6)_{\An}$ is defined in~\eqref{pgla}.

\begin{theo}\label{lemma-isoclasses-pp}
Let $X$  be a smooth GM variety of dimension $n \ge 3$ defined over a quadratically closed field, with corresponding Lagrangian data $(V_6,V_5,\An)$.
There is a natural bijection (defined in the proof) between the set of isomorphism classes of period partners of $X$ and the set 
$\bigl(Y_{\An^\bot}^{5-n} \sqcup Y_{\An^\bot}^{6-n}\bigr)/\PGL(V_6)_{\An}$.
\end{theo}

This theorem gives a precise answer to  the question asked at the end of \cite[Section~4.5]{iliev-manivel}.

\begin{proof}
By Proposition~\ref{corollary-aut-x-aut-a}, an lci GM variety $X'$ is a period partner of $X$ 
if and only if there is an isomorphism $\varphi\colon V_6(X') \isomto V_6$ such that $(\bw3\varphi)(\An(X')) = \An$,
and $\varphi(V_5(X')) \in Y_{\An^\bot}^{5-n}$ if $X'$ is ordinary, or
$\varphi(V_5(X')) \in Y_{\An^\bot}^{6-n}$ if $X'$ is special (by \eqref{eq22}, this  ensures the equality of dimensions).
Since   such an $X'$ is automatically a smooth GM variety by Theorem~\ref{theorem-singdec}, and the isomorphism 
$\varphi$ is defined up to the action of $\PGL(V_6)_{\An}$, this proves the lemma.
\end{proof}

In Section~\ref{section-birationalities}, we show that  over $\C$, period partners are always birationally isomorphic. 
In the   article~\cite{DK}, we show that they have the same primitive Hodge structure in the middle cohomology  (in case of even dimensions). 
Finally, in the joint work  {\cite{KP}} of the second author with Alexander Perry, we  {discuss a relation between derived categories of period partners}.

\subsection{Duality}

We introduce here a notion of duality for GM varieties. It is similar in flavor
to the notion of period  partnership  and we will see later that dual varieties share
many geometric  properties, as do  period partners.

\begin{defi}\label{dual}
 {Normal} lci GM varieties $X $ and $X'$ are called {\sf dual} if 
$\dim (X ) = \dim (X')$
and there exists an isomorphism $\varphi \colon V_6(X ) \isomto V_6(X')^\vee$ 
such that $(\bw3\varphi)(A(X )) = A(X' )^\perp$.
\end{defi}

The last condition in the definition implies $\varphi(Y_{A(X )}) = Y_{A(X')^\perp}$.
Conversely, as in Remark~\ref{rem326}, if $\dim (X) = \dim (X')\ge 3$, if there exists an isomorphism $\varphi \colon V_6(X ) \isomto V_6(X')^\vee$ 
such that $\varphi(Y_{A(X )}) = Y_{A(X')^\perp}$, and if $X $ is smooth,   $X'$ is also  smooth and dual to $X $.

Duality is a symmetric relation. Moreover, if $X'$ and $X''$ are both dual to $X$, they are period partners. Analogously,
if $X'$ is dual to $X$ and $X''$ is a period partner of $X'$ then $X''$ is dual to $X$.

One can give a description of all possible duals of a given GM variety analogous 
to the description of the set of all period partners.

\begin{theo}\label{lemma-isoclasses-dual}
Let $X$ be a smooth GM variety of dimension $n \ge 3$ defined over a quadratically closed field, with  corresponding Lagrangian data $(V_6,V_5,\An)$.
There is a natural bijection between the set of isomorphism classes of its dual varieties  and the set 
$\left(Y_{\An}^{5-n} \sqcup Y_{\An}^{6-n}\right)/\PGL(V_6)_{\An}$.
\end{theo}

\begin{proof}
The proof is identical to that of Theorem~\ref{lemma-isoclasses-pp}.
 \end{proof}

Let $X$ and $X'$ be   dual  smooth GM varieties. Set $V_6 := V_6(X)$ and choose an isomorphism
$V_6(X') \cong V_6^\vee$. The hyperplane $V_5 := V_5(X) \subset V_6$ gives a line $V_5^\perp \subset V_6^\vee$ 
and conversely, the hyperplane $V'_5 := V_5(X') \subset V_6^\vee$ gives a line ${V'_5}^\perp \subset V_6$.
Since $V_6$ parameterizes quadrics cutting out $X$, and $V_6^\vee$ parameterizes those cutting out $X'$, these lines define quadrics
\begin{equation*}
Q(X) \subset \P(W(X')) \subset \P(\bw2V'_5 \oplus \k)
\quad\text{and}\quad
Q(X') \subset \P(W(X)) \subset \P(\bw2V_5 \oplus \k).
\end{equation*}
 {Consider also the quadrics
\begin{equation*}
Q_0(X) := Q(X) \cap \P(\bw2V'_5) \subset \P(\bw2V'_5)
\quad\text{and}\quad
Q_0(X') := Q(X') \cap \P(\bw2V_5) \subset \P(\bw2V_5).
\end{equation*}} 

\begin{prop}\label{prop326}
In the above setup, assume that the line ${V'_5}^\perp$ is not contained in the hyperplane $V_5 \subset V_6$.
Then the duality between $V_6$ and $V_6^\vee$ induces a duality between the spaces $V'_5$ and $V_5$ under which  the quadrics   {$Q_0(X)$ and $Q_0(X')$}
are projectively dual.
\end{prop}

\begin{proof}
 The first statement is obvious. For the second, choose generators
 $v\in {V'_5}^\perp \subset V_6$ and $v' \in V_5^\perp \subset V_6^\vee$. By assumption, $v \notin V_5$ and $v' \notin V'_5$.
Consider the direct sum decomposition
\begin{equation*} 
\symv = 
 \bigl(\bw3{V_5} \oplus L) \oplus \bigl((v\wedge\bw2{V_5}) \oplus \k \bigr).
\end{equation*}
The quadric $Q(X') = Q(v)$ is the quadric corresponding to the Lagrangian subspace $\hat{A} \subset \symv$
via Proposition~\ref{lemma-lagrangian-quadratic}, denoted there by $Q_2^{\hat{A}}$. By Lemma~\ref{lemma-isotropic-reduction},
the quadric $Q_0(X')$ is the result of the isotropic reduction with respect to the isotropic subspace $L \subset \symv$.
Therefore, $Q_0(X')$ is the quadric corresponding to the Lagrangian subspace $\An \subset \symv_0 = \bw3V_6$
with respect to the direct sum decomposition
\begin{equation*}
\bw3V_6 =  \bw3{V_5} \oplus (v\wedge\bw2{V_5}),
\end{equation*}
\ie, $Q_0(X') = Q_2^{\An}$. The same argument shows   $Q_0(X) = Q_2^{\An^\perp}$ for the direct sum decomposition
\begin{equation*}
\bw3V_6^\vee =  \bw3{V'_5} \oplus (v'\wedge\bw2{V'_5}).
\end{equation*}
Finally, note  that the isomorphism $\bw3V_6 \isomto \bw3V_6^\vee$ given by the symplectic form interchanges the summands
of these decompositions. To  prove this, since the summands are Lagrangian,  it is enough to show that $v'\wedge \bw2V'_5$
annihilates $\bw3V_5$ and that $v\wedge \bw2V_5$ annihilates $\bw3V'_5$. This follows from the fact that $v'$
annihilates $V_5$ (by its definition) and that $v$ annihilates $V'_5$. 

{Since $X$ and $X'$ are dual,} this isomorphism also 
takes $\An$ to $\An^\perp$, hence identifies $Q_2^{\An^\perp}$ with $Q_1^{\An} \subset \P(\bw3V_5)$. The proposition follows, since
the quadrics $Q_1^{\An}$ and $Q_2^{\An}$ are projectively dual by Proposition~\ref{lemma-lagrangian-quadratic}.
 \end{proof}

 If $\k$ is quadratically closed, the variety $X$ is determined by the quadric $Q_0(X')$: if $X$ is ordinary, we have $X = \Gr(2,V_5) \cap Q_0(X')$,
and if $X$ is special,  $X = (\Gr(2,V_5) \cap Q_0(X'))^\spe$. Analogously, $X'$ is determined by the quadric $Q_0(X)$.~Proposition \ref{prop326} thus shows that   projective duality of quadrics governs duality of GM varieties.

\begin{rema}
One can also show that if  $X$ and $X'$ are both special GM varieties, the duality between $V_5$ and~$V'_5$ can be extended
to a duality between $\bw2V_5 \oplus \k$ and $\bw2V'_5 \oplus \k$ in such a way that the quadric
$Q(X') \subset \P(\bw2V_5 \oplus \k)$ is projectively dual to the quadric $Q(X) \subset \P(\bw2V'_5 \oplus \k)$.
\end{rema}

\section{Rationality and birationalities}\label{section-birationalities}

The goal of this section is to prove that  {smooth} complex  GM varieties $X_1$ and $X_2$ of dimension $\ge3$ which are either 
period partners in the sense of Definition \ref{pepa} (so that the  Lagrangians $A(X_1)$ and $A(X_2)$  are  isomorphic) 
or dual in the sense of Definition \ref{dual} (so that $A(X_2)$ is isomorphic to $A(X_1)^\perp$)  are birationally isomorphic.

The proof is based on the construction of two
quadric fibrations. Although all dimensions could be treated with our methods, we prefer, for the sake of simplicity, to prove first that all smooth complex GM varieties  of dimensions 5 or 6 are rational, and then   concentrate on the cases of dimensions 3 or 4, for which  birationality is really meaningful.

\subsection{Rationality of smooth complex GM varieties  of dimensions 5 or 6}

The rationality of a general complex GM fivefold  is explained in \cite[Section 5]{rot}.
We give a different argument, showing that all smooth GM fivefolds and sixfolds are rational.

\begin{lemm}\label{lemma-56fold-delpezzo}
Any smooth complex GM  of dimension $5$ or $6$ contains a smooth quintic del~Pezzo surface.
\end{lemm}

\begin{proof}
Let $(V_6,V_5,\An)$ be a  Lagrangian data corresponding to such a variety $X$  and let $Y_{\An} \subset \P(V_6)$ be the corresponding EPW sextic.
By Theorem~\ref{theorem-singdec}, $\An$ contains no decomposable vectors. By Lemma~\ref{lemma-y2-hyperplane},
there is a point $v \in Y^2_{\An} \setminus \P(V_5)$ and by~\eqref{kernel-qv}, the corresponding quadric $Q(v) \subset \P(W)$  has corank 2 and $X = \CGr(2,V_5) \cap Q(v)$. 

Let $n:=\dim(X)\in\{5,6\}$. 
The quadric $Q(v)$ has rank $n+3$ in $\P^{n+4}$, hence has an isotropic subspace $I \subset W$ of dimension $2 + \left\lfloor\tfrac{n+3}{2}\right\rfloor \ge6$.
Then $\P(I) \cap X = \P(I) \cap \CGr(2,V_5)$ is a quintic del Pezzo surface if the intersection is dimensionally transverse and smooth. 
We show that this is true for a general choice of~$I$.

The space of all $Q(v)$-isotropic 6-subspaces is $\OGr(4,n+3)$, a smooth variety which we denote by $B$. Let $\cI \subset W\otimes \cO_B$ be the tautological bundle
of $Q(v)$-isotropic subspaces over $B$ and let $\P_B(\cI)$ be its projectivization. It comes with a natural map $\P_B(\cI) \to \P(\bw2V_5 \oplus \k)$
induced by the embedding $\cI \subset W\otimes\cO_B \subset (\bw2V_5 \oplus \k)\otimes \cO_B$. The scheme 
\begin{equation*}
\cX = \P_B(\cI) \times_{{\P(\bigwedge^2{\!V_5} \oplus \k)}} \CGr(2,V_5)
\end{equation*}
maps to $B$ and the fibers are the intersections $\P(I) \cap X$. 
 
To show that the general fiber is dimensionally transverse and smooth,
it is enough to show that $\cX$ is  {itself} smooth of the expected dimension.
 {For this,} consider the projection $\cX \to \CGr(2,V_5)$. Since the fibers of $\P_B(\cI)$ are contained in $Q(v)$,
the image of $\cX$ is contained in $\CGr(2,V_5) \cap Q(v) = X$. Moreover, the fiber of this projection over a point
$w \in X \subset \P(W)$ is the space of all $Q(v)$-isotropic 6-subspaces $I \subset W$ such that $w \in I$.

Since $w$ cannot be in the kernel of $Q(v)$ (otherwise, $X$ would be singular),  
the space of such $I$ is parameterized by $\OGr(3,n+1)$.
All fibers of the projection $\cX \to X$ are therefore smooth of the same dimension, hence $\cX$ is smooth.
As   observed earlier, the general fiber of the projection $\cX \to B$ is therefore a smooth quintic del Pezzo surface.
\end{proof}

We can now prove rationality of smooth GM fivefolds and sixfolds.

\begin{prop}\label{proposition-56fold-rational}
Any smooth complex GM variety of dimension $5$ or $6$ is rational.
\end{prop}

\begin{proof}
We   use the notation of the previous proof.
Let $Q = Q(v)$ be a quadric of corank~$2$ in $\P(W)$ such that $X = \CGr(2,V_5) \cap Q$. 
Let $I_0 \subset W$ be a general $Q$-isotropic 6-subspace, so that $S_0 := \P(I_0) \cap X$ is a smooth
quintic del Pezzo surface. Consider the linear projection
\begin{equation*}
\pi_{S_0} \colon X \dashrightarrow \P^{n-2} 
\end{equation*}
 from the  linear span $\P(I_0)$ of  $S_0$, where $n = \dim (X)$. A general fiber   is also a quintic del Pezzo surface.
Indeed, the intersection of a general $\P^6$ containing $\P(I_0) = \P^5$ with $X$ is a 2-dimensional GM intersection  
containing the quintic del Pezzo surface $S_0$. It is the union  of $S_0$ and a residual quintic
del Pezzo surface $S$ (more precisely, the intersection of $Q$ with this $\P^6$ is the union 
of $\P(I_0)$ and a residual space $\P(I)$, where $I$ is another $Q$-isotropic 6-subspace, and $S = \P(I) \cap X$).

The argument of the proof of  Lemma~\ref{lemma-56fold-delpezzo} shows that for a general choice of $S_0$, the general fiber of 
$\pi_{S_0}$ is smooth. Hence the field of rational functions on $X$ is the field of rational functions on 
a smooth quintic del Pezzo surface defined over the field of rational functions on $\P^{n-2}$.
But a smooth quintic del Pezzo surface is rational over any field by a theorem of Enriques
(see~\cite{sb92}).
\end{proof}

\begin{rema}
Let $X$ be a smooth complex GM fourfold such that $Y^3_{A(X)} \setminus \P(V_5(X)) \ne \emptyset$.
The same argument shows that $X$ is rational. More generally, using Theorem~\ref{coro4}, one can show 
that $X$ is rational as soon as $Y^3_{A(X)} \ne \emptyset$  {(see~\cite[Lemma~6.7]{KP} for more details)}.
\end{rema}

\subsection{The first quadric fibration}

Let $X$ be a smooth GM variety of dimension~$n\ge3$, with  {extended} Lagrangian data $(V_6,V_5,\An, A_1 )$ and GM data $(W,V_6,V_5,L,\mu,\bq,\eps)$.
By Theorem~\ref{theorem-singdec}, $\An$ contains no decomposable vectors. Consider the Gushel bundle $\cU_X$ and its projectivization $\P_X(\cU_X)$. 
The canonical embedding $\cU_X \hra V_5 \otimes \cO_X$ induces a regular map
\begin{equation*}
\rho_1\colon \P_X(\cU_X) \to \P(V_5).
\end{equation*}
In   Proposition \ref{proposition-kernel-plucker}, we   prove that $\rho_1$ is a quadric fibration  and identify its discriminant locus.

Let $W_0^\perp = \An \cap \bw{3}{V_5}$ be the even part of the space of linear equations of $X$.
 We consider $W_0^\perp$ as a subspace of the space $\bw2V_5^\vee$ of skew forms on $V_5$. 
Let 
\begin{equation*}
  \Sigma_1(X) \subset \P(V_5)
\end{equation*}
be the union of  the kernels of all  {non-zero skew forms in the space $W_0^\perp$}. It can be described in terms of the incidence correspondence 
$\widehat{Y}_{\An}$ and its two projections $p\colon \widehat{Y}_{\An} \to Y_{\An} \subset \P(V_6)$ and $q\colon \widehat{Y}_{\An} \to Y_{{\An}^\perp} \subset \P(V_6^\vee)$
(see~\eqref{yahat}).

\begin{lemm}\label{lemma-sigma-1}
We have $\Sigma_1(X) = p(q^{-1}([V_5]))$. In particular, $\Sigma_1(X) \subset Y_{A} \cap \P(V_5)$. Moreover, 
$\dim( \Sigma_1(X)) \le 1$ if $n=4$, and 
$\dim (\Sigma_1(X)) \le 2$ if $n=3$.
\end{lemm}

\begin{proof}
The lemma follows immediately from Remark~\ref{remark-hat-y-q}.
\end{proof}

We now prove an analogue of Proposition~\ref{proposition-kernel-nonplucker}. 
 
\begin{prop}\label{proposition-kernel-plucker}
Let $X$ be a smooth GM variety of dimension~$n\ge3$.

Over $\P(V_5) \setminus \Sigma_1(X)$, the map $\rho_1\colon \P_X(\cU_X) \to \P(V_5)$ is a relative quadratic hypersurface in a $\P^{n-2}$-fibration, whose fiber over a point $v $ has corank $k$ if and only if $v \in Y^k_{\An}  {\cap \P(V_5)}$. 
In particular, in $\P(V_5) \setminus \Sigma_1(X)$, its discriminant locus is $Y_{\An} \cap (\P(V_5) \setminus \Sigma_1(X))$.

Over $\Sigma_1(X)$, the map $\rho_1 $ is a relative quadratic hypersurface in a $\P^{n-1}$-fibration, whose fiber over a point $v $ has corank $k$ if and only if $v \in Y^{k+1}_{\An} $. 
\end{prop}

\begin{proof}
Choose an arbitrary $v_0 \in V_6 \setminus V_5$ and let $v \in V_5$.
The fiber $\rho_1^{-1}(v)$ is the set of  $V_2\subset V_5$ containing $v$ which correspond to points of $X$. This is the intersection $Q'(v)$ of the  quadric $Q(v_0) \subset \P(W)$ defining $X$  
with the subspace $\P\bigl( ((v \wedge V_5)\oplus \k )\cap W\bigr) \subset \P(W)$. 

On the other hand, by the argument of Theorem~\ref{theo-bijection}, the quadric $Q(v_0)$ is the quadric
corresponding to the Lagrangian 
 {subspace $\hat{A} = A \oplus A_1 \subset \symv$ with respect to the Lagrangian} decomposition 
$\symv = \LL_1\oplus \LL_2$  of \eqref{dsd}, with $\LL_1=  \bw3{V_5} \oplus L$ and $\LL_2= (v_0\wedge\bw2{V_5}) \oplus \k$, 
under the Lagrangian-quadratic correspondence of Appendix~\ref{section-lagrangians-and-quadrics} (it is the quadric $Q_2^{\hat{A}}$
of Proposition~\ref{lemma-lagrangian-quadratic}). Therefore, the intersection 
\begin{equation*}
Q'(v) = Q(v_0) \cap \P\bigl( (v \wedge V_5)\oplus \k )\cap W\bigr) 
\end{equation*}
can be described by the isotropic reduction of Corollary~\ref{corollary-quadric-restriction}.

For this, we set $\BLL_2:= (v_0 \wedge v \wedge V_5)\oplus \k \subset \LL_2$ and we consider 
the isotropic subspace   
\begin{equation*}
\II_v :=\LL_1\cap \BLL_2^\perp  =
v \wedge \bw{2}{V_5}\subset \LL_1. 
\end{equation*}
According to Corollary~\ref{corollary-quadric-restriction}, the codimension in $\P(\BLL_2)$ of the span of the quadric $Q'(v)$ is equal to $\dim\bigl((\hat{A} \cap \LL_1)/(\hat{A} \cap \II_v)\bigr)$ 
and its kernel is equal to $\bigl(\hat{A} \cap (\II_v \oplus \BLL_2)\bigr)/(\hat{A} \cap \II_v)$, so it remains to describe the intersections 
of $\hat{A}$ with $\II_v$, $\LL_1$, and $\II_v \oplus \BLL_2$. 

First, we have $\hat{A} \cap \II_v = \An \cap (v \wedge \bw2V_5) \subset \An \cap \bw3V_5 = W_0^\perp$, the space of linear equations of~$X$. 
A skew form on~$V_5$, considered as an element of $\bw3V_5$, belongs to $v \wedge \bw2V_5$ if and only if $v$ is in its kernel. 
This means that $\hat{A} \cap \II_v$ is non-zero only for $v \in \Sigma_1(X)$, and the intersection $\hat{A} \cap \II_v$
is then 1-dimensional, since otherwise the space $\An \cap (v \wedge \bw2V_5)$ would contain decomposable vectors. Therefore,
\begin{equation}\label{equation-dim-a-cap-iv}
\dim(\hat{A} \cap \II_v) = 
\begin{cases}
1  & \text{if $v \in \Sigma_1(X)$,}\\
0 & \text{if $v \notin \Sigma_1(X)$.}
\end{cases}
\end{equation} 
Since the space $\hat{A} \cap \LL_1 = W^\perp$ has dimension $6 - n$,   the codimension of the span of $Q'(v)$
in~$\PP(\BLL_2)$ is $5-n$ for $v \in \Sigma_1(X)$ and $6-n$ for $v \notin \Sigma_1(X)$. 
Since the dimension of $\P(\BLL_2)$ is 4, it follows that $Q'(v)$ is a quadratic hypersurface in $\P^{n-1}$ or   $\P^{n-2}$, depending on whether $v$ is in $\Sigma_1(X)$ or not.
Finally, we have 
\begin{equation*}
\II_v \oplus \BLL_2 = ((v \wedge \bw{2}{V_5})\oplus (v_0 \wedge v \wedge V_5)) \oplus \k = v \wedge \bw2V_6 \oplus \k.
\end{equation*}
Therefore, the dimension of the intersection $\hat{A} \cap (\II_v \oplus \BLL_2) = \An \cap (v \wedge \bw2V_6)$ is detected 
by the position of $v \in \P(V_6)$ with respect to the EPW stratification defined by $\An$ on $\P(V_6)$. Taking into account~\eqref{equation-dim-a-cap-iv},
we obtain that for $v \in Y^k_{\An} \setminus \Sigma_1(X)$, the corank of $Q'(v)$ is $k$, and for $v \in Y^k_{\An} \cap \Sigma_1(X)$, the corank of $Q'(v)$ is $k-1$
(this gives   another proof of the inclusion $\Sigma_1(X) \subset Y_{A}$). 
 \end{proof}

\subsection{Birationality of period partners of dimension 3}\label{subsec-bir-pp}

Let $X_1$ and $X_2$ be smooth complex GM threefolds. Assume that they are
period partners, so that they are constructed from   Lagrangian data with the same $V_6$ and $\An$,
 but possibly different Pl\"ucker hyperplanes $V_5^1 \subset V_6$ and $V_5^2 \subset V_6$ and 
 {possibly different Lagrangian subspaces $A_1$ and $A'_1$}.
 The aim of this section is to prove that the varieties $X_1$ and $X_2$ are birationally isomorphic.
 
Consider the diagrams
\begin{equation*}
\xymatrix@R=5mm@C=6mm{
& \P_{X_1}(\cU_{X_1}) \ar[dl] \ar[dr]^{\rho_1^1} \\
X_1 && \P(V_5^1) 
} \qquad
\xymatrix@R=5mm@C=6mm{
& \P_{X_2}(\cU_{X_2}) \ar[dl]_{\rho_1^2} \ar[dr] \\
 \P(V_5^2) && X_2,
}
\end{equation*}
where $\rho_1^1$ and $\rho_1^2$ are the first quadratic fibrations associated with the threefolds $X_1$ and $X_2$ respectively.
Denote by $\Sigma_1(X_1) \subset Y_{\An} \cap \P(V_5^1)$ and $\Sigma_1(X_2) \subset Y_{\An} \cap \P(V_5^2)$ the associated
subsets of the previous section. 
Assume moreover $V_5^1 \ne V_5^2 $, set $V_4 := V_5^1 \cap V_5^2 \subset V_6$, and restrict everything to the common base $\P(V_4)$. We get a diagram
\begin{equation*}
\xymatrix@R=5mm@C=7mm{
& \wtilde X_1 \ar[dl] \ar[dr]^{\tilde\rho_1^1} && \wtilde X_2 \ar[dl]_{\tilde\rho_1^2} \ar[dr] \\
X_1 && \P(V_4) && X_2,
}
\end{equation*}
where $\wtilde{X}_i = \P_{X_i}(\cU_{X_i}) \times_{\P(V_5^1)} \P(V_4)$ and $\tilde\rho^i_1$ is the restriction of $\rho^i_1$ to $\wtilde{X}_i$.

\begin{lemm}\label{lemma-tilde-x-1}
For each $i\in\{1,2\}$, the map  $\wtilde{X}_i \to X_i$ is the blow up  of  
$X_i \cap \CGr(2,V_4)$. In particular, $\wtilde{X}_i$
is irreducible,  generically  reduced, and birational to $X_i$.
\end{lemm}

\begin{proof}
By definition, $\wtilde{X}_1 = \P_{X_1}(\cU_{X_1}) \times_{\P(V_5)} \P(V_4)$ is a divisor in a smooth 4-dimensional variety
$\P_{X_1}(\cU_{X_1})$, hence each irreducible component of $\wtilde{X}_1$ is at least $3$-dimensional.
The fiber of $\wtilde{X}_1 \to X_1$ over a point $w \in X_1$ is a linear section of the fiber of $\P_{X_1}(\cU_{X_1}) \to X_1$,
so it is a point if $\cU_{X_1,w} \not\subset V_4$, or the line $\P(\cU_{X_1,w})$ if $\cU_{X_1,w} \subset V_4$.
Therefore, the locus of   non-trivial fibers is $X_1 \cap \CGr(2,V_4)$. 

To show that the map $\wtilde{X}_1\to X_1$ is the blow up, it is enough by~\cite[Lemma~2.1]{K} to check that the intersection $X_1 \cap \CGr(2,V_4)$ has dimension $\le 1$.
Since $X_1$ is cut out in $\CGr(2,V_5)$ by three hyperplanes and a quadric, the intersection $X_1 \cap \CGr(2,V_4)$ is 
also cut out in $\CGr(2,V_4)$ by three hyperplanes and a quadric, hence its  {dimension is at least~1 and its degree is at most~4}.
But~$X_1$ is a smooth threefold  {with $\Pic(X_1) = \Z H$}, so by the Lefschetz Theorem, it  contains no surfaces of degree  {less than} $H^3 = 10$, hence  {we have} $\dim(X_1 \cap \CGr(2,V_4)) = 1$.

The same argument works for $\wtilde{X}_2$ as well.
\end{proof}

\begin{theo}\label{proposition-birationality-3}
Let  $X_1$ and $X_2$ be smooth complex GM threefolds which are period partners,
 {\ie, with $A(X_1) = A(X_2)$. Assume $V_5(X_1) \ne V_5(X_2)$ and let $V_4 = V_5(X_1) \cap V_5(X_2) \subset V_6$.}
 Over an open subset   $U \subset \P(V_4)$, there is 
an isomorphism $(\tilde\rho_1^1)^{-1}(U) \cong (\tilde\rho_1^2)^{-1}(U)$. In particular, $\wtilde{X}_1$ is birationally isomorphic 
to $\wtilde{X}_2$ over $\P(V_4)$ and $X_1$ is birationally isomorphic to $X_2$.
\end{theo}

\begin{proof}
By Proposition \ref{proposition-kernel-plucker},  the morphism $\rho_1^1\colon \P_{X_1}(\cU_{X_1}) \to \P(V_5(X_1))$ is, 
outside the locus $\Sigma_1(X_1) \subset \P(V_5)$ (which has dimension $\le 2$ by Lemma~\ref{lemma-sigma-1}), 
a flat double cover branched along the sextic $Y_{\An}\cap \P(V_5(X_1))$.
The morphism $\tilde\rho_1^1\colon \wtilde{X}_1 \to \P(V_4)$ is obtained from $\rho_1^1$ by a base change.

If $\wtilde{X}_1 \xrightarrow{\ g_1\ } \bar{X}_1 \xrightarrow{\ h_1\ } \P(V_4)$ is its Stein factorization, $g_1$ is birational, 
$\bar{X}_1$ is isomorphic to the restriction of $\SSpec(\rho_{1*}^1\cO_{\P_{X_1}(\cU_{X_1})})$ to $\P(V_4)$, 
hence $h_1$ is the double cover of $\P(V_4)$ branched along $Y_{\An} \cap \P(V_4)$
 (note  that $Y_{\An} \cap \P(V_4) \ne \P(V_4)$ since $\wtilde{X}_1$, hence also $\bar X_1$, is generically reduced).

The same argument shows that in the Stein factorization $\tilde\rho_1^2\colon \wtilde{X}_2 \xrightarrow{\ g_2\ } \bar{X}_2 \xrightarrow{\ h_2\ } \P(V_4)$,
the  map $g_2$ is also birational and $h_2$ is again the double covering  branched along the same sextic $Y_{\An} \cap \P(V_4)$.

Since the branch divisors coincide,   
$\bar{X}_1 $ and $ \bar{X}_2 
$ are isomorphic,   
$\wtilde{X}_1$ and $\wtilde{X}_2$ are birationally isomorphic. By Lemma~\ref{lemma-tilde-x-1}, $X_1$ and $X_2$ are birationally isomorphic as well.
\end{proof}

\begin{rema}
One can explicitly decompose the birational  transformation $X_1\isomdra X_2$ into several steps. 
The first half   can be described as follows:
\begin{itemize}
\item the map 
$\wtilde{X}_1 \to X_1$ is the blow up of the union  {$X_1 \cap \CGr(2,V_4) = c_1\cup c_2$} of two conics intersecting in two points;
\item  the map $\wtilde{X}_1 \to \bar{X}_1$ contracts all lines intersecting $c_1\cup c_2$ and
all conics intersecting $c_1\cup c_2$ twice (these conics are  $\sigma$-conics and correspond to points of $\Sigma_1 \cap \P(V_4)$),
\end{itemize}
 and similarly for the second part. This transformation is very similar to the conic transformation used  in \cite[Theorem 7.4]{dim} 
 to prove  Theorem~\ref{proposition-birationality-3} when $X_1$ is a general GM threefold.
\end{rema}

One can use the same idea to prove  birationality of period partners in any dimension $n \ge 3$.
It requires, however, a long and cumbersome analysis, so we decided to omit it. Instead,
we will use birationalities for dual GM varieties (proved later in this section) to establish the result in dimension 4  {(Theorem~\ref{coro4})}.
Since in dimensions $n \ge 5$, all smooth GM varieties are rational, this shows the result holds in all dimensions $n \ge 3$.

\subsection{The second quadric fibration}\label{sqf}

Let $X$ be a smooth GM variety of dimension $n\ge3$. Instead of the Gushel bundle $\cU_X$, 
we now consider the quotient $(V_5 \otimes \cO_X)/\cU_X$, which we denote simply by $V_5/\cU_X$, 
 its projectivization $\P_X(V_5/\cU_X)$, and the map $\rho_2\colon \P_X(V_5/\cU_X) \to \Gr(3,V_5)$.
In Proposition \ref{proposition-kernel-plucker2}, we   prove that $\rho_2$ is a quadric fibration  
and identify its discriminant locus. 

Let as before $W_0^\perp = \An \cap \bw3V_5$ be the even part of the space of linear equations of $X$.
Considering its elements as skew forms on $V_5$,  we denote by
\begin{equation}\label{Sigma2}
  \Sigma_2^{\ge k}(X) \subset \Gr(3,V_5)
\end{equation}
the set of all 3-dimensional subspaces $V_3$ in $V_5$ which are isotropic for all elements of some $k$-dimensional space of these skew forms. As usual, set $\Sigma_2(X) := \Sigma_2^{\ge 1}(X)$. 

Finally, we will use the   EPW sequence  
\begin{equation*} 
Z^{\ge k}_{A}  =  \{ V_3 \in \Gr(3,V_6) \mid \dim(A \cap ( V_6 \wedge \bw{2}{V_3} ) ) \ge k  \} 
\end{equation*}
defined in \eqref{equation-epw-quartic}  and set $Z_A:=Z^{\ge 1}_{A}$.
It  plays for the second quadratic fibration the same role as the EPW sequence of $\P(V_6)$ plays for the first. 

 In the next lemma, we consider 1-dimensional subspaces $V_1 \subset V_5$ and identify the Grassmannian $\Gr(2,V_5/V_1)$
with the subscheme of $\Gr(3,V_5)$ parameterizing all 3-dimensional subspaces containing $V_1$. 

\begin{lemm}\label{lemma-sigma-2}
For   any smooth GM variety $X$ of dimension $n\ge3$, we have 
$\Sigma_2(X) \subset Z_{\An} \cap \Gr(3,V_5)$.

Assume  moreover  that either  $n=3$ and $X$     ordinary, or $n\ge 4$. If
$V_1 \notin \Sigma_1(X)$,   we have $\dim(\Sigma_2(X) \cap \Gr(2,V_5/V_1)) \le 2$.
\end{lemm}

\begin{proof}
Let $V_3\in \Sigma_2(X)$,  {\ie, the subspace $V_3$} is   isotropic for some non-zero $\omega \in \An \cap \bw3V_5$, considered as a skew form on $V_5$.
Then $\omega \in \bw2V_3 \wedge V_5$. It follows that $V_3 \in Z_{\An}$, which proves 
the first assertion of the lemma.

 Let $\omega \in \An \cap \bw3V_5$ be non-zero. Since $X$ is smooth of dimension $\ge3$, the Lagrangian $\An$ contains no decomposable vectors hence $\omega$,
viewed as a skew form on $V_5$, has 
  rank   4  and its kernel is a 1-dimensional
subspace $K_1(\omega) \subset V_5$. Any 3-dimensional subspace $V_3$ of $V_5$ which is isotropic for $\omega$ contains $K_1(\omega)$, hence the set of such subspaces can be identified with the set of 2-dimensional subspaces in $V_5/K_1(\omega)$ 
which are isotropic for the non-degenerate skew form induced by $\omega$, \ie, with  {a 3-dimensional quadric} $\LGr(2,4)$.

It remains to describe which of these subspaces $V_3$ contain $V_1$.
Since $V_1 $ is not in $ \Sigma_1(X)$, it   projects onto a non-zero subspace $V'_1 \subset V_5/K_1(\omega)$.
The subset of $\LGr(2,4)$ consisting of subspaces containing $V'_1$ is $\P( V^{\prime\perp}_1 /V'_1) = \P^1$.
It follows  that $\Sigma_2(X) \cap \Gr(2,V_5/V_1)$ is dominated by a $\P^1$-fibration over $\P(\An \cap \bw3V_5)$.
Since $\dim (\An \cap \bw3V_5) \le 2$ for GM fourfolds and ordinary GM threefolds (see \eqref{eq22}), the claim follows. 
\end{proof}

\begin{prop}\label{proposition-kernel-plucker2}
 Let $X$ be a smooth GM variety of dimension~$n\ge3$. 

Over $\Gr(3,V_5) \setminus \Sigma_2(X)$, the map $\rho_2\colon \P_X(V_5/\cU_X) \to \Gr(3,V_5)$ is a relative quadratic hypersurface in a $\P^{n-3}$-fibration, whose 
  fiber over a point $U_3 $ has corank $k$ if and only if \mbox{$U_3 \in Z^k_{\An}$}. In
  particular, in $\Gr(3,V_5) \setminus \Sigma_2(X) $, its discriminant locus is $Z_{\An} \cap (\Gr(3,V_5) \setminus \Sigma_2(X))$.

Over $\Sigma_2^l(X)$, the map $\rho_2\colon \P_X(V_5/\cU_X) \to \Gr(3,V_5)$ is a relative quadratic hypersurface in a $\P^{n+l-3}$-fibration, whose fiber over a point $U_3 $ has corank $k$ if and only if $U_3 \in Z^{k+l}_{\An}  $.
\end{prop}

\begin{proof}
We follow the proof of Proposition~\ref{proposition-kernel-plucker} and use the notation introduced therein.
Choose  $v_0 \in V_6 \setminus V_5$ and let $V_3 \subset V_5$.
The fiber $\rho_2^{-1}(V_3)$ is the set of  $V_2\subset V_3$ which correspond to points of $X$. 
This is the intersection 
 \begin{equation*}
Q'(V_3) := Q(v_0) \cap \P\bigl( ( \bw2V_3 \oplus \k ) \cap W\bigr), 
\end{equation*}
so it can be described by the isotropic reduction of Corollary~\ref{corollary-quadric-restriction}.

For this, we let $\BLL_2:= \bw2V_3 \oplus \k \subset \LL_2$ and we consider 
the isotropic subspace   
\begin{equation*}
\II_{V_3} :=\LL_1\cap \BLL_2^\perp  = (\bw3V_5 \oplus L) \cap (\bw2V_3 \oplus \k)^\perp = V_5 \wedge \bw2V_3.
\end{equation*}
According to   Corollary~\ref{corollary-quadric-restriction}, the codimension in $\P(\BLL_2)$ of the span of the quadric $Q'(V_3)$ is equal to $\dim\bigl((\hat{A} \cap \LL_1)/(\hat{A} \cap \II_{V_3})\bigr)$ 
and its kernel is   $(\hat{A} \cap (\II_{V_3} \oplus \BLL_2))/(\hat{A} \cap \II_{V_3})$. It remains to describe the intersections 
of $\hat{A}$ with $\II_{V_3}$, $\LL_1$, and $\II_{V_3} \oplus \BLL_2$. 

We have $\hat{A} \cap \II_{V_3} = \An \cap (V_5 \wedge \bw2V_3) \subset \An \cap \bw3V_5 = W_0^\perp$, the space of linear equations of $X$. 
A skew form on~$V_5$, considered
as an element of $\bw3V_5$, is contained in the subspace $V_5 \wedge \bw2V_3$ if and only if $V_3$ is isotropic for it. 
This means that for $V_3 \in \Sigma_2^l(X)$, we have  
\begin{equation}\label{equation-dim-a-cap-iu}
\dim(\hat{A} \cap \II_{V_3}) = l.
\end{equation} 
Furthermore, the space $\hat{A} \cap \LL_1 = W^\perp$ has dimension $6 - n $, hence the codimension of the span of $Q'(V_3)$
in $\PP(\BLL_2)$ is $6 - n - l$. Since the dimension of $\P(\BLL_2)$ is 3, it follows that $Q'(V_3)$ is a quadratic 
hypersurface in $\P^{n + l - 3}$.
Finally, we have
\begin{equation*}                 
\II_{V_3} \oplus \BLL_2 = (V_5 \wedge \bw{2}{V_3}) \oplus (v_0 \wedge \bw2V_3) \oplus \k = V_6 \wedge \bw2V_3 \oplus \k.                          
\end{equation*}
Therefore, the dimension of   $\hat{A} \cap (\II_{V_3} \oplus \BLL_2) = \An \cap (V_6 \wedge \bw2V_3)$ is detected 
by the position of $V_3 \in \Gr(3,V_5)$ with respect to the    {stratification}  $Z_\An^\bullet$ of   $\Gr(3,V_6)$  defined  in \eqref{equation-epw-quartic}. Taking~\eqref{equation-dim-a-cap-iu} into account,
we deduce that for $V_3 \in Z^k_{\An} \cap \Sigma_2^l(X)$, the corank of $Q'(v)$ is $k-l$. 
\end{proof}

\subsection{Birationality of dual varieties of dimension 4}\label{subsec-bir-dual}

Let $X$ and $X'$ be smooth dual GM varieties of the same dimension $n\ge3$. 
This means that they are constructed from  dual Lagrangian data $(V_6,V_5,\An,A_1)$
and $(V_6^\vee,V'_5,\An^\perp,A'_1)$.  
Set  $V_1:=V^{\prime \bot}_5 \subset V_6$ and assume additionally
\begin{equation}\label{equation-v1-v5}
V_1\subset V_5
\end{equation} 
(we will see later that the general case reduces to this one).

\begin{rema}
The condition~\eqref{equation-v1-v5} is symmetric with respect to $X$ and $X'$. 
Indeed, it can be reformulated as the degeneracy of the restriction of the natural pairing
$V_6 \otimes V_6^\vee \to \k$ to the subspace $V_5 \otimes V'_5 \subset V_6 \otimes V_6^\vee$.
\end{rema}

The construction of birationalities is very similar to the one used in Section~\ref{subsec-bir-pp}.
Assume   $n = 3$ or $n = 4$, and consider the diagrams
\begin{equation*}
\xymatrix@R=5mm@C=5mm{
& \P_{X}(V_5/\cU_{X}) \ar[dl] \ar[dr]^{\rho_2} &&& \P_{X'}(V'_5/\cU_{X'}) \ar[dl]_{\rho'_2} \ar[dr] \\
X && \Gr(3,V_5) & \Gr(3,V'_5) && X',
}
\end{equation*}
where $\rho_2$ and $\rho'_2$ are the second quadric fibrations defined in Section \ref{sqf}, together with the subsets
 $\Sigma_2(X) \subset Z_{\An} \cap \Gr(3,V_5)$ and $\Sigma_2(X') \subset Z_{\An} \cap \Gr(3,V'_5)$. 
We have $\Gr(3,V_5) \subset \Gr(3,V_6)$ and $\Gr(3,V'_5) \subset \Gr(3,V_6^\vee) = \Gr(3,V_6)$, and their intersection is, since~\eqref{equation-v1-v5} is satisfied,   
\begin{equation*}
\Gr(2,V_5/V_1) \cong \Gr(2,4)
\end{equation*}
(it is empty if~\eqref{equation-v1-v5} does not hold). We restrict everything 
to the common base $\Gr(2,V_5/V_1)$ and get a diagram
\begin{equation}\label{defxt}
\xymatrix@R=5mm@C=6mm{
& \wtilde{X} \ar[dl] \ar[dr]^{\tilde\rho_2} && \wtilde{X}' \ar[dl]_{\tilde\rho'_2} \ar[dr] \\
X && \Gr(2,V_5/V_1) && X',
}
\end{equation}
 where $\wtilde{X} = \P_X(V_5/\cU_X) \times_{\Gr(3,V_5)} \Gr(2,V_5/V_1)$ and $\tilde\rho_2$ is the restriction 
 of $\rho_2$ to $\wtilde{X}$---and analogously for $\wtilde{X}'$ and $\tilde\rho'_2$. 

\begin{lemm}\label{lemma-tilde-x-2}
If 
\begin{equation}\label{condition}
V_1 \notin \Sigma_1(X) \cup Y_{\An(X)}^{\ge n-1},
\end{equation} 
  the map $\wtilde{X} \to X$ is the blow up of  
$X \cap \CGr(1,V_5/V_1)$. In particular, $\wtilde{X}$ is irreducible, generically  reduced, and birational to $X$.
Analogous claims hold for $\wtilde{X}'$.
\end{lemm}

\begin{proof}
The proof is analogous to that of Lemma~\ref{lemma-tilde-x-1}. 
The fiber of the map $\wtilde{X} \to X$ over a point $w \in X$ is a point if $V_1 \not\subset \cU_{w}$, and $\P(V/\cU_{w}) \cong \P^2$ otherwise.
The locus of   non-trivial fibers is therefore $X \cap \CGr(1,V_5/V_1)$. Since $X \subset \CGr(2,V_5)$ is an intersection
of $6-n$ hyperplanes and a quadric, $X \cap \CGr(1,V_5/V_1)$ is the intersection of $6-n$ hyperplanes and a quadric 
in $\CGr(1,V_5/V_1) = \P^4$, hence its expected dimension is $n-3$.

The dimension can jump if 
\begin{itemize}
\item either one of the hyperplanes
contains $\CGr(1,V_5/V_1)$, \ie, $V_1 \in \Sigma_1(X)$;
\item or  the quadric contains  {the intersection} $\P(W) \cap \CGr(1,V_5/V_1) = \P^{n-2}$; 
this is equivalent to $V_1 \in Y_{\An(X)}^{\ge n-1}$ by Proposition~\ref{proposition-kernel-plucker}.
\end{itemize}
Therefore, if \eqref{condition} is satisfied, the dimension does not jump and $\dim\bigl(X \cap \CGr(1,V_5/V_1)\bigr) \le n-3$. 
The exceptional set of $\wtilde{X} \to X$ has dimension at most $n-1$  and thus cannot be an irreducible component of $\wtilde{X}$.
It follows that $\wtilde{X}$ is integral 
and the map $\wtilde{X} \to X$ is the blow up of $X \cap \CGr(1,V_5/V_1)$. 
The same argument works for $\wtilde{X}'$.
\end{proof}

\begin{prop}\label{proposition-birationality-4}
Let  $X$ and $X'$ be smooth dual complex GM fourfolds satisfying \eqref{equation-v1-v5}
and
such that    \eqref{condition} hold for both $X$ and $X'$. 
Over a dense open subset $U \subset \Gr(2,V_5/V_1)$, there is an isomorphism $\tilde\rho_2^{-1}(U) \cong \tilde\rho_2^{\prime-1}(U)$. 
In particular, $\wtilde{X}$ is birationally isomorphic to $\wtilde{X}'$ over $\Gr(2,V_5/V_1)$ and $X$ is birationally isomorphic to $X'$.
\end{prop}

\begin{proof}
By Proposition \ref{proposition-kernel-plucker2},  the morphism $\rho_2\colon \P_X(V_5/\cU_X) \to \Gr(3,V_5)$ is, 
outside the locus $\Sigma_2(X)\subset \Gr(3,V_5)$, a flat double cover branched along the quartic $Z_{\An} \cap \Gr(3,V_5 )$. 
The morphism $\tilde\rho_2\colon \wtilde{X} \to \Gr(2,V_5/V_1)$ is obtained from $\rho_2$ by a base change. 
Outside of $\Sigma_2(X) \cap \Gr(2,V_5/V_1)$ (which has dimension at most~$2$ by Lemma~\ref{lemma-sigma-2}), 
$\tilde\rho_2$ is therefore a double covering   branched along $Z_{\An} \cap  \Gr(2,V_5/V_1)$.

Let $\wtilde{X} \xrightarrow{\ g\ } \bar{X} \xrightarrow{\ h\ } \Gr(2,V_5/V_1)$ be its Stein factorization. 
As in the proof of Proposition~\ref{proposition-birationality-3},  we see that $g$  {is} birational and $h$ is the double cover branched along  $Z_{\An} \cap  \Gr(2,V_5/V_1)$.

The same argument shows that in the Stein factorization $\tilde\rho'_2\colon \wtilde{X}' \xrightarrow{\ g'\ }  \bar{X}' \xrightarrow{\ h'\ }  \Gr(2,V_5/V_1)$,
the   map $g'$ is also birational and $h'$ is again the double covering branched 
along the same quartic $Z_{\An} \cap \Gr(2,V_5/V_1)$ (Lemma~\ref{lem39}).
As the branch divisors coincide,  $\bar{X} $ is isomorphic to~$\bar{X}'$, hence $\wtilde{X}$ and $\wtilde{X}'$ are birationally isomorphic.
By Lemma~\ref{lemma-tilde-x-2}, we also have a birational isomorphism between $X$ and $X'$.
\end{proof}

\begin{rema}
Again, one can explicitly describe the resulting birational transformation. 
The map $\wtilde{X} \to X$ is the blow up of a $\sigma$-conic  {$c = X \cap \CGr(1,V_5/V_1)$} and the map $\wtilde{X} \to \bar{X}$  contracts   
all lines intersecting~$c$ (they are parameterized by the curve $Z_{\An}^{\ge 2} \cap \Gr(2,V_5/v')$) 
as well as all conics intersecting $c$ twice (these conics are  $\rho$-conics and correspond to points of $\Sigma_2(X) \cap \Gr(2,V_5/V_1)$). 
The second half of the transformation is analogous.
 \end{rema}

\subsection{Combined birationalities}

In this section, we  combine birational isomorphisms from the previous sections to show
that all period partners and all dual varieties of any dimension $n \ge 3$ are birationally isomorphic.
We first deal with period partners.

\begin{theo}\label{coro4}
Any two smooth complex GM fourfolds which are period partners are birationally isomorphic.
\end{theo}

\begin{proof}
Let $X_1$ and $X_2$ be smooth GM fourfolds obtained from the same Lagrangian subspace~$\An$ with no decomposable vectors, possibly different hyperplanes 
$V_5^1 \subset V_6$ and $V_5^2 \subset V_6$, {and possibly different subspaces $A_1^1$ and $A_1^2$}.
By Theorem~\ref{lemma-isoclasses-dual}, smooth ordinary GM fourfolds that are dual to both  {$X_1$ and~$X_2$} are parametrized by points $V_1$ of the locus $Y_{\An}^1$ defined in \eqref{yaell}. 
If we moreover impose   conditions \eqref{equation-v1-v5} and~\eqref{condition}
for   both $ V_5^1$  and $V_5^2$, we need to take $V_1$ in $ Y_{\An}^1\cap \P(V_5^1\cap V_5^2)$,
but not in $\Sigma_1(X_1) \cup \Sigma_1(X_2)$ (which has dimension at most  1 by Lemma~\ref{lemma-sigma-1}).
If $ Y_{\An}\cap \P(V_5^1\cap V_5^2 )$ has dimension $\ge 3$, this is certainly possible; if it has dimension 2, it is a surface of degree 6, whereas 
$Y_{\An}^{\ge2}$ is an integral surface of degree~40 (Theorem \ref{theorem-ogrady-stratification}), so it is again possible.

We make this choice of $V_1$   and we let $X'$ be the corresponding GM fourfold, dual to both~$X_1$ and $X_2$.
 By the choice of $V_1$, condition~\eqref{equation-v1-v5} holds for both pairs $(X_1,X')$ and $(X_2,X')$;  we also have $V_1 \notin \Sigma_1(X_1) \cup \Sigma_1(X_2) \cup Y_{\An}^{\ge 3}$ by   construction.
Moreover, for each $i\in\{1,2\}$, we have  $(V_5^i)^\perp \notin \Sigma_1(X')$ (by Lemma~\ref{lemma-sigma-1}, $(V_5^i)^\perp \in \Sigma_1(X')$  implies
$(V_1,V_5^i) \in \widehat{Y}_{\An}$,   hence   $V_1 \in \Sigma_1(X_i)$)
and
$V_5^i \notin Y_{\An^\perp}^{\ge 3}$, because $X_1$ and $X_2$ are fourfolds.
Thus   Condition~\eqref{condition} also holds for both pairs.
Applying Proposition \ref{proposition-birationality-4}, we conclude that $X_1$ and $X'$ are birationally isomorphic, and so are  $X_2$ and~$X'$. This proves the theorem.
\end{proof}

\begin{coro}\label{coro417} 
Smooth complex period partners of the same dimension $n \ge 3$  are birationally isomorphic.
\end{coro}

\begin{proof}
For $n = 3$, this is Theorem~\ref{proposition-birationality-3}; for $n = 4$, this is Theorem~\ref{coro4};
  for $n \ge 5$, this is Proposition~\ref{proposition-56fold-rational}.
\end{proof}

The following corollary was brought to our attention by T.~de Fernex.

\begin{coro}\label{coro418}
No smooth complex GM variety $X$ of dimension  $n \ge 3$ is birationally rigid: 
there are smooth prime Fano varieties that are birationally isomorphic to, but not  {biregularly} isomorphic to,~$X$.
\end{coro}

\begin{proof}
By Theorem~\ref{lemma-isoclasses-pp}, there is a bijection between the set of isomorphism classes of period partners of $X$ and the set 
$\bigl(Y_{\An^\bot}^{5-n} \sqcup Y_{\An^\bot}^{6-n}\bigr)/\PGL(V_6)_{\An}$. Since the group $\PGL(V_6)_{\An}$ is finite 
(Theorem \ref{autoepw}(a)), these sets are   infinite. By  Corollary \ref{coro417}, all these period partners of $X$ are 
birationally isomorphic to $X$. This proves the corollary.
\end{proof}

 For threefolds, the assumption~\eqref{condition} is very restrictive; for instance, it rules out all ordinary dual varieties.
We replace it with the more flexible assumption
\begin{equation}\label{condition-2}
V_1 \in (Y_{\An}^2 \cap \PP(V_5)) \setminus \Sigma_1(X).
\end{equation}
The  {scheme} $\tilde{X} \subset \P_X(V_5/\cU_X)$  {discussed in Lemma~\ref{lemma-tilde-x-2}} is then reducible. 
We  describe its irreducible components and the restriction of the map $\tilde\rho_2$ to these components.

\begin{lemm}\label{lemma-tilde-x-special}
Let $X$ be a smooth ordinary GM threefold   and let $V_1 \subset V_5$ be a subspace  such that~\eqref{condition-2} holds.
The scheme $\wtilde{X}$ in~\eqref{defxt} has two irreducible components:   the proper preimage $\wtilde{X}_1$ of~$X$,
and  the total preimage $\wtilde{X}_2$ of a particular line $M \subset X$. 
The restriction of the map $\tilde\rho_2$
to $\wtilde{X}_2$ is a birational isomorphism onto a Schubert hyperplane section of $\Gr(2,V_5/V_1)$.
\end{lemm}

\begin{proof}
The proof of Lemma~\ref{lemma-tilde-x-2} shows that the locus of non-trivial fibers of the projection $\wtilde{X} \to X$
is  $X \cap \Gr(1,V_5/V_1)$, \ie, the intersection in $\Gr(1,V_5/V_1) = \P^3$ of two hyperplanes 
and a  quadric. Since $V_1 \notin \Sigma_1(X)$, the intersection of the hyperplanes is a line, which we denote by $M$. 
Its intersection with the remaining  quadric can be identified with the fiber of the first quadric fibration over the point $V_1 \in \P(V_5)$.
Therefore, under Condition~\eqref{condition-2}, its corank is 2, hence it coincides with $M$. Since the non-trivial fibers of the projection
$\wtilde{X} \to X$ are $\P^2$, the preimage of $M$ in $\wtilde{X}$ has dimension 3,   hence
is an irreducible component of $\wtilde{X}$ which we denote by $\wtilde{X}_2$. The other irreducible component
is the proper preimage $\wtilde{X}_1$ of $X$ in $\wtilde{X}$.

Since $M$ is a line in $\Gr(1,V_5/V_1) = \P(V_5/V_1)$,   it  can be written, as a subvariety of $\Gr(2,V_5)$, as
$M = \{ U_2 \in \Gr(2,V_5) \mid V_1 \subset U_2 \subset V_3 \}$, for some
 3-dimensional subspace $V_3 \subset V_5$
containing $V_1$.
From the definition of the fibers of the map $\wtilde{X} \to X$, we get
\begin{equation*}
\wtilde{X}_2 \cong \{ (U_2,U_3) \in \Fl(2,3;V_5) \mid V_1 \subset U_2 \subset V_3\ \text{ and }\ U_2 \subset U_3 \subset V_5 \}.
\end{equation*}
The map $\tilde\rho_2\colon \wtilde{X} \to  {\Gr(2,V_5/V_1) \subset} \Gr(3,V_5)$ takes a point $(U_2,U_3) \in \wtilde{X}_2$ to $U_3/V_1 \in \Gr(2,V_5/V_1)$, 
  hence maps $\wtilde{X}_2$ birationally onto the Schubert hyperplane section of $\Gr(2,V_5/V_1)$ of all 2-dimensional subspaces intersecting $V_3/V_1$ 
non-trivially (in fact, $\wtilde{X}_2$ is a Springer-type resolution of singularities of the Schubert hyperplane).
\end{proof}

We now  show that under assumption~\eqref{condition-2}, 
 we still have birational threefolds.

\begin{prop}\label{proposition-birationality-duality-3}
Let  $X$ and $X'$ be   dual  smooth ordinary complex GM threefolds. Assume that
$V'_1$ satisfies~\eqref{condition-2} for $X$ and that $V_5^\perp$ satisfies~\eqref{condition-2} for $X'$.
Then 
 $X$ and $X'$ are birationally isomorphic.
\end{prop}

\begin{proof}
Let $\wtilde{X} = \wtilde{X}_1 \cup \wtilde{X}_2$ and $\wtilde{X}' = \wtilde{X}'_1 \cup \wtilde{X}'_2$
be the decompositions into irreducible components of Lemma~\ref{lemma-tilde-x-special}. Consider the maps $\tilde\rho_2$
and $\tilde\rho'_2$ outside of   $(\Sigma_2(X) \cup \Sigma_2(X')) \cap \Gr(2,V_5/V'_1)$ (this
locus is 2-dimensional  {(Lemma~\ref{lemma-sigma-2})}, so it does not affect the birational geometry of $X$ and~$X'$).

 Over a point of $\Gr(2,V_5/V'_1)$ not in $\Sigma_2(X)$,  the fiber of $\wtilde{X}$ is, 
by Proposition~\ref{proposition-kernel-plucker2}, a quadric in~$\P^0$. 
It is non-empty if and only if the corresponding quadratic form is zero, \ie, if its kernel is non-zero.
By Proposition~\ref{proposition-kernel-plucker2} again, this happens
precisely for points of $Z_{\An} \cap \Gr(2,V_5/V'_1)$. This means that away from $\Sigma_2(X) \cap \Gr(2,V_5/V'_1)$,
the map $\tilde\rho_2$ is an isomorphism from $\wtilde{X}$ onto $Z_{\An} \cap \Gr(2,V_5/V'_1)$.
Analogously, away from $\Sigma_2(X') \cap \Gr(2,V_5/V'_1)$, the map $\tilde\rho'_2$ 
is an isomorphism from $\wtilde{X}'$ onto $Z_{\An} \cap \Gr(2,V_5/V'_1)$.
 
By Lemma~\ref{lemma-tilde-x-special}, the components $\wtilde{X}_2$ and $\wtilde{X}'_2$
map birationally onto the Schubert hyperplanes of $\Gr(2,V_5/V'_1)$ given by the three-dimensional spaces $V_3 \subset V_6$ and $V'_3 \subset V_6^\vee$
corresponding to the lines $M_{V'_1} \subset X$ and $M_{V_5^\perp} \subset X'$ respectively. Let us show   $V'_3 = V_3^\perp$, so that the Schubert hyperplanes coincide.

By definition of the line $M_{V'_1}$, the space $V_3$ can be constructed as follows. Let 
\begin{equation*}
\An \cap (V'_1 \wedge \bw2V_6) = \langle v\wedge \xi_1, v\wedge \xi_2 \rangle,
\end{equation*}
where $v$ is a generator of $V'_1$ and $\xi_1,\xi_2 \in \bw2V_6$. Let   $f \in V_6^\vee$ be an equation of $V_5$.
Then $V_3$ is spanned by $v$ and the contractions $\xi_i(f,-)$. Analogously, if
\begin{equation*}
\An^\perp \cap (V_5^\perp \wedge \bw2V_6^\vee) = \langle f\wedge \eta_1, f \wedge \eta_2 \rangle,
\end{equation*}
 the space $V'_3$ is spanned by $f$ and the contractions $\eta_i(v,-)$ of $\eta_i \in \bw2V_6^\vee$.
To identify $V'_3$ with~$V_3^\perp$, it is therefore enough to show   
\begin{equation*}
(\xi_i(f,-),\eta_j(v,-)) = 0 
\end{equation*}
for all $ i,j  \in\{1,2\}$. This  follows   from the equality
\begin{equation*}
(v \wedge \xi_i, f \wedge \eta_j) = (\xi_i(f,-),\eta_j(v,-)),
\end{equation*}
which in   turn follows from $f(v) = 0$.

 We proved that $\wtilde{X}_2$ and $\wtilde{X}'_2$ are mapped by $\tilde\rho_2$ and $\tilde\rho'_2$
onto the same Schubert hyperplane section $H \subset \Gr(2,V_5/V'_1)$. Since  both $\wtilde{X}$ and $\wtilde{X}'$
map birationally onto the quartic hypersurface $Z_{\An} \cap \Gr(2,V_5/V'_1)$, it follows that this quartic   has
two components $ H $ and $ Z'$, where $Z'$ is a cubic hypersurface in $\Gr(2,V_5/V'_1)$ onto which both
  $\wtilde{X}_1$ and $\wtilde{X}'_1$   map  birationally. 
Thus~$\wtilde{X}_1$ is birational to $\wtilde{X}'_1$,   hence $X$ is birational to $X'$. 
\end{proof}

The birational isomorphism of Proposition~\ref{proposition-birationality-duality-3} coincides with the   transformation of~$X$
 (with respect to the line $M$)  defined in \cite[Section~7.2]{dim}.

We can now prove the main result of this section.

\begin{theo}\label{theorem-birationality-duals}
Any two dual smooth complex GM varieties of the same dimension $n \ge 3$ are birationally isomorphic.
\end{theo}

\begin{proof}
We keep the same notation as above.
For $n \ge 5$, the result is trivial by Proposition~\ref{proposition-56fold-rational}, so we may assume $n \in \{3,4\}$.

Assume first  $n = 4$. 
 As in the proof of Theorem~\ref{coro4}, we may choose $V''_1 \subset V_5$ such that the corresponding smooth GM fourfold $X''$ is dual to $X$. Then
$X$ is birationally isomorphic to $X''$ by Proposition~\ref{proposition-birationality-4} and
$X'$ and $X''$ are birationally isomorphic   by Theorem~\ref{coro4}, hence $X$ is birationally isomorphic to $X'$.

Now let $X$ be a smooth GM threefold and set $\An := \An(X)$. By Lemma~\ref{lemma-y2y2-yhat}, we can find $v \in \PP(V_6)$
and $V_5 \subset V_6$ such that 
\begin{equation}\label{equation-choice}
v \in Y_{\An}^2,
\qquad
[V_5] \in Y_{\An^\perp}^2,
\qquad
v \in V_5,
\quad\text{and}\quad
(v,V_5) \not \in (p\times q)(\widehat{Y}_{\An}).
\end{equation}
Let $X'$ be the ordinary GM threefold corresponding to the Lagrangian $\An$ and the Pl\"ucker hyperplane $V_5 \subset V_6$, and let
$X''$ be the ordinary GM threefold corresponding to the Lagrangian $\An^\perp$ and the Pl\"ucker hyperplane $v^\perp \subset V_6^\vee$.

Then $X'$ is a period partner of $X$, hence $X'$ is birational to $X$ by Theorem~\ref{proposition-birationality-3}.
Furthermore, $X''$ is dual to $X'$ and the conditions of Proposition~\ref{proposition-birationality-duality-3} are satisfied
for the pair $(X',X'')$ (the last condition in~\eqref{equation-choice} is equivalent to $v \notin \Sigma_1(X')$ 
by Lemma~\ref{lemma-sigma-1}). Therefore, $X''$ is birational to $X'$ by Proposition~\ref{proposition-birationality-duality-3}. 
Combining these two birationalities, we conclude that $X''$ is birational to $X$. It remains to note that 
any other dual GM threefold of $X$ is a period partner of $X''$, hence is birational to $X''$ by Theorem~\ref{proposition-birationality-3}.
 \end{proof}

\appendix\section{Excess conormal sheaves}\label{section-excess-conormal-sheaves}

 Let $W$ be a $\k$-vector space and 
let $X \subset \PP(W)$ be a closed $\k$-subscheme which is an  intersection of quadrics, \ie,  the twisted ideal sheaf $\cI_X(2)$ on $\PP(W)$ is globally generated. Let 
\begin{equation*}
V_X := H^0(\PP(W),\cI_X(2))
\end{equation*}
be the space of quadrics through $X$. The canonical map $V_X\otimes\cO_{\PP(W)}(-2) \to \cI_X$ is surjective
and its restriction to $X$ induces an epimorphism $V_X \otimes \cO_X(-2) \thra \cI_X/\cI_X^2 = \cN^\vee_{X/\PP(W)}$ 
onto the conormal sheaf of $X$.

\begin{defi}
The {\sf excess conormal sheaf} $\cE\!\cN^\vee_{X/\PP(W)}$ of an intersection of quadrics $X \subset \PP(W)$
is the kernel of the canonical map $V_X \otimes \cO_X(-2) \thra \cN^\vee_{X/\PP(W)}$.
\end{defi}

 When the ambient space $\PP(W)$ is clear, we will simply write $ \cE\!\cN_X^\vee$.
  This sheaf  is defined over $\k$ and  fits in the \emph{excess conormal sequence}
\begin{equation*}\label{defen}
0 \to \cE\!\cN^\vee_{X/\PP(W)} \to V_X \otimes \cO_X(-2) \to \cN^\vee_{X/\PP(W)} \to 0.
\end{equation*}
  We will also often consider the twist
\begin{equation}\label{defent}
0 \to \cE\!\cN^\vee_{X/\PP(W)}(2) \to V_X \otimes \cO_X \to \cN^\vee_{X/\PP(W)}(2) \to 0.
\end{equation}

In this appendix, we discuss some properties of the excess conormal sheaf. 

\begin{lemm}\label{en-lf}
The excess conormal sheaf is locally free on the lci locus of $X$.
\end{lemm}

\begin{proof}
This follows   from the  local freeness of the conormal bundle on the lci locus of $X$.
\end{proof}

\begin{prop}\label{en-funct}
Consider a closed subscheme $X \subset \PP(W)$ and a cone $\cone{K}X \subset \P(W \oplus K)$.
Let $W' \subset W \oplus K$ be a vector subspace and let $X' \subset \P(W') \cap \cone{K}X$ be a closed subscheme.
Set $X'_0 := X' \cap \pcone{K}X$.
If both $X$ and $X'$ are intersections of quadrics,  there is    a canonical commutative diagram 
\begin{equation}\label{equation-excess-normal-compatibility}
\vcenter{\xymatrix@R=6mm@C=6mm{
& \mu^* \cE\!\cN^\vee_{X/\PP(W)} \ar[r] \ar[d] & V_X\otimes\cO_{X'_0}(-2) \ar[r] \ar[d] & \mu^*\cN^\vee_{X/\PP(W)} \ar[r] \ar[d] & 0 \\
0 \ar[r] &  \cE\!\cN^\vee_{X'_0/\PP(W')} \ar[r] & V_{X'}\otimes\cO_{X'_0}(-2) \ar[r] & \cN^\vee_{X'_0/\PP(W')} \ar[r] & 0 
}}
\end{equation}
 of sheaves on $X'_0$, where $\mu$ denotes the natural projection $\P(W\oplus K) \setminus \P(K) = \pcone{K}\P(W) \to \P(W)$ as well as its restriction to $X'_0$.
\end{prop}

\begin{proof}
The embedding $W' \subset W \oplus K$ induces a map $\Sym^2\!W^\vee \to \Sym^2\!W^{\prime\vee}$. 
Denote by $V'$ the image of $V_X$ under this map. Clearly, $\P(W') \cap \cone{K}X$
is cut out in $\P(W')$ by the quadrics in~$V'$. Hence the space $V_{X'}$ of quadrics 
cutting out $X'$ in $\P(W')$ contains $V'$ and   we obtain a canonical map $V_X \to V_{X'}$. 
Furthermore, on the open subset $\P(W') \setminus \P(K) \subset \P(W')$,
the pullback $\mu^*\cI_X$ of the ideal of $X$ in $\P(W)$ is contained in the ideal $\cI_{X'_0}$ of $X'_0$ in $\P(W') \setminus \P(K)$,
hence we have a commutative diagram
\begin{equation*}
\xymatrix@R=6mm@C=6mm{
V_X\otimes\cO_{\PP(W')\setminus\P(K)}(-2) \ar[r] \ar[d] & \mu^*\cI_{X} \ar[d] \\
V_{X'}\otimes\cO_{\PP(W')\setminus\P(K)}(-2) \ar[r] & \cI_{X'_0} .
}
\end{equation*}
Restricting it to $X'_0$ and extending lines to excess conormal sequences proves the claim.
\end{proof}

In some cases, the vertical map in (\ref{equation-excess-normal-compatibility}) between   excess conormal sheaves is an isomorphism.

\begin{lemm}\label{lemma:excess-cone}
In the situation of Proposition {\rm\ref{en-funct}}, assume moreover $W' = W \oplus \k$ and  $X' = \cone{\k}X \subset \P(W')$, 
so that  $X'_0 = \pcone{\k}X$ is the punctured cone. If $\mu\colon X'_0 \to X$ is the natural projection, we have $ \cE\!\cN^\vee_{X'}\vert_{X'_0} \cong \mu^* \cE\!\cN^\vee_X$.
\end{lemm}

\begin{proof}
The intersection in $\P(W')$ of the quadrics in $V_X \subset \Sym^2\!W^\vee \subset \Sym^2\!W^{\prime\vee}$ is the cone $X' = \cone{\k}X$. 
Moreover,  $\cN_{X'_0/\P(W')} \cong \mu^*\cN_{X/\P(W)}$. Thus, the central and   right vertical arrows in~\eqref{equation-excess-normal-compatibility} are isomorphisms. 
Moreover, the map $\mu$ is flat, hence the top line is exact on the left. Therefore the left arrow is an isomorphism as well.
\end{proof}

\begin{lemm}\label{en-hplane}
Assume that $X\subset \P(W)$ is an intersection of quadrics and  is linearly normal.
Let $W' \subset W$ be a  hyperplane  
such that $X' := X \cap \PP(W')$ is a dimensionally transverse intersection.
Then $X'\subset \P(W')$ is an intersection of quadrics and $ \cE\!\cN^\vee_{X'_{\rm lci} } \cong  \cE\!\cN^\vee_{X_{\rm lci} }\vert_{X'_{\rm lci} }$,
where $X_{\rm lci} $ is the lci locus of $X$ and $X'_{\rm lci} = X' \cap X_{\rm lci} $.
\end{lemm}

\begin{proof}
We have an exact sequence
\begin{equation*}
0 \to \cI_X(-1) \to \cI_X \to \cI_{X'/\P(W')} \to 0.
\end{equation*}
Twisting it by $\cO_{\P(W)}(2)$ and using the linear normality of $X$, we conclude that $X'$ is an intersection of quadrics and  $V_{X'} = V_X$.
The diagram of Proposition~\ref{en-funct} gives a commutative diagram
\begin{equation*}
\xymatrix@R=6mm@C=6mm{
&  \cE\!\cN^\vee_{X}\vert_{X'} \ar[r] \ar[d] & V_X \otimes \cO_{X'}(-2) \ar[r] \ar@{=}[d] & \cN_{X/\PP(W)}^\vee\vert_{X'} \ar[r] \ar[d] & 0 \\
0 \ar[r] &  \cE\!\cN^\vee_{X'} \ar[r] & V_{X'} \otimes \cO_{X'}(-2) \ar[r] & \cN_{X'/\PP(W')}^\vee \ar[r] & 0.
}
\end{equation*}
The  right vertical arrow is an isomorphism by the dimensionally transverse intersection condition. Moreover, 
the top line is exact on the lci locus, since the conormal bundle is locally free. Hence the left arrow is an isomorphism.
\end{proof}

\begin{lemm}\label{en-quadric}
Assume $H^0(X,\cO_X) = \k$ and let $Q \subset \PP(W)$ be a quadratic hypersurface such that $X' = X \cap Q$ is a dimensionally transverse intersection.
Then $ \cE\!\cN^\vee_{X'_{\rm lci}} \cong \cE\!\cN^\vee_{X_{\rm lci}}\vert_{X'_{\rm lci}}$, where again $X_{\rm lci} $ is the lci locus of $X$ and $X'_{\rm lci}  = X' \cap X_{\rm lci} $.
\end{lemm}

\begin{proof}
Tensoring  the exact sequences 
\begin{equation*}
0 \to \cI_X \to \cO_{\P(W)} \to \cO_X \to 0
\qquad\text{and}\qquad
0 \to \cO_{\P(W)}(-2) \to \cO_{\P(W)} \to \cO_Q \to 0
\end{equation*}
and taking into account the dimensional transversality  of the intersection $X' = X \cap Q$, we obtain an exact sequence
\begin{equation*}
0 \to \cI_X(-2) \to \cI_X \oplus \cO_{\PP(W)}(-2) \to \cI_{X'} \to 0.
\end{equation*}
Twisting it by $\cO_{\P(W)}(2)$ and taking into account that the condition $H^0(X,\cO_X) = \k$ implies $H^0(\P(W),\cI_X) = H^1(\P(W),\cI_X) = 0$,
we obtain an isomorphism $V_{X'} = V_X \oplus \k$.
It follows that the diagram of Proposition~\ref{en-funct} extends to a commutative diagram
\begin{equation*}
\xymatrix@R=5mm@C=6mm{
&& 0 \ar[d] & 0 \ar[d] \\
&  \cE\!\cN^\vee_{X}\vert_{X'} \ar[r] \ar[d] & V_X \otimes \cO_{X'}(-2) \ar[r] \ar[d] & \cN_{X/\PP(W)}^\vee\vert_{X'} \ar[r] \ar[d] & 0 \\
0 \ar[r] &  \cE\!\cN^\vee_{X'} \ar[r] & V_{X'} \otimes \cO_{X'}(-2) \ar[r] \ar[d] & \cN_{X'/\PP(W')}^\vee \ar[r] \ar[d] & 0 \\
&& \cO_{X'}(-2) \ar@{=}[r] \ar[d] & \cO_{X'}(-2) \ar[d] \\
&& 0 & 0 .
}
\end{equation*}
Its middle column is exact by the above argument  and the right column is exact by the dimensionally transverse intersection condition.
Moreover, the top line is exact on   $X_{\rm lci} $, hence the left vertical arrow is an isomorphism.
\end{proof}

Finally, we compute the excess conormal bundle for some examples. 
Let $\cU$ be  the rank-2 tautological vector bundle on the Grassmannian $\Gr(2,V_5)$.

\begin{prop}\label{en-g25}
Let $X = \Gr(2,V_5) \subset \PP(\bw{2}{V_5})$. Then $ \cE\!\cN^{\vee}_X \isom \det (V_5^\vee) \otimes \cU(-2)$.
\end{prop}

\begin{proof}
The standard resolution of the structure sheaf   $\cO_X$  {in $\P(\bw2V_5)$} is
\begin{equation*}
0 \to \det (V_5^\vee)^{\otimes 2} \otimes \cO(-5) \to \det (V_5^\vee) \otimes V_5^\vee\otimes\cO(-3) \xrightarrow{\ \alpha\ }  \det (V_5^\vee) \otimes V_5\otimes\cO(-2) \to \cO \to \cO_X \to 0,
\end{equation*}
where  {the map $\alpha$} is induced by the contraction $V_5^\vee \otimes \bw{2}{V_5} \to V_5$.
In particular, the space $\det (V_5^\vee) \otimes V_5 = \bw4{V_5^\vee}$ can be identified with the space of quadrics through $X$.
Tensoring this resolution with $\cO_X$ and taking into account that $\Tor_i(\cO_X,\cO_X) = \bw{i}{\cN^\vee_X}$, we deduce
an exact sequence
\begin{equation*}
0 \to \bw{2}{\cN^\vee_X} \to \det (V_5^\vee) \otimes V_5^\vee\otimes\cO_X(-3)  {\xrightarrow{\ \alpha\vert_X\ }} \det (V_5^\vee) \otimes V_5\otimes\cO_X(-2) \to \cN^\vee_X \to 0.
\end{equation*}
The above description of $\alpha$ shows that  {$\alpha\vert_X$} is the  twist by $\det (V_5^\vee)$ of the composition of the epimorphism
$V_5^\vee\otimes\cO_X(-3) \thra \cU^\vee(-3)$, an isomorphism $\cU^\vee(-3) \isomto \cU(-2)$, and the monomorphism 
$\cU(-2) \hra V_5\otimes\cO_X(-2)$; hence the claim.
\end{proof}

Let $V_6$ be a $\k$-vector space of dimension 6 and let again  $\cU$ be the rank-2 tautological
vector bundle on the Grassmannian  $\Gr(2,V_6)$.

\begin{prop}\label{en-g26}
Let $X = \Gr(2,V_6) \subset \PP(\bw{2}{V_6})$. Then $ \cE\!\cN^{\vee}_X $ is isomorphic to $ \det (V_6^\vee) \otimes \left((V_6\otimes\cU(-2))/(\Sym^2\!\cU)(-2)\right)$.
\end{prop}

\begin{proof}
The standard resolution (\cite[Theorem~6.4.1]{W}) of the structure sheaf of $X$  {in $\P(\bw2V_6)$} is
\begin{equation*}
\dots \to \det (V_6^\vee) \otimes \fsl(V_6)\otimes\cO(-3) \xrightarrow{\ \beta\ } \det (V_6^\vee) \otimes \bw{2}{V_6}\otimes\cO(-2) \to \cO \to \cO_X \to 0,
\end{equation*}
where $\beta$ is given by the natural Lie algebra action $\fsl(V_6) \otimes \bw{2}{V_6} \to \bw{2}{V_6}$.
Tensoring this resolution with $\cO_X$, we deduce an exact sequence
\begin{equation*}
\det (V_6^\vee) \otimes \fsl(V_6)\otimes\cO_X(-3)  {\xrightarrow{\ \beta\vert_X\ }} \det (V_6^\vee) \otimes \bw{2}{V_6}\otimes\cO_X(-2) \to \cN^\vee_X \to 0.
\end{equation*}
The above description of $\beta$ shows that   {$\beta\vert_X$} is the  twist by $\det (V_6^\vee)$ of  the composition 
 \begin{equation*}
\fsl(V_6)\otimes\cO_X(-3) \twoheadrightarrow V_6\otimes\cU^\vee(-3) \isomto V_6\otimes\cU(-2) \to \bw{2}{V_6}\otimes\cO_X(-2).
\end{equation*}
Since the first arrow is surjective,   the image of the composition (\ie, the excess conormal bundle of the Grassmannian) 
is isomorphic (up to a twist) to the image of $V_6 \otimes \cU(-2)$. The kernel of the last map is $(\Sym^2\!\cU)(-2)$, hence the claim.
 \end{proof}

 \section{Eisenbud--Popescu--Walter sextics}\label{section-epw}

These sextics were introduced by Eisenbud, Popescu, and Walter in~\cite[Example~9.3]{epw}
and were thoroughly investigated, over $\C$,  by O'Grady in the series of articles~\cite{og1,og2,og3,og4,og5,og6,og7}.
For their remarkable properties, we refer to these articles.  In this appendix, we   sketch 
the original construction of   EPW sextics from   Lagrangian subspaces  and discuss some of their properties.
The only new results here are Lemmas~\ref{lemma-y2y2-yhat} and~\ref{lemma-y2-hyperplane}. 
To be able to use O'Grady's  {results}, we work over $\C$, although many of his results are valid over more general fields.

  \subsection{Overview of EPW sextics}\label{section-epw-overview}

Let $V_6$ be a 6-dimensional complex vector space and let  $A \subset \bw{3}{V_6}$ be a Lagrangian subspace for the  $\det(V_6)$-valued  symplectic form $\omega$
defined by 
\begin{equation*}
\omega(\xi,\eta) = \xi\wedge\eta.
\end{equation*}

\begin{defi}\label{definition-epw}
For any integer $\ell$, we set
\begin{equation}\label{yabot}
Y_A^{\ge \ell}:=\bigl\{[v]\in\P(V_6) \mid \dim\bigl(A\cap (v \wedge\bw{2}{V_6} )\bigr)\ge \ell\bigr\}
\end{equation}
and endow it with a   scheme structure as in \cite[Section 2]{og1}. 
The locally closed subsets
\begin{equation}\label{yaell}
Y_A^\ell := Y_A^{\ge \ell} \setminus Y_A^{\ge \ell + 1} 
\end{equation}
form the   {\sf EPW stratification} of $\P(V_6)$ and the sequence of inclusions 
\begin{equation*}
\P(V_6) = Y_A^{\ge 0} \supset Y_A^{\ge 1} \supset Y_A^{\ge 2} \supset \cdots
\end{equation*}
is called   the {\sf EPW sequence}.
When the scheme $Y_A := Y_A^{\ge 1}$ is not the whole space $\P(V_6)$, it is a sextic hypersurface (\cite[(1.8)]{og1}) called an {\sf EPW sextic}. 
The scheme $Y_A^{\ge 2}$ is non-empty and has everywhere dimension $\ge 2$ (\cite[(2.9)]{og1}).
\end{defi}


Non-zero elements of $A$
which can be written as  $v_1\wedge v_2\wedge v_3$ are called {\sf  decomposable vectors}, so that 
$A$ 
{\sf contains no decomposable vectors} if the scheme
$$\Theta_A:=\PP(A) \cap \Gr(3,V_6)$$
is empty.   When $  \P(V_6)\ne\bigcup_{V_3\in \Theta_A }\P(V_3)$ (\eg, when $\dim(\Theta_A)\le 2$),
the scheme $Y_A$ is not $ \P(V_6)$ and its  singular locus can be described as  (\cite[Corollary 2.5]{og3})
\begin{equation}\label{singya}
\Sing(Y_A)=Y_A^{\ge 2} \cup 
 \Bigl( \bigcup_{V_3\in \Theta_A }\hskip-2mm\P(V_3)\Bigr).
\end{equation}

The following theorem gathers various results of O'Grady's. 
%

\begin{theo}[O'Grady]\label{theorem-ogrady-stratification}
Let $A \subset \bw{3}{V_6}$
be a Lagrangian subspace. If $A$  contains only finitely many decomposable vectors, 
\begin{itemize}
\item[\rm (a)] $Y_A $ is an integral normal sextic hypersurface in $\P(V_6)$;
\item[\rm (b)] $Y_A^{\ge 2} $ is a surface;
\item[\rm (c)] $ Y_A^{\ge 3} $ is finite if the scheme $\Theta_A$ is moreover reduced;
\item[\rm (d)] $ Y_A^{\ge 5} $ is empty.
\end{itemize}
If moreover $A$ contains no decomposable vectors,
\begin{itemize}
\item[\rm (b$^\prime$)] $Y_A^{\ge 2} = \Sing(Y_A )$ is an integral normal  Cohen--Macaulay surface of degree $40$;
\item[\rm (c$^\prime$)] $Y_A^{\ge 3} = \Sing(Y_A^{\ge 2})$ is  finite and smooth,  and is empty for  $A$ general;
\item[\rm (d$^\prime$)] $Y_A^{\ge 4} $  is empty.
\end{itemize}
\end{theo}

{Note that} if $A $ is the dual Lagrangian (see Section \ref{depw}) associated with a strongly smooth ordinary GM curve, $A$ contains 5 decomposable vectors  and $Y_A^{\ge 4}$ is non-empty (Remark \ref{rem39}).

\begin{proof}
The finite union $S:= \bigcup_{V_3\in \Theta_A }\P(V_3)$  
is a surface if $\Theta_A$ is non-empty (finite)  and it is empty otherwise. Relation \eqref{singya} holds, hence (a) will follow from (b).

 {As we already mentioned, the scheme $Y_A^{\ge 2}$ has everywhere dimension $\ge 2$ (\cite[(2.9)]{og1}).}
We now  show that  $Y_A^{\ge 2}$ has dimension $\le 2$ at any point $v\notin S$. If $v\in Y_A^2$, this is  \cite[Proposition 2.9]{og3}.  
Assume $v\in  Y_A^{\ge 3}$; the proof of  \cite[Claim 3.7]{og4} still applies because $v\notin S$. It gives $v\in  Y_A^3$ and $v$ isolated in $ Y_A^{\ge 3}$, which proves that $Y_A^{\ge 3} \setminus S$ is finite.
This implies that in any positive-dimensional component $T$ of $Y_A^{\ge 2}\setminus S$, the open 
subset $T\cap Y_A^2$ is dense and has dimension $\le 2$. Therefore,  $\dim(T)\le 2$ and (b) is proved. 

To finish the proof of c), we show that the set $Y_A^{\ge 3} \cap\P(V_3)$ is finite for each $V_3\in \Theta_A$, under the additional assumption that $\Theta_A$ is reduced. 
Since
\begin{equation*}
\T_{\Theta_A,[V_3]}=\P(A)\cap \T_{\Gr(3,V_6),[V_3]}= \P(A\cap(\bw2V_3\wedge V_6))\subset \P(\bw3V_6),
\end{equation*}
we have $A \cap (\bw2V_3 \wedge V_6) = \bw3V_3$. 
In the notation of \cite[Definition~3.3.3]{og5}, this implies
\begin{equation*}
\cB(V_3,A) = \{ [v] \in \P(V_3) \mid  \exists V_3'\in\Theta_A\ \ V'_3\ne V_3 \ \textnormal{and}\  v\in V'_3 \}.
\end{equation*}
But this set is finite: if $V_3,V_3'\in\Theta_A$ and $\dim(V_3\cap V'_3)\ge 2$, the line spanned by $[V_3]$ and $[V'_3]$ in $\P(\bw3V_6)$ is contained in $\Gr(3,V_6)$, hence in $\Theta_A$,  contradicting the finiteness of $\Theta_A$.  
The finiteness of $ Y_A^{\ge 3} \cap\P(V_3)$ then follows from \cite[Proposition~3.3.6]{og5} and c) is proved.

Assume finally $v\in Y_A^{\ge 5}$. Inside $\P(\bw2(V_6/v))=\P^9$, the linear space $\P \bigl( A\cap (v \wedge\bw{2}{(V_6/v)} ) \bigr)$ (of dimension $\ge 4$) and the Grassmaniann $\Gr(2,V_6/v)$ (of dimension 6) 
meet along a positive-dimensional locus contained in $\Theta_A$, which is absurd. 
This proves (d).

For  the proof of (b$^\prime$), (c$^\prime$), and (d$^\prime$), we have $\Theta_A=\vide$. 

Again, we have $\Sing(Y_A) = Y_A^{\ge 2}$ by~\eqref{singya}; moreover, $Y_A^2$ is smooth of the expected dimension~2 
(\cite[Proposition 2.9]{og3}), so that $\Sing(Y_A^{\ge 2})\subset Y_A^{\ge 3} $. On the other hand, \cite[Claim 3.7]{og4} says 
that (d$^\prime$) holds; the proof actually shows that the tangent space to $Y_A^{\ge 3}$ at any point 
is zero, so that is smooth. It is moreover empty for $A$ general (\cite[Claim~2.6]{og1}).
 
All this implies that determinantal locus $Y_A^{\ge 2} $ has everywhere the expected dimension,~$2$. 
Therefore, it is Cohen--Macaulay  and $  Y_A^{\ge 3} \subset \Sing(Y_A^{\ge 2})$. This completes the proof of (c$^\prime$). 
It also implies that $Y_A^{\ge 2}$ is normal since it  is Cohen--Macaulay and its singular locus is finite.   
Its degree is $40$ by \cite[(2.9)]{og1}.

It remains to show that $Y_A^{\ge 2}$ is irreducible. When $A$ is general, 
this follows from \cite[Theorem 1.1]{og7}.
To deduce irreducibility for any $A$ with no decomposable vectors, we use a standard trick. Let $(T,0)$ be 
a smooth pointed connected curve that parametrizes Lagrangians $(A_t)_{t\in T}$ with no decomposable vectors such 
that for $t\in T$ general, $Y_{A_t}^{\ge 2}$ is smooth irreducible and $A_0 = A$. 

Let $H\subset\P(V_6)$ be a general hyperplane, 
so that $H\cap Y_{A_0}^{\ge 2}$ is a smooth curve and $H\cap Y_{A_t}^{\ge 2}$ is smooth irreducible for general $t$. 
In $T\times \P^5$, consider the union  of all surfaces $\{ t\}\times Y_{A_t}^{\ge 2}$. As a determinantal locus,   
it has everywhere codimension $\le 3$, hence  its intersection $\cY$ with  $T\times H$ has everywhere dimension $\ge 2$. 
Since the fiber   $\cY_0=H\cap Y_{A_0}^{\ge 2}$   is 1-dimensional, any component of $\cY_0$ deforms to a neighborhood 
of $0$ in $T$. Since $\cY_t$ is irreducible for $t$ general, so is $\cY_0$, hence so is $Y_{A_0}^{\ge 2}$. 
This finishes the proof of the theorem.
\end{proof}

 \subsection{Duality of EPW sextics} \label{depw}
 
 If  $A \subset \bw{3}{V_6}$ is a Lagrangian subspace, its orthogonal 
$A^ \perp \subset \bw{3}{V_6^\vee}$ is also a Lagrangian subspace. In the dual projective space $\P(V_6^\vee) = \Gr(5,V_6)$, the EPW sequence    for  $A^\bot$ can be described  {in terms of $A$} as
\begin{equation}\label{dualYA}
Y_{A^\perp}^{\ge \ell} = \big\{ V_5 \in \Gr(5,V_6) \mid \dim (A \cap \bw{3}{V_5}) \ge \ell \big\}. 
\end{equation}
The canonical identification $\Gr(3,V_6) \isom \Gr(3,V_6^\vee) $ induces an isomorphism between the scheme $\Theta_A$ of 
decomposable vectors in $\P(A)$ and the scheme $\Theta_{A^\bot}$  of decomposable vectors in $\P(A^\bot)$ (\cite[(2.82)]{og3}).
In particular, $A$ contains no decomposable vectors if and only if the same is true for $A^\perp$.

One of the interesting properties of a EPW sextic is that its projective dual is often also an EPW sextic. 
  The proof we give  is essentially equivalent to O'Grady's (\cite[Corollary 3.6]{og2}) 
but is written in  more geometrical terms.

Consider the  scheme  
\begin{equation}\label{yahat}
\widehat{Y}_A := \{ (v,V_5) \in  \Fl(1,5;V_6)\mid   A \cap (v \wedge \bw2V_5) \ne 0 \}
\end{equation}
and the   projections 
\begin{equation*}
\xymatrix@R=5mm{
& \widehat{Y}_A \ar[dl]_p \ar[dr]^q \\
Y_A && Y_{A^\perp}.
} 
\end{equation*}
 
\begin{prop}\label{prop-hat-y}
If the scheme $\Theta_A$ of decomposable vectors in $\P(A)$ is finite and reduced, the hypersurfaces $Y_A \subset \P(V_6)$ and $Y_{A^\perp} \subset \P(V_6^\vee)$ are projectively dual.
If $A$ contains no  decomposable vectors, $\widehat{Y}_A$ is irreducible and realizes the projective duality between these two hypersurfaces.
\end{prop}

\begin{proof}
  We keep the notation $S = \bigcup_{V_3 \in \Theta_A} \P(V_3)$ and define
\begin{equation*}\label{yahatp}
\widehat Y'_A := \{ (a,v,V_5) \in \PP(A) \times \Fl(1,5;V_6)\mid a \in  \P(A \cap (v \wedge \bw2V_5)) \},
\end{equation*}
with the forgetful map $\widehat Y'_A \to \widehat Y_A $ and the maps $p'\colon\widehat Y'_A\to \widehat Y_A  \xrightarrow{p } Y_A$ and $q'\colon  \widehat Y'_A\to \widehat Y_A  \xrightarrow{q } Y_{A^\perp}$.

Let us determine the fiber  of $p'$ over   $v \in Y_A$. Take   $a \in A \cap (v \wedge \bw2V_6)$ non-zero and  write $a = v \wedge \eta$ with $\eta \in \bw2(V_6/v)$. 

If $a$ is not decomposable, so is $\eta$, hence   its rank   is 4 and   there is a unique hyperplane 
$V_5 \subset V_6$ containing $v$ such that $\eta \in \bw2(V_5/v)$. 
Moreover, the space $V_5$ corresponding to $a = v \wedge \eta$ is given by the 4-form 
\begin{equation*}
\eta\wedge\eta \in \bw4(V_6/v) \cong (V_6/v)^\vee = v^\perp \subset V_6^\vee.
\end{equation*}
If $a$ is decomposable, \ie,  {$a$ belongs to}   some $\bw3V_3\subset \bw3V_6$, the hyperplanes $V_5\subset V_6$ such that $a\in v \wedge \bw2V_5$ satisfy $\bw3V_3\subset \bw3V_5$, hence $V_3\subset V_5$.
This means that the natural map
\begin{equation*}
p^{\prime-1}(v) \to \P(A \cap (v\wedge \bw2V_6))
\end{equation*}
is a $\P^2$-fibration over $\P(A \cap (v\wedge \bw2V_6)) \cap \Gr(3,V_6)$ and is an isomorphism over its complement.
If $v \not\in S$, the first case does not happen, hence
\begin{equation}\label{d1}
p^{\prime-1}(v) \cong \PP(A \cap (v \wedge \bw2V_6)).
\end{equation}

Since $\widehat Y'_A$ can be defined as the zero-locus
of a section of a rank-$14$ vector bundle on the   smooth 18-dimensional variety $\P(A) \times \Fl(1,5;V_6)$,
  any irreducible
component of $\widehat Y'_A$ has dimension at least 4. On the other hand, the description \eqref{d1} of the fibers of $p'$ outside of the surface  {$S$}
implies $\dim (p^{\prime-1}(Y_A^1\moins S))=4$, and, using also Theorem \ref{theorem-ogrady-stratification},
 \begin{eqnarray*}
\dim (p^{\prime-1}(Y_A^{\ge 2}\moins S)) &\le&  
 \max\{\dim (p^{\prime-1}(Y_A^2\moins S)), \dim (p^{\prime-1}(Y_A^{3}\moins S)), \dim (p^{\prime-1}(Y_A^{4}\moins S))\} \\
 &\le &
\max\{2+1, 0+2,   0+3\}=3.
  \end{eqnarray*}
Using the description of the fibers of $p'$ over   the surface $S$, we conclude that the irreducible components of $\widehat Y'_A$ are 
\begin{itemize}
\item a ``main''   component $ \widehat Y''_A$ which dominates $Y_A$; 
\item one component $\{a\} \times \PP(V_3) \times \PP(V_3^\perp)  $ for each decomposable vector  {$a  \in \bw3 V_3 \subset  A$}.
\end{itemize}

It   now remains to prove that  $ \widehat Y''_A$ defines the projective duality between $Y_A  $ and $Y_{A^\perp}  $. 
Let  $v \in Y_A^1\moins S$ and set $V_5 := q(p^{-1}(v))$. 
Let us show that the hyperplane $\P(V_5) \subset \P(V_6)$is tangent to $Y_A$ at $v$. 
Since $p'$ is  {an isomorphism} over a neighborhood of  $v$, the tangent space $\T_{Y_A^1,v}$ is identified 
with the tangent space $\T_{\widehat Y'_A,(a,v,V_5)}$. Let $(a +  a't,v + v't,V_5+V'_5t)$ be a tangent vector to $\widehat Y'_A$,
\ie, a $\k[t]/t^2$-point of this variety, with  $a' \in A$. Since $A$ is Lagrangian,
we have
\begin{equation}\label{aa'}
a \wedge a' = 0.
\end{equation}
On the other hand,   by definition of $\widehat{Y}'_A$, we can write $a + a't = (v + v't) \wedge (\eta + \eta't)$, which gives $a = v\wedge \eta$ and $a' = v \wedge \eta' + v' \wedge \eta$.
Substituting into \eqref{aa'}, we get
\begin{equation}\label{bb}
0 = (v\wedge \eta) \wedge (v \wedge \eta' + v'\wedge \eta) = v \wedge v' \wedge \eta \wedge \eta.
\end{equation}
But $v'$ is the image of the tangent vector in $\T_{Y_A^1,v}$, while $v \wedge \eta \wedge \eta$ is the equation of $V_5$.
Equation \eqref{bb} therefore implies $v' \in V_5$. Since this holds for any $v'$, we deduce $\T_{Y_A^1,v} \subset \P(V_5)$.
Since $v$ is a smooth point of the hypersurface $Y_A$, this implies   $\T_{Y_A ,v} = \P(V_5)$,
as required. 
 \end{proof}

\begin{rema}\label{remark-hat-y-q}
Assume that $A$ contains no decomposable vectors. 
First, the forgetful map $\widehat Y'_A \to \widehat Y_A $ is  {an isomorphism}: if $(a,v,V_5),(a',v,V_5) \in \widehat Y'_A$ with $a \ne a'$, there would be decomposable  vectors in $A$ in the span of $a$ and $a'$.
 
Second, the definition of $\widehat{Y}_A$ is symmetric with respect to $V_6$ and $V_6^\vee$. In particular,  we obtain from \eqref{d1}
an identification
 $q^{-1}(V_5) \isom \P(A \cap \bw3V_5)$
and $p$ takes $a \in A \cap \bw3V_5$ to the kernel of $a$ considered as a skew form on $V_5$ via the   isomorphism $\bw3V_5 \cong \bw2V_5^\vee$.
\end{rema}

 The following two lemmas were used in the article.

\begin{lemm}\label{lemma-y2y2-yhat}
If $A$ contains no decomposable vectors, we have 
\begin{equation*}
(Y_A^2 \times Y_{A^\perp}^2) \cap \Fl(1,5;V_6) \not\subset (p\times q)(\widehat{Y}_A).
\end{equation*}
\end{lemm}

\begin{proof}
Assume by contradiction $(Y_A^{\ge 2} \times Y_{A^\perp}^{\ge 2}) \cap \Fl(1,5;V_6) \subset (p\times q)(\widehat{Y}_A)$. 
The left side is the intersection in $\P(V_6) \times \P(V_6^\vee) $  of the fourfold $Y_A^{\ge 2} \times Y_{A^\perp}^{\ge 2}   $ 
with the hypersurface $\Fl(1,5;V_6)  $, hence has dimension everywhere at least 3. 
On the other hand, it   sits in $E:=p^{-1}(Y_A^{\ge 2})$,   the exceptional divisor of the birational morphism $p$ and it follows from the proof of Proposition~\ref{prop-hat-y} that $E$ is irreducible of dimension 3. Therefore,
  we have 
\begin{equation*}
(Y_A^{\ge 2} \times Y_{A^\perp}^{\ge 2}) \cap \Fl(1,5;V_6)   = E.
\end{equation*}
 By symmetry, we also have 
\begin{equation*}
(Y_A^{\ge 2} \times Y_{A^\perp}^{\ge 2}) \cap \Fl(1,5;V_6) = q^{-1}(Y_{A^\perp}^{\ge 2}) =: E',
\end{equation*}
  the exceptional divisor of the birational morphism $q$. It follows that $E = E'$.

Denote by $H$ and $H'$ the respective restrictions to $Y_A$ and $Y_{A^\perp}$ of the hyperplane classes from $\P(V_6)$ and $\P(V_6^\vee)$. 
The birational isomorphism $q\circ p^{-1} \colon Y_A \dashrightarrow Y_{A^\perp}$ is given by the Gauss map (\ie, by the partial derivatives of the equation of $Y_A$). 
Since the equation of $Y_A$ has degree~6 and any point of $Y_A^2$ has multiplicity 2 on $Y_A$ (\cite[Corollary 2.5]{og3}),  {the partial derivatives have degree 5 and multiplicity 1, and} we get
\begin{equation*}
H' \lin 5H - E.
\end{equation*}
By symmetry, we also have
$
H \lin 5H' - E'$.
Combining these two equations, we obtain
\begin{equation*}
E' \lin 24H - 5E.
\end{equation*}
From $E' = E$, we obtain $24H \lin 6E$, which is a contradiction, since $E$ has negative intersection
with curves contracted by $p$,
while $H$ has zero intersection with such curves.
\end{proof}

\begin{lemm}\label{lemma-y2-hyperplane}
If $A$ contains no decomposable vectors, the scheme $Y^{\ge 2}_A \subset \P(V_6)$ is not contained in a hyperplane.
\end{lemm}

\begin{proof}
We   use the notation introduced in the proof of Lemma~\ref{lemma-y2y2-yhat}. If $Y^{\ge 2}_A$ is contained in a hyperplane, 
the linear system $|H-E|$ is non-empty, hence the linear system 
\begin{equation*}
 |E'| = |24H-5E| = |19H + 5(H-E)| 
\end{equation*}
is movable. But $E'$ is the exceptional divisor of the birational morphism $q$, hence  $E'$  is rigid, so this is a contradiction.
\end{proof}

  \subsection{Double EPW sextics} 

The properties of EPW sextics that are of most interest come
from the existence of a (finite) canonical double cover $f_A\colon\widetilde Y_A\to Y_A$ (\cite[Section~1.2]{og4}) with the following properties.

\begin{theo}[O'Grady]\label{doubleepw}
Let $A \subset \bw{3}{V_6}$
be a Lagrangian subspace which contains no decomposable vectors and let $Y_A\subset \P(V_6)$ be the associated EPW sextic.
\begin{itemize}
\item[\rm (a)] The double cover $f_A\colon\widetilde Y_A\to Y_A$ is branched over the surface  $Y_A^{\ge 2}$ and induces  the universal cover of $Y_A^1$.
\item[\rm (b)] The variety $\widetilde Y_A $ is irreducible and normal, and its singular locus is the finite set $f_A^{-1}(Y_A^{\ge 3})$.
\item[\rm (c)] When $Y_A^{\ge 3}$ is empty, $\widetilde Y_A $ is a smooth   hyperk\"ahler fourfold which is a deformation of the symmetric square of a K3 surface.
\end{itemize}
\end{theo}

\begin{proof}
Item (a) is proved in \cite[Proof of Theorem 4.15, p.~179]{og4}, 
item (b) follows from statement (3) in the introduction of \cite{og4}, and item (c) is \cite[Theorem~1.1(2)]{og1}.
\end{proof}

We draw a consequence of this theorem.

\begin{prop}\label{prop33}
Let $A\subset \bw3V_6$ and $A'\subset \bw3V'_6$ be Lagrangian subspaces and let
$Y_A\subset \P(V_6)$ and $Y_{A'}\subset \P(V'_6)$ be the schemes defined in \eqref{yabot}.
  \begin{itemize}
\item[\rm (a)]
Any linear isomorphism $\varphi\colon V_6\isomlra V'_6$  such that $(\bw3\varphi)( A)=A'$ induces an isomorphism $Y_A\isomlra Y_{A'}$.
\item[\rm (b)]
Assume $Y_A\ne \P(V_6)$. Any isomorphism $Y_A\isomlra Y_{A'}$ is induced by a linear isomorphism $\varphi\colon V_6\isomlra V'_6$ and, if
 $A$ contains no decomposable vectors, $A'$ contains no decomposable vectors and $(\bw3\varphi)( A)=A'$.
\end{itemize}
\end{prop}

\begin{proof} 
Item (a) follows from the definition of $Y_A$.

(b) If $Y_A\ne \P(V_6)$, it follows from the Lefschetz theorem  \cite[Expos\'e XII, cor.~3.7]{sga2}  that $\Pic(Y_A)$ is generated by the hyperplane class. Any isomorphism $Y_A\isomlra Y_{A'}$ is therefore induced by a linear isomorphism $\varphi\colon V_6\isomlra V'_6$.

If   $A$ contains no decomposable vectors, $\Sing(Y_A)$ is an integral surface of degree 40 (Theorem \ref{theorem-ogrady-stratification}) hence  
contains no planes, and the same holds for $\Sing(Y_{A'})$. It follows from \eqref{singya} that $A'$ contains 
no decomposable vectors. Applying Theorem \ref{theorem-ogrady-stratification} again, we obtain $\varphi (Y_A^1) = Y_{A'}^1$. This isomorphism lifts to the universal covers 
 to an isomorphism $f_A^{-1}(Y_A^1)\isomto f_{A'}^{-1}(Y_{A'}^1)$, and extends to the normal completions  to an isomorphism $\widetilde Y_A\isomto \widetilde Y_{A'}$ (Theorem \ref{doubleepw}).
By \cite[Proof of Proposition 1.0.5]{og5}, 
this implies $(\bw3\varphi)(A) = A' $. 
\end{proof}

Finally, we   characterize the automorphism group of an EPW sextic.

\begin{prop}[O'Grady]\label{autoepw}
Let $A\subset \bw3V_6$  be a Lagrangian subspace which contains no decomposable vectors  and let $Y_A\subset \P(V_6)$ be the associated EPW sextic.
 \begin{itemize}
\item[\rm (a)] 
The automorphism group of $Y_A$ is finite and equal to the group
\begin{equation}\label{pgla}
\PGL(V_6)_{A} := \{ g \in \PGL(V_6) \mid  (\bw3g)(A) = A \}.
\end{equation}
\item[\rm (b)]
When $A$ is very general, these groups are trivial.
\end{itemize}
\end{prop}

\begin{proof} 
(a) The equality $\Aut(Y_A)=\PGL(V_6)_{A} $ follows from Proposition \ref{prop33}. 
 These groups are finite because $A$ is a stable point of $\LGr(\bw{3}{V_6})$ for the action of $\PGL(V_6)$: the non-stable locus has 12 components, listed in \cite[Table 1]{og5}, and one checks that they are all contained in  the locus of Lagrangians 
which contain   decomposable vectors. 

(b) Let $f_A\colon \widetilde Y_A\to Y_A$ be the canonical double cover, where $\widetilde Y_A$ is a smooth hyperk\"ahler fourfold (Theorem~\ref{doubleepw}(c)). As explained in the proof of Proposition~\ref{prop33}(b), any automorphism $g$ of $Y_A$ is linear and lifts to an automorphism $\tilde g$ of $\widetilde Y_A$.

Let $h$ be the class in $H^2(\widetilde Y_A,\Z)$ of $f_A^*(\cO_{Y_A}(1))$.
 The automorphism $\tilde g$  
 fixes $h$ hence acts   on its orthogonal $h^{\bot}\subset H^2(\widetilde Y_A,\C)$ for the Beauville--Bogomolov form as a linear map which we denote by $\tilde g^*$. The line
 $H^{2,0}(\widetilde Y_A)\subset H^2(\widetilde Y_A,\C)$ is contained in an eigenspace of $\tilde g^*$ and the small
 deformations $A_t$
 of $A$  for which $g$ extends
 must maintain $H^{2,0}(\widetilde Y_{A_t})$ in that same eigenspace.
 It follows that if $g$ extends for all small deformations of $A$, the eigenspace containing
 $H^{2,0}(\widetilde Y_A)$ is all of $h^{\bot}$. Thus $\tilde g^*=\zeta \Id_{h^{\bot}}$, where $\zeta$ is some root of unity. Since $\tilde g^*$ is real, we have $\zeta=\pm 1$.

But  the morphism
 $\Aut(\widetilde Y_A) \to \Aut(H^2(\widetilde Y_A,\Z))$ is injective: its kernel is deforma\-tion  invariant (\cite[Theorem 2.1]{hats}) and it is trivial for symmetric squares of K3 surfaces (\cite[prop.~10]{bea}); the statement therefore follows from Theorem~\ref{doubleepw}(c).
This implies that   $\tilde g$ is either the identity or the involution associated with the double covering $f_A$. In both cases, $g=\Id_{Y_A}$.
\end{proof}

\subsection{EPW quartics}
We   introduce certain quartic hypersurfaces in $\Gr(3,V_6)$ also associated with Lagrangian subspaces $A \subset \bw3V_6$. For each $\ell \ge 0$, we set
\begin{equation}\label{equation-epw-quartic}
Z^{\ge \ell}_{A} := \big\{ V_3 \in \Gr(3,V_6) \mid \dim(A \cap ( V_6 \wedge \bw{2}{V_3} ) ) \ge \ell \big\}.
\end{equation}

The subschemes $Z^{\ge \ell}_{A} \subset \Gr(3,V_6)$ were recently studied in \cite{ikkr}. In particular, if $A$ contains no decomposable vectors, the scheme $Z_{A} = Z^{\ge 1}_{A}$ is an integral quartic hypersurface (\cite[Corollary 2.10]{ikkr}).
We will call it the {\sf EPW quartic} associated with $A $.
We   will only need the following duality property.

\begin{lemm}\label{lem39}
The natural isomorphism $\Gr(3,V_6) \cong \Gr(3,V_6^\vee)$ induces an identification between the EPW quartics 
$Z_{A} \subset \Gr(3,V_6)$ and $Z_{A^\perp} \subset \Gr(3,V_6^\vee)$.
\end{lemm}

\begin{proof}
The isomorphism of the Grassmannians takes $V_3 \subset V_6$ to $V_3^\perp \subset V_6^\vee$.
It is therefore enough to show that the isomorphism $\bw3V_6 \to \bw3V_6^\vee$ given by the symplectic
form takes the subspace $V_6\wedge \bw2V_3$ to the subspace $V_6^\vee\wedge\bw2V_3^\perp$.
As both subspaces are Lagrangian, it is enough to verify that the pairing between 
$V_6\wedge \bw2V_3$ and $V_6^\vee\wedge\bw2V_3^\perp$ induced by the pairing between $V_6$ and $V_6^\vee$
is trivial, and this is straightforward.
\end{proof}

 \section{Projective duality of quadrics and Lagrangian subspaces}\label{section-lagrangians-and-quadrics}

 {Let $V$ be a $\k$-vector space. We discuss projective duality of quadrics  (of all possible dimensions)  in $\P(V)$
and its interpretation in terms of Lagrangian geometry. To simplify the statements, it is convenient to define 
a quadric $Q$ in $\P(V)$ as a subvariety of a (possibly empty) linear subspace $\P(W) \subset \P(V)$ defined by a (possibly zero) quadratic form $q \in \Sym^2W^\vee$
(if $q = 0$, then $Q = \P(W)$ is a {\em linear} subspace of $\P(V)$; it should not be confused with rank-1 quadrics, which   are {\em double} linear subspaces).
With each such form $q$, we can associate   its kernel space $K \subset W$. The induced form on the (possibly zero) quotient space $W/K$ is then non-degenerate.}

The projective dual variety $Q^\vee \subset \P(V^\vee)$ is also a quadric which can be constructed as follows.
The quadratic form $q$ induces an isomorphism
\begin{equation*}
q\colon W/K \isomlra (W/K)^\vee = K^\perp/W^\perp,
\end{equation*}
where $W^\perp \subset K^\perp \subset V^\vee$ are the orthogonals.  
The inverse isomorphism 
\begin{equation*}
q^{-1}\colon K^\perp/W^\perp \isomlra W/K = (K^\perp/W^\perp)^\vee
\end{equation*}
defines a quadratic form $q^\vee$
on $K^\perp \subset V^\vee$ with kernel $W^\perp$.
The corresponding quadric in $\P(K^\perp) \subset \P(V^\vee)$ is the projective dual $Q^\vee$ of $Q$.

It is a classical observation that projective duality of quadrics can be   described
in terms of Lagrangian subspaces in a symplectic vector space equipped with a Lagrangian
direct sum decomposition. For the reader's convenience, we summarize this relation   
and develop it a bit.

We recall how a Lagrangian subspace gives a pair of projectively dual quadrics.
 
\begin{prop}\label{lemma-lagrangian-quadratic}
Let $(\VV,\omega)$ be a  symplectic $\k$-vector space and let $\LL_1$ and $\LL_2  $ be Lagrangian subspaces such that
\begin{equation*}
\VV = \LL_1 \oplus \LL_2.
\end{equation*}
Denote by $\pr_1$ and $\pr_2 \in \End(\VV)$
 the projectors to $\LL_1$ and $\LL_2$.

{\rm (a)} For any Lagrangian subspace $A \subset \VV$, the bilinear form $q^A$ on $A$ defined by
\begin{equation}\label{equation-quadric-s}
q^A(x,y) := \omega(\pr_1(x),\pr_2(y))
\end{equation}
  is symmetric.
 Its kernel is given by
\begin{equation}\label{equation-kernel-qs}
\Ker (q^A) =  (A \cap \LL_1  ) \oplus  (A \cap \LL_2  ).
\end{equation} 

{\rm (b)}   The quadratic form $q^A$ induces quadratic forms $q^A_1$ and $q^A_2$ on the subspaces
\begin{equation*}
W_1 := \pr_1(A) = (A \cap \LL_2)^\perp \subset \LL_1
\qquad\text{and}\qquad 
W_2 := \pr_2(A) = (A \cap \LL_1)^\perp \subset \LL_2
\end{equation*}
with respective kernels
\begin{equation*}
 {K_1 := } A \cap \LL_1 \subset W_1
\qquad\text{and}\qquad 
 {K_2 := } A \cap \LL_2 \subset W_2.
\end{equation*}
The corresponding quadrics $Q^A_1 \subset \P(W_1) \subset \P(\LL_1)$ and $Q^A_2\subset \P(W_2) \subset \P(\LL_2)$
are projectively dual with respect to the duality between $\LL_1$ and $\LL_2$ induced by the symplectic form~$\omega$.
\end{prop}

\begin{proof}
(a)
 For all $x, y \in A$, we have 
\begin{multline*}
q^A(x,y) - q^A(y,x) 
 = \omega(\pr_1(x),\pr_2(y)) + \omega(\pr_2(x),\pr_1(y)) 
\\
= \omega(\pr_1(x) + \pr_2(x),\pr_1(y) + \pr_2(y)) - \omega(\pr_1(x),\pr_1(y)) - \omega(\pr_2(x),\pr_2(y)).
\end{multline*}
In this last expression, the first term equals $\omega(x,y)$, hence is zero since $A$ is Lagrangian; the second and the third terms are zero since $\LL_1$ and $\LL_2$ are Lagrangian.
Thus the expression vanishes  and   $q^A$ is symmetric.

It remains to compute the kernel of $q^A$. Since $x = \pr_1(x) + \pr_2(x)$ and $\LL_2$ is Lagrangian, we have
\begin{equation*}
q^A(x,y) = \omega(\pr_1(x),\pr_2(y)) = \omega(x - \pr_2(x),\pr_2(y)) =\omega(x,\pr_2(y)).
\end{equation*}
Since $A$ is Lagrangian, its $\omega$-orthogonal coincides with $A$, so 
$y \in \Ker (q^A)$ is equivalent to $\pr_2(y) \in A$. Writing $y = \pr_1(y) + \pr_2(y)$, this implies 
that also $\pr_1(y) \in A$. Thus $\pr_1(y) \in A \cap \LL_1$ and $\pr_2(y) \in A \cap \LL_2$.
This proves one inclusion. For the other inclusion, both $A \cap \LL_1$ and $A \cap \LL_2$
are in the kernel of $q^A$: for $A \cap \LL_1$, this is because $\pr_2(y) = 0$ for  $y \in A \cap \LL_1$,
and for $A \cap \LL_2$, because   $\pr_1(x) = 0$ for $x \in A \cap \LL_2$.

(b) By~\eqref{equation-quadric-s}, the kernels $A \cap \LL_2$ and $A \cap \LL_1$ of the restrictions $\pr_1\!\vert_A$ and $\pr_2\!\vert_A$
are contained in the kernel of $q^A$, hence   $q^A$ induces quadratic forms 
on  $W_1 = \pr_1(A) \subset \LL_1$ and $W_2 = \pr_2(A) \subset \LL_2$. The equalities $W_1 = (A \cap \LL_2)^\perp$
and $W_2 = (A \cap \LL_1)^\perp$ follow from the Lagrangian property of~$A$. The kernels $K_1$ and $K_2$
of the induced quadratic forms $q^A_1$ and $q^A_2$ are the  respective images of the kernel of~$q^A$ under $\pr_1$ and $\pr_2$, so~\eqref{equation-kernel-qs} implies $K_1 = A \cap \LL_1$ and $K_2 = A \cap \LL_2$.

Note that $A/(K_1 \oplus K_2) \subset W_1/K_1 \oplus W_2/K_2$
is the graph of an isomorphism $W_1/K_1 \cong W_2/K_2$. The duality between
$W_1/K_1$ and $W_2/K_2$ given by the symplectic form~$\omega$ identifies this isomorphism
with the quadratic form $q^A_1$  and its inverse with $q^A_2$. This means that the quadrics
$Q^A_1$ and $Q^A_2$ are projectively dual.
 \end{proof}

The following construction shows that any pair of projectively dual quadrics comes from a Lagrangian subspace.

\begin{lemm}\label{lemma-quadric-lagrangian}
Let $Q \subset \P(\LL)$ be a quadric and let $Q^\vee \subset \P(\LL^\vee)$ be its projective dual.
In the   space $\VV := \LL \oplus \LL^\vee$ endowed with the symplectic form 
\begin{equation*}
\omega(x,y) = (x_1,y_2) - (y_1,x_2),
\end{equation*}
there is a unique Lagrangian subspace $A \subset \VV$ such that $Q^A_1 = Q$ and $Q^A_2 = Q^\vee$.
\end{lemm}

\begin{proof}
 Let $K \subset \LL$ and $W \subset \LL$ be the kernel and span of $Q$  and let  $q \in \Sym^2 \!W^\vee$ be a quadratic form defining $Q$. 
 Let $A$ be the kernel of the sum of   the induced map $q\colon W \to W^\vee$  with the canonical surjection
$\LL^\vee \to W^\vee$. It fits into an exact sequence
\begin{equation}\label{eqlem}
0 \to A \to W \oplus \LL^\vee \to W^\vee \to 0.
\end{equation}
The embedding $W \to \LL$ induces an embedding $A \to \LL \oplus \LL^\vee = \VV$. Let us show the image is Lagrangian.
It can be explicitly described as
\begin{equation*}
A = \{ (x_1,x_2) \in \LL \oplus \LL^\vee \mid x_1 \in W \quad\text{and}\quad x_2 \equiv q(x_1) \bmod W^\perp \}.
\end{equation*}
Given    $x = (x_1,x_2)$ and $y = (y_1,y_2)$ in $A$, we have
\begin{equation*}
\omega(x,y) = (x_1,y_2) - (y_1,x_2) = (x_1,q(y_1)) - (y_1,q(x_1)) = q(y_1,x_1) - q(x_1,y_1) = 0,
\end{equation*}
so $A$ is indeed Lagrangian. Note that $\pr_1(A) = W$ by definition of $A$. Moreover, if $q^A$ 
is the quadratic form on $A$ defined in Lemma~\ref{lemma-lagrangian-quadratic},  
for $x = (x_1,x_2)$ and $y = (y_1,y_2)$ in $A$, we have
\begin{equation*}
q^A(x,y) = \omega(\pr_1(x),\pr_2(y)) = \omega(x_1,y_2) = (x_1,q(y_1)) = q(x_1,y_1) = q(\pr_1(x),\pr_1(y)).
\end{equation*}
This shows $q^A_1 = q$, so $Q^A_1 = Q$. Since $Q^A_2$ is projectively dual to $Q^A_1$, we obtain $Q^A_2 = Q^\vee$.

For the uniqueness, it is enough to check that the construction of this lemma is inverse to the construction
of Proposition~\ref{lemma-lagrangian-quadratic}, which is straightforward.
\end{proof}

We now discuss the   so-called
``isotropic reduction'' of Lagrangian subspaces.

Let $(\VV,\omega)$ be a symplectic vector space. Let $\II \subset \VV$ be an isotropic subspace and let $\II^\perp$ be its $\omega$-orthogonal.
The restriction of the symplectic form $\omega$ to $\II^\perp$ is degenerate and its kernel   is   $\II \subset \II^\perp$.
Hence the form descends to a symplectic form $\bar\omega$ on   $\BVV := \II^\perp/\II$. Moreover,
for each Lagrangian subspace $\LL \subset \VV$, the subspace
\begin{equation*}
\BLL := (\LL \cap \II^\perp)/(\LL \cap \II) \subset \BVV
\end{equation*}
is Lagrangian. We call the symplectic space $\BVV$  the {\sf isotropic reduction} of $\VV$,
and the Lagrangian subspace $\BLL$  the {\sf isotropic reduction} of $\LL$.

Let $\VV = \LL_1 \oplus \LL_2$  be a Lagrangian direct sum decomposition and let $\II \subset \LL_1$   (automatically isotropic). Isotropic reduction
  gives us a Lagrangian direct sum decomposition
\begin{equation*}
\BVV = \BLL_1 \oplus \BLL_2,
\end{equation*}
where $\BLL_1 = \LL_1/\II$ and  $\BLL_2 = \LL_2 \cap \II^\perp$. 
Furthermore, if $A \subset \VV$ is Lagrangian,  {its isotropic reduction $\BA \subset \BVV$}
produces a projectively dual pair of quadrics $(Q^\BA_1,Q^\BA_2)$. We will say that this pair is obtained
from the pair $(Q^A_1,Q^A_2)$ by  {\sf isotropic reduction with respect to} $\II$.
The next lemma gives a geometric relation between a pair of projectively dual quadrics
and their isotropic reduction.

\begin{lemm}\label{lemma-isotropic-reduction}
In the situation described above, we have the following.

{\rm (a)} The quadratic form $q^\BA$ on $\BA$ is induced by the restriction of the  quadratic form $q^A$ to 
$A \cap \II^\perp \subset A$. In particular,
\begin{equation*}
Q^A \cap \PP(A \cap \II^\perp) = \cone{A\cap \II}(Q^\BA).
\end{equation*}

{\rm (b)}  The quadric $Q^\BA_2$ is the  linear section $  Q^A_2 \cap \PP(\BLL_2)$ of $Q^A_2$ 
and the quadric $Q^\BA_1$ is its projective dual. 
 Moreover, the span and the kernel of $Q^\BA_2$ are
\begin{equation}\label{span-ker}
\mathop{\Span}(Q^\BA_2) = ((A \cap \LL_1)/(A \cap \II))^\perp \subset \BLL_2,
\quad
\Ker(Q^\BA_2) = (A \cap (\II \oplus \BLL_2))/(A \cap \II).
\end{equation}  
\end{lemm}

\begin{proof}
(a) First, we have $A \cap \II \subset A \cap \LL_1 \subset \Ker (q^A)$.
Thus $q^A$ induces a quadratic form on $\BA$. Moreover, the projectors $\overline\pr_1$ and $\overline\pr_2$ 
onto the components of the decomposition $\BVV = \BLL_1 \oplus \BLL_2$ are induced 
by the projectors $\pr_1$ and $\pr_2$. Hence, for $x,y \in A \cap \II^\perp$, we have
\begin{equation*}
q^\BA(x,y) = \overline\omega(\overline\pr_1(x),\overline\pr_2(y)) = \omega(\pr_1(x),\pr_2(y)) = q^A(x,y).
\end{equation*}
This means that the quadratic form induced on $\BA$ is $q^\BA$
and thus proves the first equality. 

(b) We prove that $Q^\BA_2$ is this linear section of $Q^A_2$. Since $\LL_1 \subset \II^\perp$, we have 
\begin{equation*}
\BW_2 := \overline\pr_2(\BA) = \pr_2(A \cap \II^\perp) = \pr_2(A) \cap \pr_2(\II^\perp) = W_2 \cap \BLL_2
\end{equation*}
and the kernel of the projection $A \cap \II^\perp \to \pr_2(A \cap \II^\perp)$ is $A \cap \LL_1$. Therefore we have a commutative diagram
\begin{equation*}
\xymatrix@R=5mm{
0 \ar[r] & A \cap \LL_1 \ar[r] \ar[d] & A \cap \II^\perp \ar[r]^{\pr_2} \ar@{->>}[d] & W_2 \cap \BLL_2 \ar@{=}[d] \ar[r] & 0\\
0 \ar[r] & \BA \cap \BLL_1 \ar[r] & \BA \ar[r]^{\overline\pr_2} & \BW_2 \ar[r] & 0.
}
\end{equation*}
Since $A \cap \LL_1$ is contained in the kernel of the quadratic form $q^A\vert_{A \cap \II^\perp}$,  the quadratic form
induced by it on $W_2 \cap \BLL_2$ via the projection $\pr_2$ coincides with the quadratic form obtained in two steps, 
by first inducing a quadratic form on $\BA$ via the middle vertical arrow, and then on $\BW_2$ via $\overline\pr_2$. 
The quadratic form induced via $\pr_2$ is $q^A_2\vert_{W_2\cap\BLL_2}$, the quadratic form induced via the vertical map is $q^\BA$,
and the quadratic form induced by it via $\overline\pr_2$ is $q^\BA_2$. 
Therefore, by the commutativity of the right square, we deduce $q^A_2\vert_{W_2\cap\BLL_2} = q^\BA_2$. 

We now identify the span and the kernel of $Q^\BA_2$. By Proposition~\ref{lemma-lagrangian-quadratic},
the span is  $(\BA \cap \BLL_1)^\perp$ and the kernel is   $\BA \cap \BLL_2$. To describe the first, 
consider the  commutative diagram above. Since the middle vertical map is surjective with kernel $A \cap \II$, 
 the same is true for the left vertical map, hence
$\BA \cap \BLL_1 = (A \cap \LL_1) / (A \cap \II)$. 

For the second, recall that $\BA$ is the image of $A \cap \II^\perp$ under the projection $\II^\perp \to \BVV$ with kernel $\II$.
Therefore, $\BA \cap \BLL_2$ is the image in $\BVV$ of the intersection of $A \cap \II^\perp$ with the preimage of $\BLL_2$ in $\II^\perp$, \ie, with $\II \oplus \BLL_2$.
Thus it is the image in $\BVV$ of $A \cap (\II \oplus \BLL_2)$. Since the kernel of the projection $\II^\perp \to \BVV$ is $\II$, we   see that
$\BA \cap \BLL_2 = (A \cap (\II \oplus \BLL_2))/(A \cap \II)$. 
 \end{proof}

Choosing an isotropic space appropriately, we can realize any subquadric in $Q_2$ as the result 
of an isotropic reduction applied to the pair $(Q_1,Q_2)$. 
This gives a convenient way to control the span and the kernel of a subquadric.

\begin{coro}\label{corollary-quadric-restriction}
Let $(Q_1,Q_2)$ be a projectively dual pair of quadrics corresponding to a Lagrangian decomposition $\VV = \LL_1 \oplus \LL_2$
and a Lagrangian subspace $A \subset \VV$. Let $\BLL_2 \subset \LL_2$ be an arbitrary subspace and set
\begin{equation*}
\II := \LL_1 \cap \BLL_2^\perp.
\end{equation*}
The isotropic reduction with respect to $\II$ produces a projectively dual pair of quadrics $(\overline{Q}_1,\overline{Q}_2)$,
where $\overline{Q}_2 = Q_2 \cap \P(\BLL_2)$. In particular, the span and the kernel of $\overline{Q}_2$ are given by~\eqref{span-ker}.
\end{coro}
\begin{proof}
Since the pairing between $\LL_1$ and $\LL_2$ is non-degenerate and $\II$ is the orthogonal of $\BLL_2$ in $\LL_1$,
it follows that $\BLL_2$ is the orthogonal of $\II$ in $\LL_2$, \ie, $\BLL_2 = \LL_2 \cap \II^\perp$, and Lemma~\ref{lemma-isotropic-reduction} applies.
\end{proof}

 To conclude, we  discuss a family version of the correspondence discussed in this appendix.
Let $S$ be a scheme, let $\cV$ be a vector bundle on $S$ equipped with a symplectic form $\omega\colon \bw2\cV \to \cM$ with values in a line bundle $\cM$,
and let $\cV = \cL_1 \oplus \cL_2$ be a Lagrangian direct sum decomposition. It induces an isomorphism
\begin{equation*}
\cL_2^\vee \cong \cL_1 \otimes \cM^{-1}. 
\end{equation*}
Let   $\cA \subset \cV$ be   another Lagrangian subbundle.
Formula~\eqref{equation-quadric-s} then defines a quadratic form $q_\cA\colon \cM^{-1} \to \Sym^2\cA^\vee$ on $\cA$, but 
 if we want to consider one of the two projectively dual quadrics,
we should impose some constant rank condition  to ensure that the linear span of the quadric in question is a vector bundle.

\begin{lemm}
Assume the morphism   $\cA \hookrightarrow \cV \xrightarrow{\ \pr_2\ } \cL_2$ has constant rank and let $\cW_2 \subset \cL_2$ be its image. Then the family of quadrics $q_\cA$ induces a family
$q_{\cA,2}\colon \cM^{-1} \to \Sym^2\!\cW_2^\vee$ of quadrics on $\cW_2$  and there are isomorphisms of sheaves
\begin{equation*}\label{equation:cokernels-iso}
\begin{aligned}
\Coker(\cW_2 \otimes \cM^{-1} \xrightarrow{\ q_{\cA,2}\ } \cW_2^\vee) & \cong 
\Coker(\cA \otimes \cM^{-1} \xrightarrow{\ \pr_1\ } \cL_2^\vee) \\ 
& \cong
\Coker(\cL_2 \otimes \cM^{-1} \xrightarrow{\ \pr_1^\vee\ } \cA^\vee)
\end{aligned}
\end{equation*}
 {and similarly for the kernels.}
\end{lemm}

\begin{proof}
By definition of $\cW_2$, there are epimorphisms $\cA \twoheadrightarrow \cW_2$  
and  $\cL_1 \otimes \cM^{-1} \cong \cL_2^\vee \twoheadrightarrow \cW_2^\vee$. 
They combine into a commutative diagram
\begin{equation*}
\xymatrix{
\cA \otimes \cM^{-1} \ar[d]^{\pr_1} \ar@{->>}[r] &
\cW_2 \otimes \cM^{-1} \ar[d]^{q_{\cA,2}} 
\ar@{^{(}->}[r] & 
\cL_2 \otimes \cM^{-1} 
 \ar[d]^{\pr_1^\vee}
\\
 \cL_2^\vee
\ar@{->>}[r] &
\cW_2^\vee 
\ar@{^{(}->}[r] & 
\cA^\vee,
}
\end{equation*}
where the kernels and   cokernels of the horizontal compositions are identified by the maps $\pr_1$ and $\pr_1^\vee$ (the kernels also can be characterized as 
the intersection of the subbundles $\cA$ and $\cL_1$ inside $\cV$, twisted by $\cM^{-1}$, and the cokernels are dual up to a twist to the kernels).
It follows that the cokernels of the vertical maps $\pr_1$, $q_{\cA,2}$, and $\pr_1^\vee$   are isomorphic, hence the lemma.
 \end{proof}

We also discuss a family version of   isotropic reduction. Assume that in addition to the above, we are given a vector subbundle $\cI \subset \cL_1$.
The vector bundle $\overline\cV$ defined as the cohomology of the monad $\cI \hookrightarrow \cV \twoheadrightarrow \cI^\vee\otimes \cM$ then has a natural 
symplectic structure  and if $\cL \subset \cV$ is a Lagrangian subbundle {\em such that the composition $\cL \hookrightarrow \cV \twoheadrightarrow \cI^\vee \otimes \cM$
has constant rank,}
 the sheaf $\overline{\cL} := \Ker(\cL \to \cI^\vee \otimes \cM)/\Ker(\cI \to \cL^\vee \otimes \cM)$
is locally free and  is a Lagrangian subbundle in $\overline{\cV}$. In particular, when applied to $\cL = \cL_1$ and $\cL = \cL_2$, we get
\begin{equation*}
\overline\cL_1 = \cL_1/\cI 
\qquad\text{and}\qquad 
\overline\cL_2 = \Ker(\cL_2 \twoheadrightarrow \cI^\vee \otimes \cM).
\end{equation*}

\begin{lemm}
 One has a Lagrangian direct sum decomposition $\overline{\cV} = \overline{\cL}_1 \oplus \overline{\cL}_2$.
Moreover, if $\cA \subset \cV$ is a Lagrangian subbundle such that the composition $\cA \to \cV \to \cI^\vee\otimes\cM$ has constant rank,   
the vector bundle $\overline{\cW}_2$ is a subbundle in $\cW_2$ and the quadratic form $q_{\overline{\cA},2}$ on $\overline{\cW}_2$ equals
the restriction of the quadratic form $q_{\cA,2}$. Finally, the subbundle $\wtilde\cL_2 := \cI \oplus \overline\cL_2 \subset \cL_1 \oplus \cL_2 = \cV$
is Lagrangian  and there is an exact sequence
\begin{equation*}
0 \to 
 \Coker(\overline\cW_2 \otimes \cM^{-1} \to \overline\cW_2^\vee) \to
 \Coker(\cA \otimes \cM^{-1} \to \wtilde\cL_2^\vee) \to
\Coker(\cA \otimes \cM^{-1} \to \cI^\vee) \to
0.
\end{equation*}
In particular, if the map $\cA \to \cI^\vee \otimes \cM$ is an epimorphism,   
\begin{equation*}
\Coker(\overline\cW_2 \otimes \cM^{-1} \to \overline\cW_2^\vee) \cong
\Coker(\cA \otimes \cM^{-1} \to \wtilde\cL_2^\vee) \cong
\Coker(\wtilde\cL_2 \otimes \cM^{-1} \to \cA^\vee).
\end{equation*}
\end{lemm}
\begin{proof}
Just repeat  the proof of Lemma~\ref{lemma-isotropic-reduction}. 
\end{proof}


\begin{thebibliography}{ACGH}
 
 \bibitem[A]{alexeev1991theorems}Alexeev, V.A.,
Theorems about good divisors on log Fano varieties (case of index $r>n-2$), in {\it Algebraic geometry (Chicago, IL, 1989),} 1--9,
Lecture Notes in Math. {\bf1479}, Springer, Berlin, 1991. 

 
  
  \bibitem[AH]{ah}    Arbarello, E., Harris, J.,
Canonical curves and quadrics of rank 4,
{\it Compos. Math.} {\bf 43} (1981),  145--179. 

\bibitem[ACGH]{acgh}
  Arbarello, E.,   Cornalba, M.,  Griffiths, P. A.,   Harris, J., 
{\it Geometry of algebraic curves. Vol. I,} Grundlehren der Mathematischen
  Wissenschaften {\bf 267}, Springer-Verlag, New York, 1985.

   \bibitem[B]{bea}    Beauville, A., Some remarks on K\"ahler manifolds with $c_1 = 0$,  {\em Classification of algebraic and analytic manifolds (Katata, 1982),} 1--26, Prog. Math. {\bf 39}, Birkh\"auser, Boston, MA, 1983.



 
  \bibitem[DIM]{dim} Debarre, O., Iliev, A., Manivel, L., On the period map for prime Fano threefolds of degree 10, {\it J. Algebraic Geom.} {\bf21} (2012), 21--59.
  
  \bibitem[DIM2]{dims} Debarre, O., Iliev, A., Manivel, L., Special prime Fano fourfolds of degree 10 and index 2, 
  {\it  Recent Advances in Algebraic Geometry,} 123--155,  C. Hacon, M. Musta\c t\u a, and M. Popa editors, London Math. Soc. Lecture Note Ser. {\bf417}, Cambridge University Press, 2014. 

\bibitem[DK]{DK}
Debarre, O., Kuznetsov A.,
Gushel--Mukai varieties: linear spaces and periods, eprint {\tt arXiv:1605.05648}.
  
  \bibitem[D]{dol} Dolgachev, I., {\it Classical Algebraic Geometry: a modern view,} Cambridge University Press, 2012.

        \bibitem[EPW]{epw}  Eisenbud, D., Popescu, S., Walter, C., Lagrangian subbundles and codimension
3 subcanonical subschemes, {\it Duke Math. J.} {\bf107} (2001), 427--467.

   \bibitem[GL]{gl} Green, M., Lazarsfeld, R., Special divisors on curves on a K3 surface, {\it Invent. Math.}   {\bf 89} (1987), 357--370.
   
    \bibitem[GLT]{glt}    Greer, F.,   Li, Z., Tian, Z., Picard groups on moduli of K3 surfaces with Mukai models, {\em Int. Math. Res. Not.} {\bf16} (2015), 7238--7257. 
 
   
   \bibitem[SGA2]{sga2}   Grothendieck, A.,  {\it Cohomologie 
locale des faisceaux coh\'erents et th\'eor\`emes de Lefschetz locaux et globaux,} S\'eminaire de G\'eom\'etrie Alg\'ebrique du Bois Marie, 1962. Augment\'e d'un expos\'e de Mich\`ele Raynaud. With a preface and edited by Yves Laszlo. Revised reprint of the 1968 French original. Documents Math\'ematiques  {\bf4}, Soci\'et\'e Math\'ematique de France, Paris, 2005.

 
  \bibitem[G]{gushel1983fano} Gushel', N.P.,
Fano varieties of genus $6$  (in  Russian), {\it 
Izv. Ross. Akad. Nauk Ser. Mat.} {\bf46} (1982),  1159--1174, 1343.  English transl.: {\it    Izv. Math.} {\bf21} (1983),  445--459.

  \bibitem[HT]{hats}   Hassett, B.,   Tschinkel, Yu.,
Hodge theory and Lagrangian planes on generalized Kummer fourfolds, {\it Mosc. Math. J.}  {\bf 13} (2013), 33--56.

  

  \bibitem[IKKR]{ikkr} Iliev, A.,    Kapustka, G., Kapustka, M., Ranestad, K.,   EPW cubes,  {\tt   arXiv:1505.02389.}  
  
 \bibitem[IM]{iliev-manivel} Iliev, A., Manivel, L., Fano manifolds of degree 10 and EPW sextics,
{\it  Ann. Sci. \'Ecole Norm. Sup.}  {\bf 44} (2011), 393--426.

    \bibitem[I]{iskovskih1977fano}	Iskovskikh,   V.A.,    Fano 3-folds. I,    (in  Russian), {\it 
Izv. Akad. Nauk SSSR Ser. Mat.} {\bf41} (1977), 516--562, 717.
    English transl.: {\it Math. USSR Izv.} {\bf11} (1977),  485--527.
    
    
    \bibitem[IP]{ip}	Iskovskikh,  V.A., Prokhorov, Yu.,  Fano varieties,   {\it Algebraic geometry, V}, 
    Encyclopaedia Math. Sci. {\bf47}, Springer-Verlag, Berlin, 1999.
    
\bibitem[JK]{johnsen2004k3} 	 Johnsen, T., Knutsen,  A.L., {\it K3 Projective Models in Scrolls,}  Lecture Notes in Math.  {\bf 1842}, Springer-Verlag, Berlin, 2004. 
 
\bibitem[K]{K}
Kuznetsov, A.,
K\"uchle fivefolds of type c5, 
{\it Math. Z.} {\bf284} (2016),  1245--1278.

\bibitem[KP]{KP}
Kuznetsov, A., Perry, A.,
Derived categories of Gushel--Mukai varieties,
eprint {\tt arXiv:1605.06568}.
 
     \bibitem[M]{mel}   Mella, M., 
Existence of good divisors on Mukai varieties, {\it
J. Algebraic Geom.}  {\bf8} (1999),  197--206.

  \bibitem[MM]{momu}  Mori, S.,  Mukai, S.,  On Fano 3-folds with $B_2\ge2$, {\it Algebraic varieties and analytic varieties (Tokyo, 1981),} 101--129, Adv. Stud. Pure Math.  {\bf1}, North-Holland, Amsterdam, 1983.
  
 \bibitem[Mu1]{mukai1989biregular}    Mukai, S.,  Biregular Classification of Fano 3-Folds and Fano Manifolds of Coindex 3, {\it Proc. Natl. Acad. Sci. USA} {\bf 86} (1989), 3000--3002.
   
 \bibitem[Mu2]{mukcg} \bysame, 
Curves and Grassmannians, {\em Algebraic geometry and related topics (Inchon,~1992),} 19--40,
Conf. Proc. Lecture Notes Algebraic Geom., I, International Press, Cambridge, MA, 1993. 

        \bibitem[Mu3]{mukai1995new}    \bysame, New development of theory of Fano 3-folds: vector bundle
method and moduli problem, (in Japanese) {\em S$\bar{\mbox{u}}$gaku} {\bf47} (1995), 125--144.   English transl.: {\it Sugaku  Expositions} {\bf11} (2002),  125--150.

        \bibitem[Mum]{mum}  Mumford, D., Varieties defined by quadratic equations,   {\em    Questions on Algebraic Varieties (C.I.M.E., III Ciclo, Varenna, 1969),}  29--100, Edizioni Cremonese, Rome, 1970.


  \bibitem[O1]{og1}	O'Grady,  K.,  Irreducible symplectic 4-folds and Eisenbud-Popescu-Walter sextics, {\it Duke Math. J.}  {\bf 134} (2006),  99--137. Extended version at {\tt   arXiv:math/0507398.}
  
  \bibitem[O2]{og2}	\bysame,      Dual double EPW-sextics and their periods, {\it  Pure Appl. Math. Q.}  {\bf  4}  (2008),   427--468.
  
      \bibitem[O3]{og3}	\bysame,   EPW-sextics: taxonomy, {\it Manuscripta Math.}  {\bf138} (2012), 221--272.
  
   \bibitem[O4]{og4}	\bysame,    Double covers of EPW-sextics, {\it Michigan Math. J.} {\bf 62} (2013), 143--184.  

  
       
\bibitem[O5]{og5}	\bysame, 
  Moduli of double EPW-sextics,      {\em Mem. Amer. Math. Soc.} {\bf240} (2016), no. 1136.
 
\bibitem[O6]{og6}	\bysame,    Periods of double EPW-sextics,  {\it Math. Z.} {\bf 280} (2015),  485--524. 

\bibitem[O7]{og7}	\bysame, Irreducible symplectic 4-folds numerically equivalent to $(K3)^{[2]}$, 
{\it Commun. Contemp. Math.} {\bf10} (2008),  553--608.

 \bibitem[PV]{pv}   Piontkowski, J.,    Van de Ven, A., The automorphism group of linear sections of the Grassmannian $G(1,N)$, {\it Doc. Math.} {\bf4} (1999), 623--664.
 
   \bibitem[P]{pro}   Prokhorov, Yu., Rationality constructions of some Fano fourfolds of index 2, {\it Moscow Univ. Math. Bull.} {\bf 48} (1993), 32--35.
  

\bibitem[PCS]{przhyjalkowski2005hyperelliptic} Przhiyalkovski\u\i, V.V., Chelʹtsov, I.A., Shramov, K.A.,
Hyperelliptic and trigonal Fano threefolds (in  Russian), {\it Izv. Ross. Akad. Nauk Ser. Mat.}  {\bf 69} (2005),  145--204.  English transl.: {\it 
Izv. Math.}  {\bf 69} (2005),  365--421.

  \bibitem[RT]{rs}  Ravindra, G.V., Srinivas, V.,   The Grothendieck-Lefschetz theorem for normal projective varieties, {\it
J. Algebraic Geom.}  {\bf 15} (2006),  563--590. 

  \bibitem[R]{rot} Roth, L., Algebraic varieties with canonical curve sections,  {\it Ann. Mat. Pura Appl.}  {\bf 29} (1949), 91--97. 
 

    \bibitem[S-B1]{sb}  Shepherd-Barron, N.I.,
Invariant theory for $S_5$ and the rationality of $M_6$,
{\it Compos. Math.} {\bf70} (1989),   13--25. 

    \bibitem[S-B2]{sb92}  Shepherd-Barron, N.I.,
The rationality of quintic del Pezzo surfaces---a short proof,
{\it Bull. London Math. Soc.} {\bf24} (1992), 249--250.

\bibitem[W]{W}
Weyman, J., 
Cohomology of vector bundles and syzygies. 
{\it Cambridge Tracts in Mathematics}, 149. Cambridge University Press, Cambridge, 2003. xiv+371 pp. 

 \end{thebibliography}
\end{document}